\documentclass[reqno,letterpaper,11pt]{article}

\usepackage{hyperref}

\usepackage{palatino}

\usepackage{amsmath}
\usepackage{amsfonts,amssymb,amsmath,latexsym,wasysym,mathrsfs,bbm,stmaryrd}
\usepackage[all]{xy}
\usepackage[usenames]{color}
\usepackage{epsfig}

\usepackage{amsthm}
\usepackage{thmtools}
\usepackage{thm-restate}
\declaretheorem[numberwithin=section]{theorem}
\declaretheorem[sibling=theorem]{proposition}
\declaretheorem[sibling=theorem]{definition}
\declaretheorem[sibling=theorem]{corollary}
\declaretheorem[sibling=theorem]{lemma}

\declaretheorem[sibling=theorem]{notation}

\declaretheorem[sibling=theorem,style=remark]{remark}
\declaretheorem[sibling=theorem,style=remark]{example}

\numberwithin{equation}{section}

\usepackage[mathcal]{euscript}
\usepackage{times}

\def\R{\mathbb R}
\def\C{\mathbb C}
\def\Z{\mathbb Z}
\def\N{\mathbb N}
\def\E{\mathbb E}

\def\del{\partial}
\def\1{\mathbbm{1}}
\def\tensor{\otimes}

\def\U{\mathbb U}
\def\M{\mathbb M}
\def\GL{\mathbb{GL}}
\def\u{\mathfrak{u}}
\def\gl{\mathfrak{gl}}
\def\G{\mathscr{G}}
\def\H{\mathscr{H}}
\def\HH{\mathcal{H}}
\def\D{\mathcal{D}}
\def\L{\mathcal{L}}
\def\EX{\mathscr{E}}

\newcommand{\tr}{\mathrm{tr}}
\newcommand{\Ur}{\mathrm{Tr}\,}
\newcommand{\mx}[1]{\mathbf{#1}}

\long\def\symbolfootnote[#1]#2{\begingroup%
\def\thefootnote{\fnsymbol{footnote}}\footnote[#1]{#2}\endgroup}

\begin{document}

\title{The Large-$N$ Limit of the Segal--Bargmann Transform on $\U_N$}
\author{Bruce K. Driver\thanks{Supported in part by NSF Grant DMS-1106270} \\
Department of Mathematics \\
University of California, San Diego \\
La Jolla, CA 92093-0112 \\
\texttt{bdriver@math.ucsd.edu}
\and
Brian C. Hall\thanks{Supported in part by NSF Grant DMS-1001328} \\
Department of Mathematics \\
University of Notre Dame \\
Notre Dame, IN 46556 \\
\texttt{bhall@nd.edu}
\and
Todd Kemp\thanks{Supported in part by NSF Grant DMS-1001894} \\
Department of Mathematics \\
University of California, San Diego \\
La Jolla, CA 92093-0112 \\
\texttt{tkemp@math.ucsd.edu}
}

\date{\today} 

\maketitle

\begin{abstract}
We study the (two-parameter) Segal--Bargmann transform $\mathbf{B}_{s,t}^N$ on the unitary group $\U_N$, for large $N$.  Acting on matrix valued functions that are equivariant under the adjoint action of the group, the transform has a meaningful limit $\G_{s,t}$ as $N\to\infty$, which can be identified as an operator on the space of complex Laurent polynomials.  We introduce the space of {\em trace polynomials}, and use it to give effective computational methods to determine the action of the heat operator, and thus the Segal--Bargmann transform.  We prove several concentration of measure and limit theorems, giving a direct connection from the finite-dimensional transform $\mathbf{B}_{s,t}^N$ to its limit $\G_{s,t}$.  We characterize the operator $\G_{s,t}$ through its inverse action on the standard polynomial basis.  Finally, we show that, in the case $s=t$, the limit transform $\G_{t,t}$ is the ``free Hall transform'' $\mathscr{G}^t$ introduced by Biane.
\end{abstract}

\tableofcontents

\section{Introduction\label{section Introduction}}

The {\em Segal--Bargmann transform} (also known in the physics literature as the {\em Bargmann transform} or {\em Coherent State transform}) is a unitary isomorphism from $L^2$ to holomorphic $L^2$.  It was originally introduced by Segal \cite{Segal1962,Segal1963,Segal1978} and Bargmann \cite{Bargmann1961,Bargmann1962}, as a map
\[ S_t\colon L^2(\R^N,\gamma^{N}_t)\to \mathcal{H}L^2(\C^N,\gamma^{2N}_{t/2}) \]
where $\gamma^N_t$ is the standard Gaussian heat kernel measure $(\frac{1}{4\pi t})^{N/2}\exp(-\frac{1}{4t}|\mx{x}|^2)\,d\mx{x}$ on $\R^N$, and $\mathcal{H}L^2$ denotes the subspace of square-integrable {\em holomorphic} functions.  The transform $S_t$ is given by convolution with the heat kernel, followed by analytic continuation.

In \cite{Hall1994}, the second author introduced an analog of the
Segal--Bargmann transform for any compact Lie group $K$.  Let $\Delta_{K}$ denote the Laplace operator over
$K$ (determined, up to scale, by the $\mathrm{Ad}$-invariant inner product on the Lie algebra $\mathfrak{k}$ of $K$), and denote by $e^{\frac{t}{2}\Delta_{K}}$ the corresponding heat operator. The generalized Segal--Bargmann transform $B_t$ maps functions on $K$ to holomorphic functions on the complexification
$K_{\mathbb{C}}$ of $K$, by application of the heat operator and analytic continuation. \label{SBT}

In this paper, we will work with the classical unitary groups $K=\U_N$, and identify a limit as $N\to\infty$ of the Segal--Bargmann transform on $\U_N$.

\subsection{Main Definitions and Theorems\label{section Main Theorems}}

Denote by $\M_N$ the algebra of $N\times N$ complex matrices, with unit $I_N$.  Let $\U_N$ denote the group of unitary matrices $\U_N=\{U\in\M_N\colon UU^\ast = I_N\}$, and let $\GL_N$ denote the group of all invertible matrices in $\M_N$; $\GL_N$ is the complexification of $\U_N$.  The Lie algebra of $\U_N$ is $\mathfrak{u}_N = \{X\in\M_N\colon X^\ast=-X\}$, while the Lie algebra of $\GL_N$ is $\gl_N = \M_N$.  To describe the Laplace operator $\Delta_{\U_N}$ explicitly, we fix the following notation.

\begin{notation} \label{n.new.1} For $Z\in\M_N$ let
\[ \mathrm{Tr}_N(Z) \equiv \sum_{n=1}^N Z_{nn} \qquad \text{and} \qquad \tr_N(Z)\equiv \frac1N\mathrm{Tr}_N(Z) = \frac1N\sum_{n=1}^N Z_{nn} \]
denote the trace and normalized trace of $Z$, respectively.  [We will usually drop the subscripts and write simply $\Ur$ and $\tr$, as the dimension will always be clear from context.]  We also define (scaled) Hilbert-Schmidt norms on $\u_N$ and on $\M_N$ by
\begin{align}
\label{scalingMetric} \|X\|^2_{\u_N}&\equiv N^2\tr(XX^\ast) = N\Ur(XX^\ast)=N\sum_{j,k=1}^{N}|X_{jk}|^2, \quad X\in\u_N, \quad \text{and} \\
\label{eq M_N metric} \|Z\|_{\M_N}^2 &\equiv \tr(ZZ^\ast) = \frac{1}{N}\Ur(ZZ^\ast) = \frac1N\sum_{j,k=1}^N |Z_{jk}|^2, \quad Z\in\M_N.
\end{align}
\end{notation}

\begin{definition} \label{definition left-invariant derivative} For $\xi\in\M_N$, let $\del_\xi$ denote the left-invariant vector field on $\GL_N$, whose action on smooth functions $f\colon\GL_N\to\C$ is given by
\begin{equation} \label{eq del_X} (\del_\xi f)(Z) = \left.\frac{d}{dt}\right|_{t=0} f(Ze^{t\xi}), \qquad Z\in \GL_N. \end{equation}
If $\xi=X\in\u_N$ then $\del_X$ is tangential to $\U_N$ and so restricts to a left-invariant vector field on $\U_N$ whose action on smooth functions $f\colon\U_N\to\C$ is still given by (\ref{eq del_X}).
\end{definition}

\begin{definition} \label{definition Ast and must} The {\bf Laplace operator} $\Delta_{\U_N}$ is the second order elliptic operator on $\U_N$ whose action on smooth functions $f\colon\U_N\to\C$ is given by
\begin{equation} \label{eq def Delta} \Delta_{\U_N}f = \sum_{X\in\beta_N} \del_X^2f \end{equation}
where $\beta_N$ is an orthonormal basis for $\u_N$ (with norm $\|\cdot\|_{\u_N}$ given in (\ref{scalingMetric})); the operator does not depend on which orthonormal basis is chosen.

Similarly, for $s,t>0$ with $s>t/2$, let $A^N_{s,t}$ be the second order elliptic operator on $\GL_N$ whose action on smooth functions $f\colon\GL_N\to\C$ is given by
\begin{equation} \label{eq Ast def} A^N_{s,t}f = \left(s-\frac{t}{2}\right)\sum_{X\in\beta_N} \del_X^2f + \frac{t}{2}\sum_{X\in\beta_N} \del_{iX}^2f. \end{equation}
\end{definition}
Let $C_{c}^{\infty}(\GL_N)$ denote the smooth compactly supported functions from $\GL_N$ to $\C$.
It is well known that the operators $\Delta_{\U_N}$ and $A_{s,t}^{N}|_{C_{c}^{\infty}(\GL_N)}
$ are non-positive and essentially self-adjoint on $L^2(\U_N)$ and $L^2(\GL_N)$ respectively,
where the measures on $\U_N$ and $\GL_N$ are taken to be any right invariant Haar measures. The
self adjoint closures of these operators induce (heat) semigroups $\left\{e^{\frac{\tau}{2}\Delta_{\U_N}}:\tau\geq0\right\}$ and $\left\{e^{\frac{\tau}{2}A_{s,t}}:\tau\geq0\right\}  $ on $L^{2}(\U_N)$ and $L^2(\GL_N)$ respectively. These semigroups then induce two (heat kernel) measures, $\rho_{t}^{N}$ and $\mu_{s,t}^{N},$ which satisfy
\begin{align} \label{eq rho t def} \int_{\U_N} f(U)\rho_t^N(dU) &= \left(e^{\frac{t}{2}\Delta_{\U_N}}f\right)(I_N), \qquad f\in C(\U_N), \\
\label{eq must def} \int_{\GL_N} f(Z)\,\mu^N_{s,t}(dZ) &= \left(e^{\frac12A_{s,t}}f\right)(I_N), \qquad f\in C_c(\GL_N).
\end{align}
We will sometimes write $\E_{\rho_t^N}(f)=\int_{\U_N} f(U)\,\rho_t^N(dU)$ and $\E_{\mu_{s,t}^N}(f)=\int_{\GL_N} f(Z)\,\mu_{s,t}(dZ)$.

\begin{remark} \label{r.Robinson} The test functions $f$ on $\GL_N$ we will use tend not to be compactly-supported (or bounded), but they do have sufficiently slow growth that (\ref{eq must def}) still holds true for such functions.  This follows from Langland's Theorem; cf.\ \cite[Theorem 2.1 (p.\ 152)]{Robinson1991}. \ref{section heat kernels} gives a concise sketch of the heat kernel results we need in this paper.  \end{remark}

Let $\HH L^2(\GL_N,\mu^N_{s,t})$ denote the Hilbert subspace of $L^2(\GL_N,\mu^N_{s,t})$ consisting of those $L^2$ functions which possess a holomorphic representative.  The following theorem with $s=t$ is a special case of a the Lie group analogue
of the Segal--Bargmann transform $B_t$ from page \pageref{SBT}.  The two parameter form of this transform which we use here was introduced by the first and second authors in \cite{Driver1999}; see also \cite{Driver1995,Hall1994,Hall2001b,Hall1998}.

\begin{theorem}[D, H, \cite{Driver1999}]
\label{t.new1.4}Fix $s,t>0$ with $s>t/2$. For each $f\in L^{2}(\U_N,\rho_{s}^{N})$,
the function $e^{\frac{t}{2}\Delta_{\U_N}}f$ has a representative which has a
unique analytic continuation to $\GL_N$; denote this analytic continuation by $B_{s,t}^{N}f$.
Then $B^N_{s,t}f \in\mathcal{H}L^{2}(\GL_N),\mu_{s,t}^{N})$, and the resulting transform
\[ B_{s,t}^{N}\colon L^{2}(\U_N,\rho_{s}^{N})\to\HH L^{2}(\GL_N,\mu_{s,t}^{N})
\]
is a unitary isomorphism.
\end{theorem}

In this paper, we are interested in a slight extension of $B^N_{s,t}$ to matrix-valued functions.

\begin{definition}[Boosted Segal-Bargmann Transform]\label{d.new1.5} Given a $\M_N$-valued function $F$ on either $\U_N$ or $\GL_N$, denote by $\|F\|_{\M_N}$ the scalar-valued function $Z\mapsto \|F(Z)\|_{\M_N}$. Fix $s,t>0$ with $s>t/2$, and let
\begin{align*} L^2(\U_N,\rho^N_s;\M_N) &= \left\{F\colon\U_N\to\M_N\,;\,\|F\|_{\M_N}\in L^2(\U_N,\rho^N_s)\right\}, \quad \text{and} \\
L^2(\GL_N,\mu^N_{s,t};\M_N) &= \left\{F\colon\GL_N\to\M_N\,;\,\|F\|_{\M_N}\in L^2(\GL_N,\mu^N_{s,t})\right\}.
\end{align*}
Let $\HH L^2(\GL_N,\mu^N_{s,t};\M_N)\subset L^2(\GL_N,\mu^N_{s,t};\M_N)$ denote the subspace of (matrix-valued) holomorphic functions. These are Hilbert spaces in the norms
\begin{align} \label{e.L2norm1} \|F\|_{L^2(\U_N),\rho^N_s;\M_N)}^2 &\equiv \int_{\U_N} \|F(U)\|_{\M_N}^2\,\rho^N_s(dU) \\
\label{e.L2norm2} \|H\|_{L^2(\GL_N,\mu^N_{s,t};\M_N)}^2 &\equiv \int_{\GL_N} \|H(Z)\|_{\M_N}^2\,\mu^N_{s,t}(dZ).
\end{align}
The {\bf boosted Segal--Bargmann transform}
\[ \mathbf{B}_{s,t}^{N}:L^{2}(\U_N,\rho_{s}^{N};\M_{N})\rightarrow\HH L^{2}(\GL_N,\mu_{s,t}^{N};\M_{N}) \]
is the unitary isomorphism determined by applying $B^N_{s,t}$ componentwise; that is, it is determined by
\[ \mathbf{B}_{s,t}^{N}\left(f\cdot V\right)  =B_{s,t}^{N}f\cdot V \quad \text{ for } \quad f\in L^{2}(\U_N,\rho_{s}^{N}) \; \text{ and } \; V\in \M_{N}. \]
\end{definition}
\noindent The space $L^2(\U_N,\rho^N_s;\M_N)$ can be naturally identified with the Hilbert space tensor product $L^2(\U_N,\rho^N_s)\tensor_\C \M_N$; under this identification, $\mx{B}^N_{s,t}\cong B^N_{s,t}\tensor\mathrm{id}_{\M_N}$.  To understand its action, consider the matrix-valued function $F(U) = U^2$ on $\U_N$.  Then, as calculated in Example \ref{example heat U^2},
\begin{equation}
(\mx{B}^N_{s,t}F)(Z) =e^{-t}\cosh(t/N)Z^2 -Ne^{-t}\sinh(t/N)Z\cdot\tr(Z).
\label{BtU2} \end{equation} 
\noindent This highlights the fact that the Segal--Bargmann transform does not preserve the space of polynomial functions of a $\U_N$-variable; in general, it maps such functions to {\bf trace polynomials}.

\begin{definition} \label{d.Laurent.trace.new} Let $\C[u,u^{-1}]$ denote the algebra of {\bf Laurent polynomials} in a single variable $u$:
\begin{equation} \label{e.poly.fc} \C[u,u^{-1}] = \left\{\sum_{k\in\Z} a_ku^k \colon a_k\in\C, a_k=0\text{ for all but finitely-many }k\right\}, \end{equation}
with the usual polynomial multiplication.  The subalgebras $\C[u]$ and $\C[u^{-1}]$ denote polynomials in $u$ and $u^{-1}$ respectively.

We define the {\bf Laurent polynomial functional calculus} as follows: for $f\in\C[u,u^{-1}]$ as in (\ref{e.poly.fc}), the function $f_N\colon \GL_N\to\M_N$ is given by
\begin{equation} \label{e.L.fc} f_N(Z) = \sum_{k\in\Z} a_k Z^k, \end{equation}
where the $k=0$ term is interpreted as $a_0I_N$.

Let $\C[\mx{v}]$ denote the algebra of complex polynomials in infinitely-many commuting variables $\mx{v} = \{v_{\pm 1},v_{\pm 2},\ldots\}$, and let $\C[u,u^{-1};\mx{v}]$ denote the algebra of polynomials in the variables $u,u^{-1},v_{\pm 1},v_{\pm 2},\ldots$ (although we do note treat $u$ and $u^{-1}$ as independent in general).  Thus
\begin{equation} \label{e.poly.fc2} \C[u,u^{-1};\mx{v}] = \left\{\sum_{k\in\Z} u^k Q_k(\mx{v}) \colon Q_k(\mx{v})\in\C[\mx{v}], Q_k=0\text{ for all but finitely-many }k\right\}. \end{equation}
In other words, we can realize $\C[u,u^{-1};\mx{v}]$ as the algebra $\left(\C[\mx{v}]\right)[u,u^{-1}]$ of Laurent polynomials in $u$ with coefficients in the ring $\C[\mx{v}]$; equivalently, $\C[u,u^{-1};\mx{v}] \cong \C[u,u^{-1}]\tensor_{\C} \C[\mx{v}]$.  We denote elements of $\C[u,u^{-1};\mx{v}]$ by $P = P(u;\mx{v})$.

Define the {\bf trace polynomial functional calculus} as follows: for $P\in\C[u,u^{-1};\mx{v}]$, the function $P_N\colon \GL_N\to\M_N$ is given by
\[ P_N(Z) \equiv \left.P(u;\mx{v})\right|_{u=Z,v_k = \tr(Z^k), k\ne 0}. \]
Functions of the form $P_N$ for $P\in\C[u,u^{-1};\mx{v}]$ are called {\bf trace  polynomials}.  \end{definition}
\noindent It might be more accurate to call such functions trace {\em Laurent} polynomials, but we will simply use {\em trace polynomials} as it should cause no confusion.  For a concrete example: if $P(u;\mx{v}) = v_2v_{-4}^2u^5 + 8v_1^6v_{-3}$ then
\[ P_N(Z) = \tr(Z^2)\tr(Z^{-4})^2Z^5 + 8\tr(Z)^6\tr(Z^{-3})I_N. \]

\begin{remark} It is important to note that, for any finite $N$, there will be many {\em distinct} elements $P\in\C[u,u^{-1};\mx{v}]$ that induce the same trace polynomial, i.e.\ there will be $P\ne Q$ with $P_N = Q_N$.  Nevertheless, it is true that if $P_N=Q_N$ {\em for all sufficiently large} $N$, then $P=Q$; this is the statement of Theorem \ref{prop P-->P_N unique} below. \end{remark}

%
%
%

\begin{theorem} \label{t.new1.8} Let $P\in\C[u,u^{-1};\mx{v}]$ as in Definition \ref{d.Laurent.trace.new}, and let $N\in\N$ and $s,t>0$ with $s>t/2$.  There exists an element $P^N_t\in\C[u,u^{-1};\mx{v}]$ such that
\begin{equation} \label{e.new1.3} \mx{B}^N_{s,t}P_N = [P^N_t]_N. \end{equation}
The polynomial $P^N_t$ can be computed as $P^N_t = e^{\frac{t}{2}\D_N}P$ where $\D_N$ is a certain pseudodifferential operator on $\C[u,u^{-1};\mx{v}]$; cf.\ Theorem \ref{t.intertwine.new} and Definition \ref{Def LN L} below. \end{theorem}
\noindent The proof of Theorem \ref{t.new1.8} is on page \pageref{proof of t.new1.8}.

\begin{remark} \label{r.pseudo1} We call $\D_N$ a pseudodifferential operator because, if we identify $u$ as variable in the unit circle $\U$, then $\mathcal{D}_{N}$ acts as a first order differential operator composed with a linear combination of the identity operator and the Hilbert transform on the circle.  As explained in \cite{Melo1997}, the Hilbert transform on $\U$ is a pseudodifferential operator.  See Definition \ref{def A pm} and Remark \ref{r.pseudo2} for more details.  \end{remark}

For each $N>1$, the boosted Segal--Bargmann transform's range on Laurent polynomial calculus functions is contained in the larger space of trace polynomials.  But as $N\to\infty$, its image concentrates back on Laurent polynomials.  This is our main theorem.

\begin{theorem} \label{Main Theorem} Let $s,t>0$ with $s>\frac{t}{2}$.  For each $f\in\C[u,u^{-1}]$, there exist unique $g_{s,t},h_{s,t}\in\C[u,u^{-1}]$ such that
\begin{align} \label{eq SB limit 1} &\| \mx{B}^N_{s,t} f_N - [g_{s,t}]_N \|^2_{L^2(\GL_N,\mu_{s,t}^N;\M_N)} = O\left(\frac{1}{N^2}\right), \qquad \text{and} \\ \label{eq SB limit 2}
&\| (\mx{B}^N_{s,t})^{-1}f_N - [h_{s,t}]_N \|^2_{L^2(\U_N,\rho_s^N;\M_N)} = O\left(\frac{1}{N^2}\right). \end{align}
We denote that map $\G_{s,t}\colon\C[u,u^{-1}]\to\C[u,u^{-1}]$ given by $f\mapsto g_{s,t}$ as the {\bf free unitary Segal--Bargmann transform}, and we denote the map $\H_{s,t}\colon\C[u,u^{-1}]\to\C[u,u^{-1}]$ given by $f\mapsto h_{s,t}$ as the {\bf free unitary inverse Segal--Bargmann transform}.
\end{theorem}

\begin{remark} A concurrent paper by G.\ C\'ebron has recently proven a similar theorem; in particular, the $s=t$ case of (\ref{eq SB limit 1}) is equivalent to \cite[Theorem 4.7]{Guillaume2013}.  C\'ebron's framework is somewhat different from ours, and should be consulted for a complementary approach.  See Remark \ref{remark Guillaume} for a detailed comparison.  \end{remark}

Theorem \ref{Main Theorem} is proved on page \pageref{proof of Main Theorem}.  The ``inverse'' terminology is justified by the following, whose proof is on page \pageref{proof of thm inverse is inverse}.
\begin{theorem} \label{thm inverse is inverse} For $s,t>0$ with $s>\frac{t}{2}$, the maps $\G_{s,t}$ and $\H_{s,t}$ are invertible linear operators on $\C[u,u^{-1}]$, and $\G_{s,t}^{-1} = \H_{s,t}$.
\end{theorem}

To explain how the concentration phenomenon of Theorem \ref{Main Theorem} occurs, we recall the following theorem of Biane.
\begin{theorem}
[{Biane, \cite[Lemma 11]{Biane1997b}}]\label{t.new1.9} For each $s>0$ and
$k\in\mathbb{Z}$,
\[ \lim_{N\rightarrow\infty}\int_{\U_N}\tr(U^k)\,\rho_s^N(dU) =\nu_k(s), \]
where $\nu_{0}(s)=1$ and, for $k\neq0$,
\begin{equation}
\nu_{k}(s) =e^{-\frac{|k|}{2}s}\sum_{j=0}^{|k| -1}\frac{(-s)^j}{j!}|k|^{j-1}\binom{|k|}{j+1}. \label{e.new1.4}%
\end{equation}
\end{theorem}
From (\ref{e.new1.4}) it is clear that $\nu_{k}=\nu_{-k}$ for all $k\in\mathbb{N}$ and that each $\nu_{k}(\cdot)$ has an analytic
continuation to a holomorphic function on $\mathbb{C}$ which we still denote by $\nu_{k}$.  For each $s\in\R$, these constants are the moments of a probability measure $\nu_s$, supported on either the unit circle $\U$ (for $s\ge0$) or the positive real half-line $(0,\infty)$ (for $s\le0$).  For $s>0$, $\nu_k(s)$ are the moments of the {\em free unitary Brownian motion} $u_s$; see \cite[Prop.\ 10]{Biane1997b}.  These functions also encode the large-$N$ limits of the moments of the measures $\mu_{s,t}^{N}$.
\begin{theorem}
\label{t.new1.10} Let $s,t>0$ with $s>\frac{t}{2}$, and let $k\in\mathbb{Z}$; then
\[ \lim_{N\rightarrow\infty}\int_{\GL_N}\tr(Z^k)\,\mu_{s,t}^N(dZ) =\nu_k(s-t). \]
\end{theorem}
\noindent The proof of Theorem \ref{t.new1.10} is on page \pageref{proof of t.new1.10}.  See, also, the third author's concurrent paper \cite{Kemp2013} for several new convergence results for the empirical eigenvalues and singular values of random matrices sampled from $\rho^N_s$ and $\mu^N_{s,t}$.

Consider, again, the calculation of (\ref{BtU2}), which shows that, if $f(u)=u^2$, then the polynomial $f^N_t\in\C[u,u^{-1};\mx{v}]$ of Theorem \ref{t.new1.8} can be identified as
\[ f^N_t(u;\mx{v}) = e^{-t}\cosh(t/N)u^2 - Ne^{-t}\sinh(t/N)uv_1 = e^{-t}(u^2-tuv_1)+O\left(\frac{1}{N^2}\right). \]
The trace polynomial functional calculus evaluates $f_t^N(u;\mx{v})$ at $Z\in\GL_N$ by setting $u=Z$ and $v_1=\tr(Z)$; but as $N\to\infty$, $\tr(Z)\to \nu_1(s-t) = e^{-(s-t)/2}$ by Theorem \ref{t.new1.10}.  This illustrates the fact that, in this case,
\[ (\G_{s,t}f)(u) = e^{-t}(u^2-te^{-(s-t)/2}u). \]
In general, this is how $g_{s,t}$ in (\ref{eq SB limit 1}) is produced: by evaluating the traces in the trace polynomial $P^N_t$ in Theorem \ref{t.new1.8} at the moments $\nu_k(s-t)$ of Theorem \ref{t.new1.10}, and taking the large-$N$ limit of the resulting Laurent polynomial.  To fully justify this, we prove the following concentration theorem, which shows, in a strong way, that the trace random variables $Z\mapsto \tr(Z^k)$ over $\U_N$ and $\GL_N$ concentrate on their means as $N\to\infty$.

\begin{theorem} \label{t.concentration} For $s\in\R$, define the {\bf trace evaluation map} $\pi_s\colon\C[u,u^{-1};\mx{v}]\to\C[u,u^{-1}]$ by
\begin{equation} \label{e.pi_s.new} \left(\pi_sP\right)(u) = \left.P(u;\mx{v})\right|_{v_k = \nu_k(s), k\ne 0}.  \end{equation}
Let $s,t>0$, with $s>t/2$.  For any $P\in\C[u,u^{-1};\mx{v}]$,
\begin{align} \label{e.conc1}
 &\|P_N- [\pi_{s}P]_N\|^2_{L^2(\U_N,\rho_s^N;\M_N)} = O\left(\frac{1}{N^2}\right), \quad \text{and} \\
\label{e.conc2} &\|P_N-[\pi_{s-t}P]_N\|^2_{L^2(\GL_N,\mu_{s,t}^N;\M_N)} = O\left(\frac{1}{N^2}\right). \end{align}
\end{theorem}
\noindent  The proof of Theorem \ref{t.concentration} can be found on page \pageref{proof of t.concentration}. Combining it with Theorem \ref{t.new1.8}, we see that the limit Segal--Bargmann transform $\G_{s,t}f$ in (\ref{eq SB limit 1}) is given by $\G_{s,t}f  = \lim_{N\to\infty}\pi_{s-t}(f^N_t)$; see (\ref{e.SBdef.new}) below.



Finally, we explicitly describe the action of $\H_{s,t}$ via a generating function.
\begin{theorem} \label{thm Biane transform} Let $s,t>0$ with $s>t/2$. For $k\ge 1$, let $f_k(u) \equiv u^k$ and $p^{s,t}_k \equiv \H_{s,t}(f_k)$.  Then the generating function for $\{p^{s,t}_k\}$ is given by the power series
\[ \Pi(s,t,u,z) = \sum_{k\ge 1} p^{s,t}_k(u)z^k, \]
which converges for all sufficiently small $u,z\in\C$.  This generating function is determined by the implicit formula
\begin{equation} \label{eq formula for gen fn} \Pi(s,t,u,ze^{\frac12(s-t)\frac{1+z}{1-z}}) = \left(1-uze^{\frac{s}{2}\frac{1+z}{1-z}}\right)^{-1}-1. \end{equation}
\end{theorem}
\noindent In the special case $s=t$, this yields the generating function corresponding to the transform $\mathscr{G}^t$ of \cite[Proposition 13]{Biane1997b}, which Biane called the {\em free Hall transform} (after the second author of this paper).  Thus, $\G_{t,t}=\mathscr{G}^t$, and the free unitary Segal--Bargmann transform is a generalization of the free Hall transform.  The proof of Theorem \ref{thm Biane transform} can be found on page \pageref{proof of thm Biane transform}.

\subsection{Intertwining Operators and Partial Product Rule\label{section Key Proof Ingredients}}

The key ingredient needed to prove all the main theorems of this paper is the following intertwining formula, which shows that the Laplace operator $\Delta_{\U_N}$ factors through a pseudodifferential operator on $\C[u,u^{-1};\mx{v}]$.

\begin{theorem}[Intertwining Formulas] \label{t.intertwine.new} Let $\C[u,u^{-1};\mx{v}]$ be the polynomial space of Definition \ref{d.Laurent.trace.new}, let $t\ge 0$, and let $N\in\N$. There exists a first order pseudodifferential operator $\D$ on $\C[u,u^{-1};\mx{v}]$ and a second order differential operator $\L$ on $\C[u,u^{-1};\mx{v}]$ (cf.\ (\ref{e.new1.7}) and (\ref{e.deq}) below) such that, setting
\begin{equation} \label{e.DN.new} \D_N = \D - \frac{1}{N^2}\L, \end{equation}
it follows that
\begin{equation} \label{eq intertwining formula 1} \Delta_{\U_N} P_N = [\D_N P]_N, \quad \text{for all }\;P\in\C[u,u^{-1};\mx{v}]. \end{equation} 
Moreover, the heat operator is given by
\begin{equation} \label{e.int2.new} e^{\frac{t}{2}\Delta_{\U_N}} P_N = [e^{\frac{t}{2}\D_N}P]_N, \quad \text{for all }\;P\in\C[u,u^{-1};\mx{v}]. \end{equation}
A similar intertwining formula holds for the operator $A_{s,t}$; cf.\ Theorem \ref{t.inter2} on page \pageref{t.inter2}.
\end{theorem}
\noindent The proof of Theorem \ref{t.intertwine.new} is on page \pageref{proof of t.intertwine.new}.


\begin{remark} \begin{itemize} \item[(1)] We will see in Section \ref{section Intertwining I} below that the space $\C[u,u^{-1};\mx{v}]$ is the union of a family $\{\C_n[u,u^{-1};\mx{v}]\}_{n\in\N}$ of finite-dimensional subspaces, each of which is invariant under $\D_N$.  Hence, the exponential $e^{\frac{t}{2}\D_N}$ makes sense as an operator on $\C[u,u^{-1};\mx{v}]$, for all $t\in\R$.
\item[(2)] Our intertwining formula (\ref{eq intertwining formula 1}) is closely related to results due to E.\ M.\ Rains \cite{Rains1997} and A.\ N.\ Sengupta \cite{Sengupta2008}.  In both cases, the Laplacian $\Delta_{\U_N}$ was identified by a decomposition similar to (\ref{e.DN.new}) for some operators like our $\D$ and $\L$.  We show that the component operators $\D$ and $\L$ can be realized as pseudodifferential operators on a polynomial intertwining space, which simplifies much of our analysis.
\end{itemize}
\end{remark}

Since $\D_N = \D+O(1/N^2)$, it follows that $e^{\frac{t}{2}\D_N} = e^{\frac{t}{2}\D} + O(1/N^2)$; this is made precise in Lemma \ref{l.findim} below.  As such, we will show in the proof of Theorem \ref{Main Theorem} that that the free unitary Segal--Bargmann transform and its inverse are given by
\begin{equation} \label{e.SBdef.new} \G_{s,t} = \pi_{s-t}\circ e^{\frac{t}{2}\D}, \qquad \text{and} \qquad \H_{s,t} = \pi_s\circ e^{-\frac{t}{2}\D}. \end{equation}
\noindent  See Section \ref{section Main Theorem} for details. The two operators $e^{\frac{t}{2}\mathcal{D}}$ and $e^{-\frac{t}{2}\mathcal{D}}$ are, of course, inverse to each other; Theorem \ref{thm inverse is inverse} shows that this holds true even with the composed evaluations maps.

The operator $\D$ is a first order {\em pseudo}differential operator, but it is not a differential operator: it does not satisfy the Leibnitz product rule.  It does, however, satisfy the following {\bf partial product rule} which is of both computational and conceptual importance.

\begin{theorem}[Partial Product Rule] \label{cor product rule and homom} Let $P\in\C[u,u^{-1};\mx{v}]$ and $Q\in\C[\mx{v}]$.  Then
\begin{equation} \label{eq partial product 1} \mathcal{D}(PQ) = (\mathcal{D}P)Q + P(\mathcal{D}Q). \end{equation}
Thus, for any $t\in\R$,
\begin{equation} \label{eq partial product 2} e^{\frac{t}{2}\mathcal{D}}(PQ) = e^{\frac{t}{2}\mathcal{D}}P\cdot e^{\frac{t}{2}\mathcal{D}}Q. \end{equation}
\end{theorem}
\noindent The proof of Theorem \ref{cor product rule and homom} can be found on page \pageref{proof of cor product rule and homom}.

\subsection{History and Discussion}

Since the classical Segal--Bargmann transform $S_t$ for Euclidean spaces admits an
infinite dimensional version \cite{Segal1978}, it is natural to attempt to
construct an infinite dimensional limit of the transform for compact Lie
groups. One successful approach to such a limit is found in the paper \cite{Hall1998}
of the second author and A. N. Sengupta, in which they develop a version of the 
Segal--Bargmann transform for the path group with values in a compact Lie
group $K$. The paper \cite{Hall1998} is an extension of the work of L. Gross and P.
Malliavin \cite{Gross1996} and reflects the origins of the generalized Segal--Bargmann
transform for compact Lie groups in the work of Gross \cite{Gross1993}.

A different approach to an infinite dimensional limit is to consider the
transform on a nested family of compact Lie groups, such as $\U_N$ for $N=1,2,3,\ldots$ The most obvious approach to the $N\rightarrow\infty$ limit would be to use on each $\u_N$ a fixed (i.e.\ $N$-independent)
multiple of the Hilbert--Schmidt norm $\|X\|_{\mathrm{HS}}^2 = \Ur(XX^\ast)$. Work of M. Gordina \cite{Gordina2000a,Gordina200b}, however,
showed that this approach does not work, because the target Hilbert space
becomes undefined in the limit. Indeed, Gordina showed that, with the metrics normalized  the this way,
in the large-$N$ limit all nonconstant holomorphic functions on $\GL_N$ have infinite norm
with respect to the heat kernel measure $\mu^N_{t,t}$.

In \cite{Biane1997b}, Biane proposed scaling the Hilbert-Schmidt norm with $N$ as in (\ref{scalingMetric}); he successfully carried out a large-$N$ limit of the {\em Lie algebra version} of the transform.  That is: taking the underlying space to
be the Lie algebra $\u_N$ rather than the group $\U_N$, he considered a version of the classical Euclidean Segal--Bargmann transform, $\mx{S}^N_t$ acting on functions from $\u_N$ to $\M_N$ given by polynomial functional calculus (cf.\ Definition \ref{d.Laurent.trace.new}).    If $f\in\C[u]$, the transformed functions $\mx{S}^N_t f_N$ have a limit (in a sense analogous to our Theorem \ref{Main Theorem}) which can be thought of as a polynomial $f_t\in\C[u]$.  This defines a unitary transformation $\mathscr{F}^t\colon f\mapsto f_t$ \cite[Theorem 3]{Biane1997b} on the limiting $L^2$ closure of polynomials with respect to the limit heat kernel measure---in this context Wigner's semicircle law.

\begin{remark} \label{remark functional calculus} The results of \cite[Section 1]{Biane1997b} are formulated in terms of  the large-$N$ limit of $\mx{S}_t^N$ on the space $\mathscr{X}_N = i\u_N$ of Hermitian $N\times N$ matrices, which is of course equivalent to the formulation above.  It also deals with a more general functional calculus on $\mathscr{X}_N$; cf.\ Section \ref{section Functional Calculus} below.  We have restricted our attention almost exclusively to the space of Laurent polynomial functions, for clarity of exposition.  Section \ref{section Equivariant} also discusses {\em equivariant functions}: an extension of the space of functional calculus functions which forms a natural domain for the Segal--Bargmann transform, and subsumes all other function spaces discussed in this paper.
\end{remark}

Biane proceeded in \cite{Biane1997b} to construct the free Hall transform transform $\mathscr{G}^t$ as a kind of large-$N$ limit $\U_N$ Segal--Bargmann, not by taking this limit directly as we have done, but instead developing a free probabilistic version of the Malliavin calculus techniques used by Gross and Malliavin \cite{Gross1996} to derive the properties of $B_t$ from an infinite dimensional version of $S_t$.  This laid the foundation for the modern theory of free Malliavin calculus and free stochastic differential equations, subsequently studied in \cite{Biane1998a,Biane2001,Kemp2012a} and many other papers, and was groundbreaking in many respects.  Biane conjectured that his transform $\mathscr{G}^t$ {\em is} the direct $N\to\infty$ limit of the Segal--Bargmann transforms $\mx{B}_{t,t}^N$ on $\U_N$, and suggested that this could be proved using the methods of stochastic analysis, but left the details of such an argument out of \cite{Biane1997b} (see the Remark on page 263).  One of the main motivations for the present paper is to prove (Theorems \ref{Main Theorem} and \ref{thm Biane transform}) that this connection indeed holds.  Our methods and ideas are very different from those Biane suggested, however; they are analytic and geometric, rather than probabilistic.  Moreover, we find the large-$N$ limit of the {\em two-parameter} Segal--Bargmann transform $\mx{B}^N_{s,t}$, and this generalization is essential to our proof that $\lim_{N\to\infty} \mx{B}^N_{t,t} = \mathscr{G}^t$.

\begin{remark} \label{remark Guillaume} As noted above, the complementary paper \cite{Guillaume2013} answers many of the same questions we do, using a somewhat different framework.  C\'ebron's paper uses the tools of free probability to construct a space of ``formal trace polynomials'' on which the limit Segal--Bargmann transform acts.  He also realizes the Laplace operator $\Delta_{\U_N}$ via an intertwining formula, in his case formulated in terms of free conditional expectation, and finds a crucial $O(1/N^2)$-decomposition analogous to our (\ref{e.DN.new}).  On the other hand, our method for connecting the large-$N$ limit of the Segal--Bargmann transform to the work of Biane (Theorem \ref{thm Biane transform})  is completely different from that of \cite{Guillaume2013}, using PDE methods to derive the polynomial generating function for the limiting transform; moreover, our methods extend naturally to the two-parameter transform.  A more complete understanding of the large-$N$ limit of the Segal--Bargmann transform on $\U_N$ is likely achieved by considering both our approach and C\'ebron's together.
\end{remark}

\section{Equivariant Functions and Trace Polynomials\label{section Equivariant}}

In this section, we consider function spaces over $\U_N$ and $\GL_N$ that are very natural domains for the Segal-Bargmann transform and its inverse.

\begin{definition} \label{definition equivariant} Let $G\subset \M_N$ be a matrix group.  A function $F\colon G\to \M_N$ is called {\bf equivariant} if $F(BAB^{-1}) = BF(A)B^{-1}$ for all $A,B\in G$ (it is equivariant under the adjoint action of $G$). \end{definition}
\noindent The set of equivariant functions is a $\C$-algebra.  If $P\in\C[u,u^{-1};\mx{v}]$, then the trace  polynomial $P_N$ is equivariant, as can be easily verified.  This shows that the {\bf equivariant subspaces} 
\[ L^2(\U_N,\rho_s^N;\M_N)_{\mathrm{eq}} \quad \text{and} \quad \mathcal{H}L^2(\GL_N,\mu_{s,t}^N;\M_N)_{\mathrm{eq}}, \]
are non-trivial.  The main results of this section, Theorem \ref{BstEquivariant.prop} and \ref{LaurentDense.prop }, show that $\mx{B}^N_{s,t}$ maps $L^2(\rho^N_s)_{\mathrm{eq}}$ onto $\mathcal{H}L^2(\mu_{s,t}^N)_{\mathrm{eq}}$ (extending Theorem \ref{t.new1.8}), and that trace  polynomials are dense in these equivariant $L^2$-spaces.  We conclude this section with Theorem \ref{prop P-->P_N unique}, showing that the map $\C[u,u^{-1};\mx{v}]\to L^2(\rho_s^N)_{\mathrm{eq}}$ given by $P\mapsto P_N$ is one-to-one when restricted to polynomials if a fixed maximal degree.

We begin with a brief discussion of {\bf functional calculus}, which featured prominently in \cite{Biane1997b}, and whose image is a (small) subspace of equivariant functions.

\subsection{Functional Calculus\label{section Functional Calculus}}

\begin{definition} \label{def functional calculus U(N)} Let $\U$ denote the unit circle in $\C$.
For every measurable function $f:\U\rightarrow\mathbb{C}$, let $f_{N}$ be
the unique function mapping $\U_N$ into $M_{N}(\mathbb{C})$ with the property
that%
\[
f_{N}\left(  V\left(
\begin{array}
[c]{ccc}%
\lambda_{1} &  & \\
& \ddots & \\
&  & \lambda_{N}%
\end{array}
\right)  V^{-1}\right)  =V\left(
\begin{array}
[c]{ccc}%
f(\lambda_{1}) &  & \\
& \ddots & \\
&  & f(\lambda_{N})
\end{array}
\right)  V^{-1}%
\]
for all $V\in \U_N$ and all $\lambda_{1},\ldots,\lambda_{N}\in \U$. The
function $f_N$ is called the {\bf functional calculus function}
associated to the function $f$. The space of those functional calculus
functions that are in $L^{2}(\U_N,\rho_{s}^N;M_{N}(\mathbb{C}))$ is called the
{\bf functional calculus subspace}.
\end{definition}

It is easy to check that $f_{N}(U)$ is well defined, independent of the choice
of diagonalization. If, for example, $f$ is the function given by
$f(\lambda)=e^{\lambda}$, then $f_{N}(U)=e^{U}$, computed by the usual power
series. If $f\in\C[u,u^{-1}]$, then $f_{N}$ is the function given in (\ref{e.L.fc}); thus
our notation $f_N$ for both is consistent.  (By comparison: in \cite{Biane1997b}, the functional
calculus function $f_N$ is denoted $\theta^N_f$.) Trace polynomials are {\em not}, in general, functional calculus functions.
For example, the function $F(U)=U\mathrm{tr}(U)$ is not a functional calculus
function on $\U_N$, except when $N=1$.  Indeed, if $N\ge 2$ and $\U_N\ni U = \mathrm{diag}(\lambda_1,\lambda_2)$, the $(1,1)$-entry of the diagonal matrix $U\tr U$ is $\frac12(\lambda_1+\lambda_2)\lambda_1$, which is not a function of $\lambda_1$ alone.  This violates Definition \ref{def functional calculus U(N)}.  Functional calculus functions are, however, equivariant.

Since $\Lambda(f) \equiv \int_{\U_N} \tr(f_N(U))\,\rho_s^N(dU)$ defines a positive linear functional on $C(\U)$ with $\Lambda(1)=1$, by the Riesz Representation Theorem \cite[Theorem 2.14]{RudinBook} there is a probability measure $\nu_s^N$ on $\U$ such that
\begin{equation} \label{eq nu_s^N} \int_{\U_N}\tr(f_N(U))\,\rho_s^N(dU) = \Lambda(f) = \int_{\U} f(\xi)\,\nu_s^N(d\xi), \qquad f\in C(\U). \end{equation}
(Theorem \ref{t.concentration} shows, in particular, that $\nu_s^N$ converges weakly to $\nu_s$; cf.\ Theorem \ref{t.new1.9}.)  For any function $f$ on $\U$, one can easily verify from Definition \ref{def functional calculus U(N)} that $[|f|^2]_N(U) = f_N(U)f_N(U)^\ast$; hence, by the density of $C(\U)$ in $L^2(\U,\nu_s^N)$, (\ref{eq nu_s^N}) shows that
\begin{equation} \label{eq L^2 nu_s^N}
\left\Vert f_{N}\right\Vert _{L^{2}(\U_N,\rho^N_{s};M_{N}(\mathbb{C}%
))}=\left\Vert f\right\Vert _{L^{2}(\U,\nu_{s}^{N})}, \qquad f\in L^2(\U,\nu^N_s).
\end{equation}
It follows that the functional calculus subspace is a closed subspace of $L^2(\rho_s^N)_{\mathrm{eq}}$, and contains the functions $\{f_1\colon f\in\C[u,u^{-1}]\}$ as a dense subspace.  That this density result extends to trace  polynomials in the full space $L^2(\rho_s^N)_{\mathrm{eq}}$ is Theorem \ref{LaurentDense.prop } below.

If $F$ is a holomorphic function on $\mathbb{C}^\ast$, there is a unique
holomorphic function $F_{N}$ from $\GL_N$ to $\M_N$
which satisfies%
\[
F_{N}\left(  A\left(
\begin{array}
[c]{ccc}%
\lambda_{1} &  & \\
& \ddots & \\
&  & \lambda_{N}%
\end{array}
\right)  A^{-1}\right)  =A\left(
\begin{array}
[c]{ccc}%
F(\lambda_{1}) &  & \\
& \ddots & \\
&  & F(\lambda_{N})
\end{array}
\right)  A^{-1}%
\]
for every $A\in \GL_N$ and all $\lambda_{1},\ldots,\lambda_{N}
\in\mathbb{C}^{\ast}$; indeed, $F_N$ is given by the same Laurent series expansion as $F$, applied to the matrix variable. 
We call such a function a {\bf holomorphic functional calculus function} on $\GL_N$.
As (\ref{BtU2}) shows, the boosted Segal--Bargmann transform $\mathbf{B}^N_{s,t}$
{\em does not}, in general, map functional calculus functions on $\U_N$ to holomorphic
functional calculus functions on $\GL_N$. Nevertheless, \cite{Biane1997b} suggests that {\em in the large-$N$ limit}, $\mathbf{B}^N_{s,t}$ ought to map functional calculus functions to holomorphic functional calculus functions (at least in the $s=t$ case). Since single-variable Laurent polynomial functions are dense in the functional calculus subspace, Theorem \ref{Main Theorem} can be interpreted as a rigorous version of this idea.

\subsection{Results on Equivariant Functions}

\begin{theorem}
\label{BstEquivariant.prop}Let $s,t>0$ with $s>t/2$.  The Segal--Bargmann transform $\mathbf{B}^N_{s,t}$ maps the equivariant subspace $L^{2}(\U_N,\rho^N _{s};M_{N}(\mathbb{C}))_{\mathrm{eq}}$ isometrically onto $\mathcal{H}L^{2}(\GL_N,\mu^N _{s,t};M_{N}(\mathbb{C}))_{\mathrm{eq}}$.
\end{theorem}

\noindent We begin with the following lemma.

\begin{lemma} Let $G\subset \M_N$ be a group.  For any function $F\colon G\to \M_N$, define
\begin{equation} \label{eq conj action}
C_{V}(F)(A)=V^{-1}F(VAV^{-1})V, \qquad V,A\in G.
\end{equation}
Let $s,t>0$ with $s>t/2$.  Then for all $F\in L^2(\U_N,\rho_s^N;\M_N)$ and $V\in \U_N$,
\begin{equation} \label{e.Bstconj} \mx{B}_{s,t}^N(C_VF) = C_V(\mx{B}_{s,t}^NF). \end{equation}
\end{lemma}

\begin{proof} Since $\Delta_{\U_N}$ is bi-invariant, it commutes with the left- and right-actions of the group; hence it, and therefore the semigroup $e^{\frac{t}{2}\Delta_{\U_N}}$, commutes with the adjoint action $\mathrm{Ad}_V(U) = VUV^{-1}$ on functions: for any $V\in \U_N$,
\begin{equation} \label{e.conj1} e^{\frac{t}{2}\Delta_{\U_N}} \left(F\circ(\mathrm{Ad}_V)\right) = \left(e^{\frac{t}{2}\Delta_{\U_N}}F\right)\circ\mathrm{Ad}_V. \end{equation}
Conjugating  both sides of (\ref{e.conj1}) by $V^{-1}$ in the range of $F$ (which commutes with the heat operator), it follows that
\begin{equation} \label{e.conj2} C_V(e^{\frac{t}{2}\Delta_{\U_N}}F) = e^{\frac{t}{2}\Delta_{\U_N}}(C_VF), \qquad V\in \U_N. \end{equation}
Uniqueness of analytic continuation now proves (\ref{e.Bstconj}) from (\ref{e.conj2}).
\end{proof}

Theorem \ref{BstEquivariant.prop} now follows by analytically continuing (\ref{e.Bstconj}) in the $V$ variable.

\begin{proof}[Proof of Theorem \ref{BstEquivariant.prop}] Let $F\in L^2(\U_N,\rho^N_s;\M_N)$ be equivariant; thus $C_VF = F$ for all $V\in \U_N$.  Then (\ref{e.Bstconj}) shows that $C_V(\mx{B}_{s,t}^N F) - \mx{B}_{s,t}^NF\equiv 0$ for each $V\in \U_N$.  Since $\mx{B}_{s,t}^N F$ is holomorphic, it follows by uniqueness of analytic continuation that the function $Z\mapsto C_Z(\mx{B}_{s,t}^N F) - \mx{B}_{s,t}^N F$ is $0$ for $Z\in \GL_N$; thus, $\mx{B}_{s,t}^N F$ is equivariant under $\GL_N$, as required. An entirely analogous argument applies to the inverse transform, establishing the theorem.
\end{proof}

Let us remark here on an intuitive approach to the concentration of measure results in Section \ref{Section Limit Theorems}.  If $U_t$ is a random matrix sampled from the distribution $\rho_t^N$ on $\U_N$, its (random) eigenvalues converge to their (deterministic) mean as $N\to\infty$.  To be precise: if $\lambda_1^N,\ldots,\lambda_N^N$ are the eigenvalues of $U_t$, the {\em empirical eigenvalue measure}
\[ \widetilde{\nu}^N_t = \frac1N\sum_{j=1}^N \delta_{\lambda_j^N} \]
converges weakly almost surely to $\nu_t$.  (The mean of the random measure $\widetilde{\nu}_t^N$ is the measure $\nu_t^N$ of (\ref{eq nu_s^N}) which converges weakly to $\nu_t$; cf.\ Theorem \ref{t.new1.9}.  The stronger statement that the convergence is almost sure, not just in expectation, was first proved in \cite{Rains1997}.  See \cite{Kemp2013} for the strongest known convergence results.)

The conjugacy classes in the group $\U_N$ are in one-to-one correspondence with the (symmetrized) list of eigenvalues.  Each such list is, in turn, determined by its empirical measure $\widetilde{\nu}_t^N$.  The convergence of the random eigenvalues of $U_t$ to a deterministic limit therefore suggests that the heat kernel measure $\rho_t^N$ concentrates its mass on a {\em single conjugacy class} as $N\to\infty$.  The following proposition therefore offers some insight into Theorem \ref{t.concentration} (that trace  polynomials concentrate on single-variable Laurent polynomials).  Indeed, on a fixed conjugacy class, {\em any equivariant function} is given by a polynomial.



\begin{proposition}
\label{singleConjClass.prop}Let $G\subseteq \M_N$ be a group, and let $C$ be a conjugacy class in $G$. If $F\colon G\rightarrow \M_N$ is equivariant, then there exists a single-variable polynomial $P_C$ such that $F(A)=P_C(A)$ for all $A\in C$.
\end{proposition}

\begin{proof} Fix a point $A_{0}$ in $C$, and let $A_1$ commute with $A_0$.  Then since $F$ is equivariant,
\[ A_1^{-1}F(A_0)A_1 = F(A_1^{-1}A_0A_1) = F(A_0), \]
which shows that $F(A_0)$ commutes with any such $A_1$: that is, $F(A_0)\in\{A_0\}''$ is in the double commutant of $A_0$.  A classical theorem in linear algebra (see, for example, \cite{Lagerstrom1945} for a
short proof) then asserts that there is a single-variable polynomial $P_{A_0}$
such that $F(A_{0})=P(A_{0})$. Every other point in the conjugacy class $C$
is of the form $A=BA_{0}B^{-1}$ for some $B\in G$. Since applying a
polynomial function to a matrix commutes with conjugation, we have
\begin{equation*}
F(A) = F(BA_{0}B^{-1})=BF(A_{0})B^{-1}=BP_{A_0}(A_{0})B^{-1}=P_{A_0}(BA_{0}B^{-1}) = P_{A_0}(A)
\end{equation*}%
which shows that the map $A_0\mapsto P_{A_0}$ is constant for $A_0\in C$, so relabel $P_{A_0} = P_C$.  Thus, the identity $F(A)=P_C(A)$ holds for all $A\in C$.
\end{proof}

\begin{remark} Proposition \ref{singleConjClass.prop} has the at-first-surprising consequence that the equivariant function $F(A) = A^{-1}$ is equal to a polynomial (not a Laurent polynomial) on any given conjugacy class.  This can be seen as a consequence of the Cayley-Hamilton Theorem; cf.\ Section \ref{section P-->P_N unique}.  Indeed, let $p_A(\lambda) = \det(\lambda I_N-A)$ be the characteristic polynomial of $A$; then $p_A(A)=0$.  This shows there are coefficients $c_k$ (determined by $A$) so that $\sum_{k=0}^N c_k A^k =0$.  Since $c_0 = (-1)^N\det(A)$, if $A$ is invertible we can therefore factor out $A$ from the $k\ge 1$ terms and solve for $A^{-1}$ as a polynomial in $A$.  The above proof shows that this $A$-dependent polynomial is, in fact, uniform over the whole conjugacy class.  \end{remark}

\subsection{Density of Trace Polynomials}

Conceptually, equivariant functions are a natural arena for the Segal--Bargmann transform in the large-$N$ limit.  Computationally, it will be convenient to work on the subclass of trace  polynomials.  In fact, trace  polynomials are dense in $L^2(\U_N,\rho^N_s;\M_N)_{\mathrm{eq}}$.  Thus, understanding the action of $\mx{B}_{s,t}^N$ on this class tells the full story.

\begin{theorem}
\label{LaurentDense.prop }For $s>0$, the space of trace  polynomials is dense in the equivariant space $L^{2}(\U_N,\rho^N_{s};M_{N}(\mathbb{C}))_{\mathrm{eq}}$.
\end{theorem}

\noindent We begin by proving that equivariant functions whose entries are polynomials in $U$ and $U^\ast$ are dense.

\begin{lemma}
\label{polyEq.lem}Every equivariant function $F\in L^{2}(\U_N,\rho^N _{s};M_{N}(\mathbb{C}))_{\mathrm{eq}}$ can be approximated by a sequence of equivariant matrix-valued functions $F_{n}$, where each entry of $F_{n}(U)$ is a polynomial in
the entries of $U$ and their conjugates.
\end{lemma}

\begin{proof}
By the Stone--Weierstrass Theorem and the density of continuous functions in $L^{2}$, any $f\in L^2(\rho^N_s)$ can be approximated by scalar-valued polynomial functions of the entries of the $\U_N$ variable and their conjugates.  Applying this result to the components of the matrix-valued function $F$, we see that there is a sequence $P_n$ of polynomials in the entries of $U$ and their conjugates such that
\begin{equation} \label{eq poly approx 0} \lim_{n\to\infty}\|P_n-F\|_{L^2(\U_N,\rho^N_s;\M_N)}= 0. \end{equation}
Now, consider again the conjugation action $C_V$ of (\ref{eq conj action}).  It is easy to verify that
this action preserves the space of homogeneous polynomials of degree $m$ in
the entries $U_{jk}$ and their conjugates.  Thus, the averaged function
\[ F_n(U) = \int_{\U_N} C_V(P_n)(U)\,dV \]
is still a polynomial in the entries of $U$ and their conjugates; and $F_n$ is evidently equivariant.  Therefore $C_V(F) = F$ for each $V\in \U_N$, and so
\[ F_n(U)-F(U) =  \int_{\U_N} C_V(P_n)(U)\,dV - F(U) = \int_{\U_N} [C_V(P_n)-C_V(F)](U)\,dV. \]
It follows from (\ref{eq poly approx 0}) (with an application of Minkowski's inequality and the dominated convergence theorem) that $F_n$ approximates $F$ in $L^2(\U_N,\rho^N_s;\M_N)$ as claimed.
\end{proof}

\begin{proof}[Proof of Theorem \ref{LaurentDense.prop }]
We will show that each of the functions $F_{n}$ in Lemma \ref{polyEq.lem} is
actually a trace  polynomial. Suppose, then, that $F$ is equivariant
and that each entry of $F(U)$ is a polynomial in the entries of $U$ and their
conjugates.  Let $T(N)\subset \U_N$ denote the diagonal subgroup.  By the spectral theorem, any $U\in \U_N$
has a unitary diagonalization $U = V\Lambda V^{-1}$ for some $\Lambda\in T(N)$.  The equivariance of $F$ then gives that
$F(U) = F(V\Lambda V^{-1}) = VF(\Lambda)V^{-1}$.  In particular, any equivariant function $F$ is completely determined by its restriction $\left.F\right|_{T(N)}$ to the diagonal subgroup.

Because $F$ is equivariant, by the same argument used in the proof of Proposition \ref{singleConjClass.prop},
$F(U)\in \{U\}''$ for each $U$.  Let $U\in T(N)$ be in the dense subset of matrices with all eigenvalues distinct; then $\{U\}'$ is the set of all diagonal matrices, and so $F(U)$ commutes with all diagonal matrices, meaning that $F(U)$ is diagonal.  By the initial assumption on $F$, all entries of $F(U)$ are polynomials in the entries and their conjugates; hence, since the off-diagonal entries are $0$ on a dense set, $F(U)$ is diagonal for all $U\in T(N)$, and its diagonal entries are polynomials in the diagonal entries $\lambda_1,\ldots,\lambda_N$ of $U$ and their conjugates. Of course, for $U\in T(N)$, the diagonal entries of $U$ satisfy $\bar{\lambda}_{j}=1/\lambda _{j}$. Thus, each of the diagonal entries of $\left.F\right|_{T(N)}(U)$ is a Laurent polynomial $q(\lambda_1,\ldots,\lambda_N)$ in the $\lambda _{j}$'s.  The symmetric group $\Sigma_N$ is a subgroup of $\U_N$, so since $\left.F\right|_{T(N)}$ is equivariant under $\U_N$, it is also equivariant under $\Sigma_N$.  Hence each of the (matrix-valued) polynomials $q$ is equivariant under the action of $\Sigma_N$ on the diagonal entries.

Taking $k$ to be larger than the largest negative degree of any variable in $q$, and setting $r(\lambda_1,\ldots,\lambda_N) = (\lambda_1\cdots \lambda_N)^k q(\lambda_1,\ldots,\lambda_N)$, $r$ is also equivariant under the action of $\Sigma_N$.  We can then express
\[ \left.F\right|_{T(N)}(U) = (\lambda_1\cdots\lambda_N)^{-k} r(\lambda_1,\ldots,\lambda_N) = \det(U^\ast)^k r(\lambda_1,\ldots,\lambda_N). \]
Since the diagonal entries of $r(\lambda _{1},\ldots \lambda _{N})$ are equivariant under permutations, the first entry of $r$ must be invariant under permutations of
the remaining $N-1$ variables. This means that the first entry of $r$ is a
linear combination of terms of the form $\lambda _{1}^{\ell}s_{\ell}(\lambda_{2},\ldots ,\lambda _{N})$, where $\ell$ ranges from $0$ up to the degree $d$ of $r$ and $s_{\ell}$ is a symmetric polynomial in $N-1$
variables. By equivariance under $\Sigma_N$, it now follows that, for $1\le j\le N$, the $j$th diagonal component of $r$ itself must be a linear combination of terms of the form
\begin{equation*}
\left\{\lambda_{j}^{\ell}s_{\ell}(\lambda _{1},\ldots ,\widehat{\lambda _{j}},\ldots \lambda_{N})\colon 0\le \ell\le d\right\}.
\end{equation*}
It is well-known that every symmetric polynomial in $N-1$ variables $\lambda_{1},\ldots,\lambda_{N-1}$ is a polynomial in power-sums $p_\ell(\lambda_1,\ldots,\lambda_{N-1})$ with $0\le\ell\le N-1$, where, for any integer $\ell$, 
\begin{equation} \label{eq power sums}
p_\ell(\lambda_1,\ldots,\lambda_{N-1}) = \lambda_{1}^{\ell}+\lambda_{2}^{\ell}+\cdots +\lambda_{N-1}^{\ell}.
\end{equation}
(This result was known at least to Newton.  For a proof, see \cite[Theorem 4.3.7]{SaganBook}.)
 Furthermore, any power sum in $N-1$ variables can be written as a linear combination of power sums of $N$ variables along with the monomials $\lambda_j^\ell$; for example
\begin{equation*}
\sum_{j=2}^{N}\lambda _{j}^{\ell}=\left( \sum_{j=1}^{N}\lambda _{j}^{\ell}\right)
-\lambda _{1}^{\ell}.
\end{equation*}%
Thus, the first entry of $r$ is actually a polynomial in power-sums of all $N$ variables and in $\lambda _{1}$ with the remaining
entries of $r$ then being determined by equivariance with respect to
permutations.

Suppose now that $r$ is the permutation-equivariant polynomial whose $j$th
entry is 
\begin{equation*}
\lambda _{j}^{\ell_{0}}\left( \lambda _{1}^{k_{1}}+\cdots +\lambda
_{N}^{k_{1}}\right) ^{\ell_{1}}\cdots \left( \lambda _{1}^{k_{M}}+\cdots
+\lambda _{N}^{k_{M}}\right) ^{\ell_{M}}.
\end{equation*}%
Then $r$ is nothing but the restriction to $T(N)$ of the trace polynomial 
\begin{equation*}
R(U) = U^{\ell_{0}}\mathrm{Tr}(U^{k_{1}})^{\ell_{1}}\cdots \mathrm{Tr}(U^{k_{M}})^{\ell_{M}}.
\end{equation*}%
Meanwhile, by the above-quoted result, the
symmetric polynomial $(\lambda _{1}\lambda _{2}\cdots \lambda _{N})^{k}$ can
be expressed as a polynomial in the power-sums of the $\lambda _{j}$s.
Taking the complex-conjugate of this result, we see that $\det (U^{\ast})^{k}$ can be expressed as a scalar trace polynomial in $U^{\ast}$; thus $U\mapsto (\det U^\ast)^kR(U)$ is a trace  polynomial.
Hence $\left.F\right|_{T(N)}$ is the restriction of the trace  polynomial function $U\mapsto (\det U^\ast)^k R(U)$, and the result
follows since $F$ is determined by $\left.F\right|_{T(N)}$.
\end{proof}

\subsection{Asymptotic Uniqueness of Trace Polynomial Representations} \label{section P-->P_N unique}

The Cayley--Hamilton theorem asserts that, for any matrix $A\in \M_N$, it follows that $p_A(A) = 0$ where $p_A(\lambda) = \det(\lambda I_N-A)$ is the characteristic polynomial of $A$.  In fact, the coefficients of the characteristic polynomial $p_A$ are all scalar trace polynomial functions of $A$: this follows from the Newton identities.  Using the operators $\mathcal{M}_{(\cdot)}$ and $\mathcal{A}_+$ of Definition \ref{Def LN L} below, there is an explicit formula for $p_A$.  Let
\[ h_A(\lambda) = \exp\left(-\sum_{m=1}^\infty \frac{1}{m\lambda^m} \Ur(A^m)\right). \]
Then for $A\in \M_N$, $p_A(\lambda) = (\mathcal{A}_+\mathcal{M}_{\lambda^N} h_A)(\lambda)$.  (See the Wikipedia entry for the Cayley-Hamilton theorem.)  Thus, the expression $p_A(A)$ is a(n $N$-dependent) trace polynomial in $A$, and the Cayley--Hamilton theorem asserts that this trace polynomial function vanishes identically on $\M_N$.  We illustrate this result in the case $N=2$.

\begin{example}
\label{example Cayley-Hamilton} For all $A\in M_{2}(\mathbb{C})$, the
Cayley--Hamilton Theorem asserts that%
\begin{equation}
A^{2}-\mathrm{Tr}(A)A+\det (A)I_2=0.  \label{Cayley}
\end{equation}%
In the $2\times 2$ case, however, it is easily seen that%
\begin{equation}
\det (A)=\frac{1}{2}(\mathrm{Tr}(A)^{2}-\mathrm{Tr}(A^{2})).  \label{detID}
\end{equation}%
Substituting (\ref{detID}) into (\ref{Cayley}) and expressing things in
terms of the normalized trace gives%
\begin{equation*}
A^{2}-2A\mathrm{tr}(A)+2\mathrm{tr}(A)^{2}I_2-\mathrm{tr}(A^{2})I_2=0
\end{equation*}%
for all $A\in M_{2}(\mathbb{C})$. In particular, if $P\in\C[u,u^{-1};\mx{v}]$ denotes the nonzero
polynomial $P(u;\mx{v}) =u^{2}-2uv_{1}+2v_{1}^{2}-v_{2}$, then $P_{2}\colon \U_2\to\M_2$ is the zero function.  Note, however, that $P_{N}$ is not the zero function on $\U_N$ for $N>2$, since the minimal polynomial of a generic element of $\U_N$ has degree $N$. This demonstrates the following theorem.
\end{example}

\begin{theorem} \label{prop P-->P_N unique}
Let $P$ be a nonzero element of $\C[u,u^{-1};\mx{v}]$.  Then, for all sufficiently large $N$, the trace  polynomial function $P_N$ is not identically zero on $\U_N$.  In particular, if $P,Q\in\C[u,u^{-1};\mx{v}]$ are such that $P_N=Q_N$ for all sufficiently large $N$, then $P=Q$.
\end{theorem}

In order to prove Theorem \ref{prop P-->P_N unique}, the following lemma (from the theory of symmetric functions) is useful.  The corresponding statement for symmetric polynomials (rather than Laurent polynomials) is a standard result.  The Laurent polynomial case must be known, but is well hidden in the literature.

\begin{lemma}
\label{algIndep.lemma}If $N\ge 2n$, then the power sums $p_{k}(\lambda
_{1},\ldots ,\lambda _{N})$ (cf.\ (\ref{eq power sums})) with $0<\left\vert k\right\vert \leq n$ are
algebraically independent elements of the ring of rational function in $N$
variables. 
\end{lemma}

\begin{proof}
Let $e_{j}$ denote the $j$th elementary symmetric polynomial in $N$ variables;
that is, $e_j$ the sum of all products of exactly $j$ of the $N$ variables. Then the power sums $p_{1},\ldots ,p_{n}$
can be expressed as linear combinations of the functions $e_{1},\ldots
,e_{n}$. Thus, it suffices to prove the independence of the functions $%
e_{j}(\lambda _{1},\ldots ,\lambda _{N})$ and $e_{j}(\lambda
_{1}^{-1},\ldots ,\lambda _{N}^{-1})$ for $1\leq j\leq n$. We may easily
see, however, that%
\begin{equation*}
e_{j}(\lambda _{1}^{-1},\ldots ,\lambda _{N}^{-1})=\frac{e_{N-j}(\lambda
_{1},\ldots ,\lambda _{N})}{e_{N}(\lambda _{1},\ldots ,\lambda _{N})}.
\end{equation*}

In the case $N=2n$, we need to establish the independence of the functions $%
e_{1},\ldots ,e_{N/2}$ and $e_{N/2}/e_{N},\ldots ,e_{N-1}/e_{N}$, which
follows easily from the known independence of $e_{1},\ldots ,e_{n}$;
cf.\ \cite[Theorem 4.3.7]{SaganBook}. In the case $N>2n$, if we had an algebraic relation
among the functions $e_{j}(\lambda _{1},\ldots ,\lambda _{N})$ and $%
e_{j}(\lambda _{1}^{-1},\ldots ,\lambda _{N}^{-1})$ for $1\leq j\leq n$, we
could clear $e_{N}$ from the denominator to obtain an algebraic relation
among the functions $e_{1},\ldots ,e_{n}$, $e_{N-1},\ldots ,e_{N-n}$ and $%
e_{N}$, which is impossible. 
\end{proof}

We now proceed with the scalar version of Theorem \ref{prop P-->P_N unique}.

\begin{lemma}
\label{scalarNonvanish.lem} Let $Q\in\C[\mx{v}]$.  Let $N\ge 2n$.  Then $Q_N$ is not identically zero on $\U_N$.
\end{lemma}

\begin{proof} Since $Q_N$ is a trace  polynomial, it also defines a holomorphic function on $\GL_N$.  By uniqueness of analytic continuation, if $Q_N\equiv 0$ on $\U_N$, then $Q_N\equiv 0$ on $\GL_N$.  To prove the lemma, it therefore suffices to find $A\in \GL_N$ with $Q_{N}(A)\neq 0$.  Actually, we will find a diagonal matrix $A\in \GL_N$ with $Q_{N}(A)\neq 0$.

For clarity, we write out the polynomial $Q$ in terms of its coefficients:
\[ Q(v_1,v_{-1},\ldots,v_n,v_{-n}) = \sum_{i_1,\ldots, i_n\atop j_1,\ldots,j_n} a_{i_1,\ldots,i_n}^{j_1,\ldots,j_n} \cdot v_1^{i_1} v_{-1}^{j_1}\cdots v_n^{i_n} v_{-n}^{j_n}. \]
Consider any diagonal matrix $\mathrm{diag}(\lambda_1,\ldots,\lambda_N)$ in $\GL_N$; for convenience, denote $\lambda = (\lambda_1,\ldots,\lambda_N)$.  Then $\tr(A^k) = p_k(\lambda)$ (the power sum of (\ref{eq power sums})), and so
\begin{equation} \label{eq independence 0} Q_N(\mathrm{diag}(\lambda)) = \sum_{i_1,\ldots, i_n\atop j_1,\ldots,j_n} a_{i_1,\ldots,i_n}^{j_1,\ldots,j_n} \cdot p_1(\lambda)^{i_1}p_{-1}(\lambda)^{j_1}\cdots p_n(\lambda)^{i_n}p_{-n}(\lambda)^{j_n}. \end{equation}
By Lemma \ref{algIndep.lemma}, the power sums $p_1(\lambda),p_{-1}(\lambda),\ldots,p_n(\lambda),p_{-n}(\lambda)$ are algebraically independent since $\lambda=(\lambda_1,\ldots,\lambda_N)$ and $N\ge 2n$.  Since $Q\ne 0$, some of the coefficients $a_{i_1,\ldots,i_n}^{j_1,\ldots,j_n}$ in (\ref{eq independence 0}) are $\ne 0$.  It follows that $Q_N(\mathrm{diag}(\lambda))$ is not identically $0$, as desired.  \end{proof}

This finally brings us to Theorem \ref{prop P-->P_N unique}.

\begin{proof}[Proof of Theorem \ref{prop P-->P_N unique}] Let $P(u;\mx{v})= \sum_\ell Q_\ell(\mx{v})u^\ell$ with $Q_\ell\in\C[\mx{v}]$; then at least one $Q_\ell\ne 0$. Let us multiply $P_{N}(U)$ by $U^{k}$ for some large $k$, so that all the
untraced powers of $U$ in $U^{k}P_{N}(U)$ are non-negative. Let $\ell$ be the
highest untraced power of $U$ occurring in the expression for $U^{k}P_{N}(U).
$ Choose $N$ large enough so that $N>\ell$ and so that (Lemma \ref{scalarNonvanish.lem})
the coefficient $Q_{\ell}$ of $U^{\ell}$ in $P_{N}(U)$ is not identically zero. Then $Q_\ell$ is nonzero on a nonempty open subset of $\U_N$.
This set contains a matrix $U_{0}$ whose minimal polynomial has degree $N>\ell$.
When we evaluate $P_{N}(U_{0})$, the result will be a linear combination of
powers of $U_{0}$ with the coefficient of $U_{0}^{\ell}$ being nonzero. Since
the minimal polynomial of $U_{0}$ has degree $N>\ell$, the value of $P_{N}(U_{0})$ is not zero.
\end{proof}

\section{The Laplacian and Heat Operator on Trace Polynomials\label{section computing the heat operator}}

This section is devoted to a complete description of the action of the Laplacian $\Delta_{\U_N}$ on trace polynomial functions, and its corresponding lift to $\mathcal{D}_N$ on the space $\C[u,u^{-1};\mx{v}]$; cf.\ Theorem \ref{t.intertwine.new}.  We begin by proving ``magic formulas'' expressing certain quadratic matrix sums in simple forms.  We use these to give derivative formulas that allow for the routine computation of $\Delta_{\U_N} P_N$ for any $P\in\C[u,u^{-1};\mx{v}]$, and we then use these to prove the intertwining formula of Theorem \ref{t.intertwine.new}.  We conclude by proving a more general intertwining formula (Theorem \ref{t.inter2}) for the action of $A_{s,t}^N$ on trace polynomial functions over $\GL_N$; in this latter case, we deal more generally with trace  polynomials in $Z$ and $Z^\ast$ as this will be of use in Section \ref{Section Limit Theorems}.

\subsection{Magic Formulas\label{section Magic Formulas}}

We define an inner-product on $\M_N$ by
\begin{equation}
\langle X,Y\rangle=N\mathrm{Tr}\,(Y^{\ast }X)=N^{2}\tr(Y^{\ast }X).  \label{eq scaled inner-product}
\end{equation}
Restricted to the Lie algebra $\u_N$ (consisting of all skew-Hermitian matrices in $\M_N$), $\langle\cdot,\cdot\rangle$ is real-valued; it is the polarized inner product corresponding to the norm $\Vert\cdot\Vert _{\mathfrak{u}_N}$ of (\ref{scalingMetric}).  (This is not to be confused with the polarized inner-product corresponding to the norm $\|\cdot\|_{\M_N}$ of (\ref{eq M_N metric}).)

The main result of this section, which underlies all computations throughout
this paper, is the following list of \textquotedblleft magic
formulas\textquotedblright .

\begin{proposition}
\label{p.magica2} Let $\beta_N$ be any
orthonormal basis for $\u_N$ with respect to the
inner-product in (\ref{eq scaled inner-product}). Then we have the following
\textquotedblleft magic\textquotedblright\ formulas: for any $A,B\in \M_N$,
\begin{align}
\sum_{X\in \beta _{N}}X^{2}& =-I_{N}  \label{e.m1a}, \\
\sum_{X\in \beta _{N}}XAX& =-\tr(A)I_{N}  \label{e.m2a}, \\
\sum_{X\in \beta _{N}}\tr(XA)X& =-\frac{1}{N^{2}}A  \label{e.m3a}, \\
\sum_{X\in \beta _{N}}\tr(XA)\tr(XB)& =-\frac{1}{%
N^{2}}\tr(AB).  \label{e.m4a}
\end{align}
\end{proposition}

\begin{remark} \begin{itemize}
\item[(1)] Eq.\ (\ref{e.m1a}) is the $A=I_{N}$ special-case of (\ref{e.m2a});
similarly, (\ref{e.m4a}) follows from (\ref{e.m3a}) by
multiplying by $B$ and taking $\tr$. We separate them out as
distinct formulas for convenience in repeated use below.
\item[(2)] These and related ``magic'' formulas appeared in \cite[Lemma 4.1]{Sengupta2008}.
\end{itemize}
\end{remark}

\begin{proof}
If $\beta_N$ is a basis for the real vector space $\u_N$, it is also
a basis for the complex vector space $\M_{N}=\u_N\oplus i\u_N$.
Furthermore, if $\beta_N$ is (real) orthonormal in $\u_N$ with respect to
the (restricted real) inner product in (\ref{eq scaled inner-product}), then $\beta_N$ is (complex) orthonormal in $\M_N$ with respect to the (complex) inner-product in (\ref{eq scaled inner-product}).

Thus, let $\tilde{\beta}_N$ be any orthonormal basis for $\M_N$ with respect to (\ref{eq scaled inner-product}), and consider the linear map $\Phi\colon \M_N\rightarrow \M_N$ given by 
\begin{equation*}
\Phi (A)=\sum_{X\in \tilde{\beta}_N}X^\ast AX.
\end{equation*}
A routine calculation shows that $\Phi $ is independent of the choice of
orthonormal basis. We compute $\Phi $ by using the basis 
\begin{equation}
\tilde{\beta}_N\equiv \left\{ \frac{1}{\sqrt{N}}E_{jk}\right\} _{j,k=1}^{N} \label{mNbasis}
\end{equation}%
where $E_{jk}$ is the $N\times N$ matrix with a 1 in the $(j,k)$-entry and
zeros elsewhere.  Writing things out in terms of indices shows that, for any $A\in\M_N$, we have%
\begin{equation*}
N\cdot [\Phi(A)]_{\ell m} = \left[ \sum_{j,k=1}^{N}E_{kj}AE_{jk}\right] _{\ell m}=\sum_{j,k,n,o=1}^{N}\delta
_{k\ell}\delta _{jn}A_{no}\delta _{jo}\delta _{km} = \sum_{o}A_{oo}\delta _{\ell m},
\end{equation*}%
which says that%
\begin{equation*}
\Phi (A)=\frac{1}{N}\mathrm{Tr}(A)I_N=\tr(A)I_N.
\end{equation*}%
The basis-independence of $\Phi$ allows us to replace (\ref{mNbasis}) by any real 
orthonormal basis $\beta_N$ of $\u_N$ (which, as noted above, is also a complex orthonormal basis for $\M_N$).
The elements $X\in\beta_N$ are skew-Hermitian, and thus we obtain
\begin{equation*}
\sum_{X\in\beta_N}XAX=-\Phi (A)=-\tr(A)I_N,
\end{equation*}%
which is (\ref{e.m2a}).

Meanwhile, if we multiply both sides of (\ref{e.m3a}) by $-N^{2}$ and recall
that each $X$ is skew, we see that (\ref{e.m3a}) is equivalent to
the assertion that%
\begin{equation*}
A = \sum_{X\in\beta_N}N^{2}\tr(X^\ast A)X= \sum_{X\in\beta_N} \langle A,X\rangle X.
\end{equation*}%
But this identity is just the expansion of $A$ in the orthonormal basis $\beta_N$ for $\M_N$. Finally, as we have already remarked, (\ref{e.m1a}) and (\ref{e.m4a}) follow from (\ref{e.m2a}) and (\ref{e.m3a}), respectively.
\end{proof}

\subsection{Derivative Formulas\label{section Derivative Formulas}}


\begin{theorem}\label{t.derivatives} Let $m,n\in\mathbb{N}$. Let $\beta_N$ denote an orthonormal basis for $\u_N$, and let $X\in\beta_N$.  The following hold true:

\begin{align} \label{e.b0} \partial_{X}U^n  &=\sum_{j=1}^{n}U^{j}XU^{n-j}, \qquad n\ge 0 \\
\label{e.b0-}\del_X U^n &= -\sum_{j=n+1}^0 U^j X U^{n-j}, \qquad n<0
\end{align}

\begin{equation}
\partial_{X}\tr(U^n) =n\cdot\operatorname{tr}\left(  XU^{n}\right), \qquad n\in\Z \label{e.b0a}%
\end{equation}

\begin{align} \label{e.b1}
\Delta_{\U_N}U^n  &=-nU^n-2\1_{n\ge 2}\sum_{j=1}^{n-1} j U^j\tr(U^{n-j}), \qquad n\ge 0  \\
\label{e.b1-} \Delta_{\U_N}U^n &= nU^n+2\1_{n\le -2}\sum_{j=n+1}^{-1} jU^j\tr(U^{n-j}), \qquad n<0
\end{align}

\begin{align}  \label{e.b2}
\Delta_{\U_N}\tr(U^n)  &=-n\tr(U^n)-2\1_{n\ge 2}\sum_{j=1}^{n-1} j \tr(U^j)\tr(U^{n-j}), \qquad n\ge 0  \\
\label{e.b2-} \Delta_{\U_N}\tr(U^n) &= n\tr(U^n)+2\1_{n\le -2}\sum_{j=n+1}^{-1} j\tr(U^j)\tr(U^{n-j}), \qquad n<0
\end{align}
\begin{align} \label{e.b0b}
\sum_{X\in\beta_N}\partial_{X}U^m\cdot\partial_{X}\tr(U^n) &=-\frac{mn}{N^{2}}U^{n+m}, \qquad m,n\in\Z  \\
 \label{e.b0c} \sum_{X\in\beta_N}\partial_{X}\tr(U^m)\cdot\partial_{X}\tr(U^n)  &=-\frac{mn}{N^{2}}\tr(U^{n+m}), \qquad m,n\in\Z.
\end{align}
These formulas are valid for all matrices $U\in \M_N$; we will normally use them for $U\in \U_N$.
\end{theorem}

\begin{proof}
By the product rule, for $n\ge 0$
\[
\partial_{X}U^n =\left.\frac{d}{dt}\right|_{t=0}\left(
Ue^{tX}\right)^{n}=\sum_{j=1}^{n}U^{j}XU^{n-j}
\]
which proves (\ref{e.b0}).  Similarly, for $m>0$
\[ \del_X U^{-m} = \left.\frac{d}{dt}\right|_{t=0}\left(
e^{-tX}U^{-1}\right)^{m} = -\sum_{k=0}^{m-1} U^{-k}XU^{-(m-k)} \]
and letting $n=-m$ and reindexing $j=-k$ proves (\ref{e.b0-}). Taking traces of (\ref{e.b0}) and (\ref{e.b0-}) then gives
(\ref{e.b0a}) after using $\tr(AB)=\tr(BA)$ repeatedly. Making use of magic
formulas (\ref{e.m1a}) and (\ref{e.m2a}), we then have, for $n\ge 0$,
\begin{align*}
\Delta_{\U_N}U^n 
&=2\1_{n\ge 2}\sum_{1\le j<k\le n}\sum_{X\in\beta_N}U\dots\overset{j}{\overbrace{UX}}\dots\overset
{k}{\overbrace{UX}}\dots U+\sum_{j=1}^n \sum_{X\in\beta_N}U\dots\overset{j}{\overbrace
{UX^{2}}}\dots\dots U\\
&  =-2\1_{n\ge 2}\sum_{1\le j<k\le n}U^{n-(k-j)}\tr(U^{k-j})-nU^{n}.
\end{align*}
A little index gymnastics then reduces this last expression to the result in (\ref{e.b1}).  An entirely analogous computation proves (\ref{e.b1-}). Equations (\ref{e.b2}) and (\ref{e.b2-}) result from taking traces of (\ref{e.b1}) and (\ref{e.b1-}), since the linear functional $\tr$ commutes with $\Delta_{\U_N}$.  Finally, from (\ref{e.b0}) and (\ref{e.b0a}), when $m\ge 0$,
\begin{align*}
\sum_{X\in\beta_N}(\partial_{X}U^{m})  \tr(\partial_{X}U^n)   &  = n\sum_{X\in\beta_N}\sum_{j=1}^{m}U^{j}XU^{m-j}\tr(XU^{n}) \\
&  = n\sum_{X\in\beta_N} \sum_{j=1}^{m}U^{j}\tr(XU^{n})XU^{m-j} \\
&=-\frac{n}{N^{2}} \sum_{j=1}^{m}U^{j}U^{n}U^{m-j}
  =-\frac{mn}{N^{2}}U^{m+n}.
\end{align*}
An analogous computation for $m<0$ yields the same result,  proving (\ref{e.b0b}); and taking the trace of this formula gives (\ref{e.b0c}).
\end{proof}

%
%

\begin{remark} \label{remark eigen scaling} Eq.\ (\ref{e.b1}) shows that the identity function $\mathrm{id}(U)=U$ on $\U_N$ satisfies $\Delta_{\U_N}\mathrm{id} = -\mathrm{id}$.  It follows, for example, that all of the coordinate functions $U\mapsto U_{jk}$ are eigenfunctions of $\Delta_{\U_N}$ with eigenvalue $-1$, {\em independent of $N$}.  This independence suggests that we are, in fact, using the ``correct" scaling of the metric on $\U_N$, which in turn determines the scaling of $\Delta_{\U_N}$. If we used the unscaled Hilbert-Schmidt norm on $\u_N$, the function $\mathrm{id}$ would be an eigenvector for the Laplacian with eigenvalue $-N$; that scaling would not bode well for an infinite dimensional limit of any quantities involving the Laplacian.
\end{remark}

To illustrate how Theorem \ref{t.derivatives} may be used, we proceed to determine the action of the heat operator $e^{\frac{t}{2}\Delta_{\U_N}}$ on the polynomial $P_N(U) = U^2$.

\begin{example} \label{example heat U^2} Eq.\ (\ref{e.b1}) shows that $\Delta_{\U_N}U^2 = -2U^2-2U\tr U$.  In order to calculate $\Delta_{\U_N}(U\tr U)$, we use the definition (\ref{eq def Delta}) of $\Delta_{\U_N}$ and the product rule twice.  For each $X\in\u_N$,
\begin{equation*} \del_X^2(U\tr U) = \del_X\left[(\del_X U)\cdot \tr U + U\cdot (\del_X \tr U)\right] = (\del_X^2 U)\cdot\tr U + 2(\del_X U)(\del_X \tr U) + U\cdot\del_X^2\tr U. \end{equation*}
Summing over $X\in\u_N$ and using (\ref{e.b1}), (\ref{e.b2}), and (\ref{e.b0b}) then shows that
\[ \Delta_{\U_N} (U\tr U) = (-U)\cdot\tr U - \frac{2}{N^2} U^2 +U\cdot(-\tr U) = -\frac{2}{N^2} U^2 - 2 U\tr U. \]
Thus, setting $P_N(U)=U^2$ and $Q_N(U) = U\tr(U)$, we have
\begin{align}
\Delta _{\U_N}P_N &=-2P_N-2Q_N, \label{eq subspace 1}\\
\Delta _{\U_N}Q_N &=-\frac{2}{N^{2}}P_N-2Q_N. \label{eq subspace 2}
\end{align}
When $N>1$, the span of the two functions $P_N,Q_N$ is a $2$-dimensional subspace of $C^\infty(\U_N)$ (when $N=1$, $P_N=Q_N$). Equations (\ref{eq subspace 1}) -- (\ref{eq subspace 2}) show that this subspace is invariant under the action of $\Delta_{\U_N}$, which is represented there by the matrix
\[ D_N=\left[\begin{array}{cc}
-2 & -2/N^{2} \\ 
-2 & -2
\end{array}%
\right] .
\]
The exponentiated matrix $e^{\frac{t}{2}D_N}$ is easily computed (cf.\ \cite[Chapter 2, Exercises 6,7]{HallLieBook})
as
\begin{equation*}
e^{\frac{t}{2}D_N}=e^{-t}\left[ 
\begin{array}{cc}
\cosh (t/N) & -1/N\sinh (t/N) \\ 
-N\sinh (t/N) & \cosh (t/N)%
\end{array}%
\right] .
\end{equation*}%
It follows immediately (reading off from the first column of this matrix) that
\begin{equation*}
e^{\frac{t}{2}\Delta _{\U_N}} P_N =e^{-t}\cosh \left( t/N\right)
P_N-e^{-t}N\sinh (t/N)Q_N
\end{equation*}%
as claimed in (\ref{BtU2}).
\end{example}

Any trace polynomial function $P_N$ on $\U_N$ is contained in a finite-dimensional subspace of matrix-valued functions that is invariant under $\Delta _{\U_N}$; this follows from Theorem \ref{t.intertwine.new} and Corollary \ref{c.Pn-inv} below.  Thus, the computation of $e^{\frac{t}{2}\Delta _{\U_N}}P_N$ for any trace polynomial $P_N$ reduces to exponentiating a matrix of finite size.

\subsection{Intertwining Formulas I \label{section Intertwining I}}

We now explore how operations on trace polynomials are reflected in their
intertwining space $\C[u,u^{-1};\mx{v}]$. The
derivative formulas of Theorem \ref{t.derivatives} show that $\Delta
_{\mathbb{U}_{N}}$ preserves the space of trace polynomials with only positive
powers of $U$, and also preserves the space of trace polynomials with only
negative powers of $U$. This motivates the following projection operators on
$\C[u,u^{-1};\mx{v}]$.

\begin{definition}
\label{def A pm} Let $\mathcal{A}_{\pm}$ denote the \emph{positive }and
\emph{negative projection} operators
\[
\mathcal{A}_{+}\colon\C[u,u^{-1};\mx{v}]
\rightarrow\C[u;\mx{v}] \qquad\text{and}%
\qquad\mathcal{A}_{-}\colon\C[u,u^{-1};\mx{v}]
\rightarrow\C[u^{-1};\mx{v}]
\]
given by
\begin{equation}
\mathcal{A}_{+}\left(  \sum_{k=-\infty}^{\infty}u^{k}q_{k}(\mathbf{v})\right)
=\sum_{k=0}^{\infty}u^{k}q_{k}(\mathbf{v}),\quad\mathcal{A}_{-}\left(
\sum_{k=-\infty}^{\infty}u^{k}q_{k}(\mathbf{v})\right)  =\sum_{k=-\infty}%
^{-1}u^{k}q_{k}(\mathbf{v}).\label{eq def A pm}%
\end{equation}
Note that $\mathcal{A}_{+}+\mathcal{A}_{-}=\mathrm{id}_{\C[u,u^{-1};\mx{v}]}$, while $\mathcal{A}_{+}-\mathcal{A}%
_{-}=\mathrm{sgn}$ is the signum operator, where $\mathrm{sgn}(u^{n}%
)=\mathrm{sgn}(n)u^{n}$, and $\mathrm{sgn}(n)=n/|n|$ when $n\neq0$ and
$\mathrm{sgn}(0)=1$.
\end{definition}

\begin{remark}
\label{r.pseudo2} The Fourier transform conjugates the Hilbert transform with
the signum multiplier; in this sense, the operators $\mathcal{A}_{\pm}$ are
linear combinations of the identity and the Hilbert transform.
\end{remark}

\begin{notation}
\label{n.new1.19}For any $k\in\mathbb{Z},$ let $\mathcal{M}_{u^{k}}$ denote
the multiplication operator, $\mathcal{M}_{u^{k}}P(u;\mx{v})
=u^{k}P(u;\mx{v})  .$ Let $\mathcal{L}$ be the second order
linear differential operator on $\C[u,u^{-1};\mx{v}]
$ defined by
\begin{equation}
\mathcal{L}=\sum_{|j|,|k|\geq1}jkv_{k+j}\frac{\partial^{2}}{\partial
v_{j}\partial v_{k}}+2\sum_{|k|\geq1}ku^{k+1}\frac{\partial^{2}}{\partial
v_{k}\partial u}\label{e.new1.7}%
\end{equation}
where, for convenience, $v_{0}=1$; and let $\mathcal{D}$ be the first-order \emph{pseudo}differential operator on
$\C[u,u^{-1};\mx{v}]  $ defined by%
\begin{align}
\mathcal{D} &  =-\sum_{|k|\geq1}|k|v_{k}\frac{\partial}{\partial v_{k}}%
-u\frac{\partial}{\partial u}(\mathcal{A}_{+}-\mathcal{A}_{-})\nonumber\\
&  -2\sum_{k=2}^{\infty}\left[  \left(  \sum_{j=1}^{k-1}jv_{j}v_{k-j}\right)
\frac{\partial}{\partial v_{k}}+\left(  \sum_{j=1}^{k-1}jv_{-j}v_{-k+j}%
\right)  \frac{\partial}{\partial v_{-k}}\right]  \nonumber\\
&  -2\sum_{k=1}^{\infty}\left[  v_{k}u\mathcal{A}_{+}\frac{\partial}{\partial
u}\mathcal{M}_{u^{-k}}\mathcal{A}_{+}+v_{-k}u\mathcal{A}_{-}\frac{\partial
}{\partial u}\mathcal{M}_{u^{k}}\mathcal{A}_{-}\right]  .\label{e.deq}%
\end{align}
It is also convenient to define%
\begin{equation}
\mathcal{D}_{N}=\mathcal{D}-\frac{1}{N^{2}}\mathcal{L}.\label{def D}%
\end{equation}

\end{notation}

For the proof of Theorem \ref{t.intertwine.new} it is useful to decompose
$\mathcal{D}$ and $\mathcal{D}_{N}$ as%
\begin{align} 
\mathcal{D} &  =-\mathcal{N}-2\mathcal{Z}-2\mathcal{Y}\label{def DN}\\ \label{d.DN.new}
\mathcal{D}_{N} &  =-\mathcal{N}-2\mathcal{Z}-2\mathcal{Y}-\frac{1}{N^{2}%
}\mathcal{L}=\mathcal{D}-\frac{1}{N^{2}}\mathcal{L}
\end{align}
where $\mathcal{N}$, $\mathcal{Y}$, and $\mathcal{Z}$ are defined as follows.

\begin{definition}
\label{Def LN L} Define the following operators on $\C[u,u^{-1};\mx{v}]$.
\begin{align}
\mathcal{N}_{1} &  =u\frac{\partial}{\partial u}(\mathcal{A}_{+}%
-\mathcal{A}_{-}),\qquad\mathcal{N}_{0}=\sum_{|k|\geq1}|k|v_{k}\frac{\partial
}{\partial v_{k}},\qquad\mathcal{N}=\mathcal{N}_{0}+\mathcal{N}_{1}%
,\label{eq N1}\\
\mathcal{Y} &  =\mathcal{Y}_{+}-\mathcal{Y}_{-}=\sum_{k=1}^{\infty}%
v_{k}u\mathcal{A}_{+}\frac{\partial}{\partial u}\mathcal{M}_{u^{-k}%
}\mathcal{A}_{+}-\sum_{k=-\infty}^{-1}v_{k}u\mathcal{A}_{-}\frac{\partial
}{\partial u}\mathcal{M}_{u^{-k}}\mathcal{A}_{-},\\ \label{eq Z}
\mathcal{Z} &  =\mathcal{Z}_{+}-\mathcal{Z}_{-}=\sum_{k=2}^{\infty}\left(
\sum_{j=1}^{k-1}jv_{j}v_{k-j}\right)  \frac{\partial}{\partial v_{k}}%
-\sum_{k=-\infty}^{-2}\left(  \sum_{j=k+1}^{-1}jv_{j}v_{k-j}\right)
\frac{\partial}{\partial v_{k}}.
\end{align}

\end{definition}

\begin{example}
\label{example Q} The first order pseudodifferential operator $\mathcal{Y}$
appears somewhat mysterious; we illustrate its action here.

\begin{itemize}
\item $\mathcal{Y}$ annihilates $\C[\mx{v}]  $; more
generally, for $P\in\C[u,u^{-1};\mx{v}]  $ and
$Q\in\C[\mx{v}]  $, $\mathcal{Y}(PQ)=\mathcal{Y}%
(P)\cdot Q$. It therefore suffices to understand the action of $\mathcal{Y}$
on $\mathbb{C}[u,u^{-1}]$.

\item $\mathcal{Y}_{-}$ annihilates $\mathbb{C}[u]$ and $\mathcal{Y}_{+}$
annihilates $\mathbb{C}[u^{-1}]$. The reader can calculate that
\begin{align*}
\mathcal{Y}(u^{n})=\mathcal{Y}_{+}(u^{n})  &  =\sum_{k=1}^{n-1}(n-k)v_{k}%
u^{n-k},\qquad n\geq0\\
-\mathcal{Y}(u^{n})=\mathcal{Y}_{-}(u^{n})  &  =\sum_{k=n+1}^{-1}%
(n-k)v_{k}u^{n-k},\qquad n<0.
\end{align*}

\end{itemize}
\end{example}

\begin{notation}
\label{n.trace} For $n\in\mathbb{Z}$ and $Z\in\mathbb{M}_{N}$ let
$W_{n}(Z)=Z^{n}$, $V_{n}(Z)=\mathrm{tr}(Z^{n})$, and $\mathbf{V}(Z)=\left\{
V_{n}(A)\right\}  _{|n|\geq1}$. (Technically we should write $V_{n}^{N}$ for
$V_{n}$ and $W_{n}^{N}$ for $W_{n}$, but we omit this extra index since the
meaning should be clear from the context.) With this notation we have
$P_{N}(U)=P(U;\mathbf{V}(U))$ for $P\in\C[u,u^{-1};\mx{v}]  $.
\end{notation}

\begin{proof}
[Proof of Theorem \ref{t.intertwine.new}]\label{proof of t.intertwine.new}%
Given the notation introduced above our goal is to show that
\begin{equation}
\Delta_{\mathbb{U}_{N}}P_{N}=\left[  \mathcal{D}_{N}P\right]  _{N}=\left[
\left(  -\mathcal{N}-2\mathcal{Z}-2\mathcal{Y}-\frac{1}{N^{2}}\mathcal{L}%
\right)  P\right]  _{N}.\label{intertwining formula repeated}%
\end{equation}
Fix $n\in\mathbb{Z}\setminus\{0\}$, and let $P(u;\mathbf{v})=u^{n}%
q(\mathbf{v})$ where $q\in\C[\mx{v}]  $; thus
$P_{N}=W_{n}\cdot q(\mathbf{V})$. For $X\in\mathfrak{u}_{N}$, by the product
rule we have
\[
\partial_{X}P_{N}=\partial_{X}\left[  W_{n}\cdot q(\mathbf{V})\right]
=\partial_{X}W_{n}\cdot q(\mathbf{V})+W_{n}\cdot\partial_{X}q(\mathbf{V})
\]
and therefore
\begin{align}
\Delta_{\mathbb{U}_{N}}P_{N} &  =\sum_{X\in\beta_{N}}\partial_{X}^{2}%
P_{N}\nonumber\\
&  =\sum_{X\in\beta_{N}}\left[  \partial_{X}^{2}W_{n}\cdot q(\mathbf{V}%
)+2\partial_{X}W_{n}\cdot\partial_{X}q(\mathbf{V})+W_{n}\cdot\partial_{X}%
^{2}q(\mathbf{V})\right]  \nonumber\\
&  =\left(  \Delta_{\mathbb{U}_{N}}W_{n}\right)  \cdot q(\mathbf{V}%
)+2\sum_{X\in\beta_{N}}\partial_{X}W_{n}\cdot\partial_{X}q(\mathbf{V}%
)+W_{n}\cdot\left(  \Delta_{\mathbb{U}_{N}}q(\mathbf{V})\right)
.\label{e.pnn}%
\end{align}

Using (\ref{e.b0b}) and the chain rule, the middle term in (\ref{e.pnn}) can
be written as
\begin{align}
\label{e.pnn2}\sum_{X\in\beta_{N}}\partial_{X}W_{n}\cdot\partial
_{X}q(\mathbf{V})  &  =\sum_{X\in\beta_{N}}\partial_{X}W_{n}\cdot\sum
_{|k|\ge1}\left(  \frac{\partial}{\partial v_{k}}q\right)  (\mathbf{V}%
)\cdot\partial_{X}V_{k}\nonumber\\
&  =\sum_{|k|\ge1}\left(  \sum_{X\in\beta_{N}}\partial_{X}W_{n}\cdot
\partial_{X}V_{k}\right)  \left(  \frac{\partial}{\partial v_{k}}q\right)
(\mathbf{V})\nonumber\\
&  =\sum_{|k|\ge1}\left(  -\frac{nk}{N^{2}}W_{n+k}\right)  \left(
\frac{\partial}{\partial v_{k}}q\right)  (\mathbf{V})\nonumber\\
&  =-\frac{1}{N^{2}}\sum_{|k|\ge1}nkW_{n+k}\left(  \frac{\partial}{\partial
v_{k}}q\right)  (\mathbf{V}).
\end{align}
Notice that $nW_{n+k} = W_{k+1}\cdot nW_{n-1} = W_{k+1}\left[  \frac{\partial
}{\partial u} u^{n}\right]  _{N}$, and so (\ref{e.pnn2}) may be written in the
form
\begin{equation}
\label{e.int1}\sum_{X\in\beta_{N}}\partial_{X}W_{n}\cdot\partial
_{X}q(\mathbf{V}) = -\frac{1}{N^{2}}\left[  \sum_{|k|\ge1} ku^{k+1}%
\frac{\partial^{2}}{\partial u\partial v_{k}} P\right]  _{N}.
\end{equation}

For the last term in (\ref{e.pnn}), we again use the chain and product rules
repeatedly to find
\begin{align}
\partial_{X}^{2}q(\mathbf{V})  &  =\partial_{X}\left(  \sum_{|k|\geq1}\left(
\frac{\partial}{\partial v_{k}}q\right)  (\mathbf{V})\cdot\partial_{X}%
V_{k}\right) \nonumber\label{e.pnn2.5}\\
&  =\sum_{|k|\geq1}\left(  \frac{\partial}{\partial v_{k}}q\right)
(\mathbf{V})\cdot\partial_{X}^{2}V_{k}+\sum_{|j|,|k|\geq1}\left(
\frac{\partial^{2}}{\partial v_{j}\partial v_{k}}q\right)  (\mathbf{V}%
)\cdot(\partial_{X}V_{j})(\partial_{X}V_{k}).
\end{align}
Summing this equation on $X\in\beta_{N}$, (\ref{e.b0c}) shows that the the
second sum in (\ref{e.pnn2.5}) simplifies to
\begin{equation}
\sum_{X\in\beta_{N}}\sum_{|j|,|k|\geq1}\left(  \frac{\partial^{2}}{\partial
v_{j}\partial v_{k}}q\right)  (\mathbf{V})\cdot(\partial_{X}V_{j}%
)(\partial_{X}V_{k})=-\frac{1}{N^{2}}\sum_{|j|,|k|\geq1}jkV_{j+k}\cdot\left(
\frac{\partial^{2}}{\partial v_{j}\partial v_{k}}q\right)  (\mathbf{V}).
\label{e.pnn3}%
\end{equation}
For the first sum in (\ref{e.pnn2.5}), we break up the sum over positive and
negative terms, and use (\ref{e.b2}) and (\ref{e.b2-}) to see that
\begin{align*}
\sum_{X\in\beta_{N}}\sum_{|k|\geq1}\left(  \frac{\partial}{\partial v_{k}%
}q\right)  (\mathbf{V})\cdot\partial_{X}^{2}V_{k}  &  =\sum_{k=1}^{\infty
}\left(  \frac{\partial}{\partial v_{k}}q\right)  (\mathbf{V})\left(
-kV_{k}-2\mathbbm{1}_{k\geq2}\sum_{j=1}^{k-1}jV_{j}V_{k-j}\right) \\
&  \quad+\sum_{k=-\infty}^{-1}\left(  \frac{\partial}{\partial v_{k}}q\right)
(\mathbf{V})\left(  kV_{k}+2\mathbbm{1}_{k\leq-2}\sum_{j=k+1}^{-1}%
jV_{j}V_{k-j}\right)
\end{align*}
which is equal to
\begin{align}
&  -\sum_{|k|\geq1}|k|V_{k}\left(  \frac{\partial}{\partial v_{k}}q\right)
(\mathbf{V})\nonumber\\
&  \qquad-2\sum_{k=2}^{\infty}\left(  \sum_{j=1}^{k-1}jV_{j}V_{k-j}\right)
\left(  \frac{\partial}{\partial v_{k}}q\right)  (\mathbf{V})+2\sum
_{k=-\infty}^{-1}\left(  \sum_{j=k+1}^{-1}jV_{j}V_{k-j}\right)  \left(
\frac{\partial}{\partial v_{k}}q\right)  (\mathbf{V}). \label{e.pnn4}%
\end{align}
Combining (\ref{e.pnn2.5}) -- (\ref{e.pnn4}) we see that the final term in
(\ref{e.pnn}) is
\[
W_{n}\cdot\Delta_{\mathbb{U}_{N}}q(\mathbf{V})=-[\mathcal{N}_{0}%
P]_{N}-2[\mathcal{Z}P]_{N}-\frac{1}{N^{2}}\left[  \sum_{|j|,|k|\geq1}%
jkv_{j+k}\frac{\partial^{2}}{\partial v_{j}\partial v_{k}}P\right]  _{N}%
\]
and combining this with (\ref{e.pnn}) and (\ref{e.int1}) gives
\begin{equation}
\Delta_{\mathbb{U}_{N}}P_{N}=\left(  \Delta_{\mathbb{U}_{N}}W_{n}\right)
\cdot q(\mathbf{V})-\left[  \left(  \mathcal{N}_{0}+2\mathcal{Z}+\frac
{1}{N^{2}}\mathcal{L}\right)  P\right]  _{N}, \label{e.int2}%
\end{equation}
where (\ref{e.int1}) and (\ref{e.pnn3}) are the terms responsible for
$\mathcal{L}$. To address the first term in (\ref{e.int2}), we treat the cases
$n\geq0$ and $n<0$ separately. When $n\geq0$, (\ref{e.b1}) gives
\[
\left(  \Delta_{\mathbb{U}_{N}}W_{n}\right)  \cdot q(\mathbf{V})=-nW_{n}\cdot
q(\mathbf{V})-2\mathbbm{1}_{n\geq2}\sum_{j=1}^{n-1}jW_{j}V_{n-j}%
q(\mathbf{V}).
\]
The first term is $-\left[  u\frac{\partial}{\partial u}u^{n}q(\mathbf{v}%
)\right]  _{N}$, and the second is (reindexing $k=n-j$)
\[
-2\left[  \sum_{k=1}^{n-1}v_{k}u^{n-k}q(\mathbf{v})\right]  _{N}%
=-2[\mathcal{Y}_{+}P]_{N}%
\]
from Example \ref{example Q}. An analogous computation in the case $n<0$,
using (\ref{e.b1-}), shows that in this case
\[
\Delta_{\mathbb{U}_{N}}W_{n}\cdot q(\mathbf{V})=\left[  u\frac{\partial
}{\partial u}P\right]  _{N}+2[\mathcal{Y}_{-}P]_{N}.
\]
Combining these with (\ref{e.int2}) concludes the proof of
(\ref{eq intertwining formula 1}); (\ref{e.int2.new}) follows immediately,
with the help of Corollary \ref{c.Pn-inv}.
\end{proof}

\begin{definition}
\label{def scalars and tracing} The \textbf{tracing map} $\mathcal{T}%
\colon\C[u,u^{-1};\mx{v}]  \rightarrow\mathbb{C}%
\left[  \mathbf{v}\right]  $ is the linear operator given as follows: if
$p\in\C[\mx{v}]$ and $k\in\mathbb{Z}\setminus
\{0\}$, then
\begin{equation}
\mathcal{T}(u^{k}p(\mathbf{v}))=v_{k}p(\mathbf{v}).\label{eq tracing map}%
\end{equation}
Regarding $\C[\mx{v}]$ as a subalgebra of $\C[u,u^{-1};\mx{v}]$, note that an element $P\in\C[u,u^{-1};\mx{v}]$ is in $\C[\mx{v}]  $ if
and only if $\mathcal{T}(P)=P$.
\end{definition}

The following intertwining formula is elementary to verify.

\begin{lemma}
\label{l.tracing} For $P\in\C[u,u^{-1};\mx{v}]  $ and
$N\in\mathbb{N}$,
\begin{equation}
\lbrack\mathcal{T}(P)]_{N}=\mathrm{tr}\circ P_{N}.\label{eq trace intertwiner}%
\end{equation}

\end{lemma}

In order to proceed further it is useful to know that $\C[u,u^{-1};\mx{v}]$ completely decomposes into the finite
dimensional eigenspaces of the operator $\mathcal{N}$. Indeed, the space
$\C[u,u^{-1};\mx{v}]$ (Definition \ref{d.Laurent.trace.new}) is the span of monomials
\[
\C[u,u^{-1};\mx{v}]  =\mathrm{span}_{\mathbb{C}%
}\left\{  u^{k_{0}}v_{1}^{k_{1}}v_{-1}^{k_{-1}}\cdots v_{n}^{k_{n}}%
v_{-n}^{k_{-n}}\colon n\geq0,\;\;k_{0}\in\mathbb{Z},\;\;k_{j}\in
\mathbb{N}\;\text{ for }\;j\in\mathbb{Z}\setminus\{0\}\right\}
\]
where each monomial is an eigenvector of $\mathcal{N}$ as the next example shows.

\begin{example}
\label{example N}The monomial, $P(u;\mathbf{v})=u^{k_{0}}v_{1}^{k_{1}}%
v_{-1}^{k_{-1}}\cdots v_{n}^{k_{n}}v_{-n}^{k_{-n}},$ is an eigenvectors of
$\mathcal{N}$, with
\begin{equation}
\mathcal{N}(P)=\left(  |k_{0}|+\sum_{1\leq|j|\leq n}|j|k_{j}\right)  P.
\label{e.trdeg0}%
\end{equation}
We will define this eigenvalue to be the {\em trace degree} of $P$.
\end{example}

\begin{definition}
\label{def P}
The {\bf trace degree} of a monomial in $\C[u,u^{-1};\mx{v}]  $ is
\begin{equation}
\deg\left(  u^{k_{0}}v_{1}^{k_{1}}v_{-1}^{k_{-1}}\cdots v_{n}^{k_{n}}%
v_{-n}^{k_{-n}}\right)  =|k_{0}|+\sum_{1\leq|j|\leq n}|j|k_{j}%
.\label{eq trace degree}%
\end{equation}
More generally, the trace degree of any element of $\C[u,u^{-1};\mx{v}]$ is the maximum of the trace degrees of its
monomial terms. For $n\geq0$, denote by $\C_n[u,u^{-1};\mx{v}]  \subset\C[u,u^{-1};\mx{v}]  $ the
subspace of polynomials of trace degree $\leq n$:
\begin{equation}
\C_n[u,u^{-1};\mx{v}]  =\{P\in\C[u,u^{-1};\mx{v}]\colon\deg P\leq n\}.\label{eq Pn}%
\end{equation}
Note that $\C_n[u,u^{-1};\mx{v}]  $ is finite
dimensional; indeed, it is contained in $\C[u,u^{-1};v_{\pm1},\ldots,v_{\pm n}]$. Moreover, $\C[u,u^{-1};\mx{v}]
=\bigcup_{n\geq0}\C_n[u,u^{-1};\mx{v}]  $. Define
$\C_n[u;\mx{v}]$, $\C_n[u^{-1};\mx{v}]$, $\C_n[\mx{v}]$,
$\C_{n}[u,u^{-1}]$, $\C_{n}[u]$, and $\C_{n}[u^{-1}]$ similarly.
\end{definition}

\begin{remark}

The trace degree reflects the nature of the variables $v_{\pm1},v_{\pm2},\ldots$ in $\C[\mx{v}]$ as stand-ins
for traces of powers of a matrix variable. Informally, the trace degree of
$P\in\C[u,u^{-1};\mx{v}]  $ is the total degree of
$P_{N}(Z)$, counting all instances of $Z$ inside and outside traces, where the
degree of $Z^{k}$ is defined to be $|k|$.
\end{remark}

\begin{lemma}
[$\mathcal{D}$ and $\mathcal{D}_{N}$ commute with $\mathcal{T}$]%
\label{lemma degree and trace preserving} Let $\mathcal{L},\mathcal{D}%
,\mathcal{D}_{N}\colon\C[u,u^{-1};\mx{v}]
\rightarrow\C[u,u^{-1};\mx{v}]  $ be given as in
Definition \ref{Def LN L} and (\ref{def DN}). The operators $\mathcal{D}%
_{N}$, $\mathcal{D}$, and $\mathcal{L}$ preserve trace degree
(\ref{eq trace degree}), and commute with the tracing map $\mathcal{T}$
(\ref{eq tracing map}).
\end{lemma}

\begin{proof}
\label{proof of lemma degree and trace preserving}Let $\mathcal{N}%
,\mathcal{Y}_{\pm},\mathcal{Z}_{\pm}$ be as in be given as in Definition
\ref{Def LN L}. The reader may readily verify that $\mathcal{N}$,
$\mathcal{Y}_{\pm}$, $\mathcal{Z}_{\pm}$, and $\mathcal{L}$ all preserve trace
degree. What's more, it is elementary to calculate that $[\mathcal{T}%
,\mathcal{N}]=0$, while
\[
\mathcal{Z}_{\pm}\mathcal{T}=\mathcal{T}[\mathcal{Z}_{\pm}+\mathcal{Y}_{\pm
}],\quad\mathcal{Y}_{\pm}\mathcal{T}=0.
\]
Hence, it follows that $\mathcal{D}=-\mathcal{N}-2(\mathcal{Z}+\mathcal{Y}%
)=-\mathcal{N}-2(\mathcal{Z}_{+}+\mathcal{Y}_{+})+2(\mathcal{Z}_{-}%
+\mathcal{Y}_{-})$ commutes with $\mathcal{T}$. Since $\mathcal{D}%
_{N}=\mathcal{D}-\frac{1}{N^{2}}\mathcal{L}$ (cf.\ (\ref{def D})), we are left
only to prove that $[\mathcal{T},\mathcal{L}]=0$. This is also straightforward
to compute; instead, we offer an alternative proof. From
(\ref{eq trace intertwiner}), we see that, for any $P\in\C[u,u^{-1};\mx{v}]$,
\[
\lbrack\mathcal{T}\mathcal{D}_{N}(P)]_{N}=\mathrm{tr}(\Delta_{\mathbb{U}_{N}%
}P_{N})=\Delta_{\mathbb{U}_{N}}\mathrm{tr}(P_{N})=[\mathcal{D}_{N}%
\mathcal{T}(P)]_{N}.
\]
That is: $([\mathcal{T},\mathcal{D}_{N}]P)_{N}\equiv0$. It follows, using the
fact that $[\mathcal{T},\mathcal{D}]=0$, that
\begin{equation}
\left(  \lbrack\mathcal{T},\mathcal{L}]P\right)  _{N}=\left(  [\mathcal{T}%
,N^{2}(\mathcal{D}_{N}-\mathcal{D})]P\right)  _{N}=N^{2}\left(  [\mathcal{T}%
,\mathcal{D}_{N}]P\right)  _{N}\equiv0,\quad\text{for all }%
\;\;N.\label{eq [T,L_N]}%
\end{equation}
Theorem \ref{prop P-->P_N unique} now proves that $[\mathcal{T},\mathcal{L}%
]P=0$. Since this holds true for any $P\in\C[u,u^{-1};\mx{v}]  $, the result is proved.
\end{proof}

We now prove Theorem \ref{cor product rule and homom}.

\begin{proof}
[Proof of Theorem \ref{cor product rule and homom}]%
\label{proof of cor product rule and homom} For convenience, we restate
(\ref{eq partial product 1}): the desired property is
\[
\mathcal{D}(PQ)=(\mathcal{D}P)Q+P(\mathcal{D}Q),\qquad P\in\C[u,u^{-1};\mx{v}],\;Q\in\C[\mx{v}]  .
\]
Recall from Definition \ref{Def LN L} and (\ref{def DN}) that $\mathcal{D}=-\mathcal{N}%
-2\mathcal{Y}-2\mathcal{Z}=-(\mathcal{N}_{0}+2\mathcal{Y})-(\mathcal{N}%
_{1}+2\mathcal{Z})$, where $\mathcal{N}_{1}$ and $\mathcal{Z}$ are first order
differential operators on $\C[\mx{v}]  $, while
$\mathcal{N}_{0}$ and $\mathcal{Y}$ annihilate $\C[\mx{v}]$ and satisfy
\[
\mathcal{N}_{0}(PQ)=(\mathcal{N}_{0}P)Q,\qquad\mathcal{Y}(PQ)=(\mathcal{Y}%
P)Q,\qquad P\in\C[u,u^{-1};\mx{v}],\;Q\in
\C[\mx{v}].
\]
Hence
\[
(\mathcal{N}_{0}+2\mathcal{Y})(PQ)=[(\mathcal{N}_{0}+2\mathcal{Y}%
)P]Q=[(\mathcal{N}_{0}+2\mathcal{Y})P]Q+P[(\mathcal{N}_{0}+2\mathcal{Y})Q].
\]
Since $\mathcal{N}_{1}+2\mathcal{Z}$ satisfies the product rule on
$\C[u,u^{-1};\mx{v}]$ in general, this proves
(\ref{eq partial product 1}); (\ref{eq partial product 2}) follows thence from
the standard power series argument.
\end{proof}

\begin{remark}
\label{r.alt.int} We could alternatively describe the intertwining operator
$\mathcal{D}$ as the unique operator on $\C[u,u^{-1};\mx{v}]  $ which satisfies the partial product rule
(\ref{eq partial product 1}), commutes with the tracing map $\mathcal{T}$, and
satisfies
\[
\mathcal{D}(u^{k})=-|k|u^{k}-2\mathbbm{1}_{k\geq2}\sum_{\ell=1}^{k-1}%
(k-\ell)v_{\ell}u^{k-\ell}+2\mathbbm{1}_{k\leq-2}\sum_{\ell=k+1}^{-1}%
(k-\ell)v_{\ell}u^{k-\ell}.
\]

\end{remark}

The next corollary follows immediately from the first statement of Lemma
\ref{lemma degree and trace preserving}.

\begin{corollary}
\label{c.Pn-inv} For $n,N\in\mathbb{N}$, the finite dimensional subspace
$\C_n[u,u^{-1};\mx{v}]  \subset\C[u,u^{-1};\mx{v}]$ is invariant under $\mathcal{D}_{N}$ and
$\mathcal{D}$. Hence, for $t\in\mathbb{R}$, $e^{\frac{t}{2}\mathcal{D}_{N}}$
and $e^{\frac{t}{2}\mathcal{D}}$ are well-defined operators on $\C[u,u^{-1};\mx{v}]$ that leave $\C_n[u,u^{-1};\mx{v}]$ invariant.
\end{corollary}

This brings us to the proof of Theorem \ref{t.new1.8}.

\begin{proof}
[Proof of Theorem \ref{t.new1.8}]\label{proof of t.new1.8} For convenience, we
restate the desired property (\ref{e.new1.3}): we will show that, for any
$P\in\C[u,u^{-1};\mx{v}]$, $N\in\mathbb{N}$, and
$t>0$, there exists $P_{t}^{N}\in\C[u,u^{-1};\mx{v}]
$ with
\[
\mathbf{B}_{s,t}^{N}P_{N}=[P_{t}^{N}]_{N}.
\]
Indeed, let $\mathcal{D}_{N}$ be as in (\ref{def DN}), and define $P_{t}%
^{N}=e^{\frac{t}{2}\mathcal{D}_{N}}P\in\C[u,u^{-1};\mx{v}]  $. By (\ref{e.int2.new}) of Theorem \ref{t.intertwine.new}%
, we then have $[P_{t}^{N}]_{N}=e^{\frac{t}{2}\Delta_{\mathbb{U}_{N}}}P_{N}$.
Since $[P_{t}^{N}]_{N}$ is a trace polynomial, the entries of $[P_{t}^{N}%
]_{N}(U)$ are (holomorphic) polynomials in the entries of $U$. Thus,
$[P_{t}^{N}]_{N}$ has an analytic continuation to an entire function on
$\mathbb{GL}_{N}$, whose entries are the very same polynomials. It follows
that $[P_{t}^{N}]_{N}$, interpreted as a function on $\mathbb{GL}_{N}$, is the
analytic continuation of $e^{\frac{t}{2}\Delta_{\mathbb{U}_{N}}}P_{N}$, which
is, by Definition \ref{d.new1.5}, equal to $\mathbf{B}_{s,t}^{N}P_{N}$.
\end{proof}

We conclude with the following Corollary to the proof of Theorem
\ref{t.new1.8}, characterizing the range of $\mathbf{B}^{N}_{s,t}$ on trace polynomials.

\begin{corollary}
\label{c.Brange} Let $s,t>0$ with $s>t/2$, and let $N\in\mathbb{N}$. If
$P\in\C[u,u^{-1};\mx{v}]  $, there exists
$Q\in\C[u,u^{-1};\mx{v}]  $ such that $\mathbf{B}%
_{s,t}^{N}Q_{N}=P_{N}$. Thus, $\mathbf{B}_{s,t}^{N}$ maps the space
$\left[\C[u,u^{-1};\mx{v}]\right]_{N}$ of trace polynomials
onto itself.
\end{corollary}

\begin{proof}
Set $Q=e^{-\frac{t}{2}\mathcal{D}_{N}}P$. Then the intertwining formula
(\ref{e.int2.new}), combined with the above discussion, shows that
\[
\mathbf{B}_{s,t}^{N}Q_{N}=[e^{\frac{t}{2}\mathcal{D}_{N}}Q]_{N}=P_{N}%
\]
as claimed.
\end{proof}

\subsection{Intertwining Formulas II}

\label{section Intertwining Formula II}

This section is devoted to proving an intertwining formula for $\GL_N$ (cf.\ Theorem \ref{t.inter2}) which is analogous to the intertwining formula for $\U_N$ in Theorem \ref{t.intertwine.new}. This result is only needed in order to prove concentration of measures on $\GL_N$ (Eq.\ (\ref{e.conc2}) of Theorem \ref{t.concentration}) and hence we do not need as much detailed information about the operators involved. On the other hand, we will now have to consider scalar trace polynomials in both $Z$ and $Z^\ast$, which complicates the notation somewhat.

\begin{notation}
\label{n.omega} For $n\in\N$, let $\EX_n$ denote the set of functions (words) $\varepsilon\colon\{1,\ldots,n\}\rightarrow\left\{\pm1,\pm\ast\right\}$.  For $\varepsilon\in\EX_n$, we denote $|\varepsilon|=n$.  Set $\EX = \bigcup_n \EX_n$.  We define the {\bf word polynomial space} $\mathscr{W}$ as
\[ \mathscr{W} =\mathbb{C}\left[  \left\{v_{\varepsilon}\right\}_{\varepsilon\in\EX}\right]
\]
the space of polynomials in the indeterminates $\left\{ v_{\varepsilon
}\right\}_{\varepsilon\in\EX}$.  Of frequent use will be the words
\begin{equation}
\varepsilon(j,k)  =(  \overset{|j|\text{ times}}{\overbrace{\pm1,\dots,\pm1}},\overset{|k|\text{ times}}{\overbrace{\pm\ast,\dots,\pm\ast}})\in\EX_{j+k},\label{e.emn}
\end{equation}
where we use $+1$ in the first slots if $j>0$ and $-1$ if $j<0$, and similarly we use $+\ast$ in the last slots if $k>0$ and $-\ast$ if $k<0$.
\end{notation}

\begin{notation} For $\varepsilon\in\EX_n$ and $Z\in \GL_N$ we define $Z^\varepsilon = Z^{\varepsilon_1}Z^{\varepsilon_2}\cdots Z^{\varepsilon_n}$, where $Z^{+\ast}\equiv Z^\ast$ and $Z^{-\ast} \equiv (Z^\ast)^{-1} = (Z^{-1})^\ast$. Given $P\in\mathscr{W}$, we let $P_{N}\colon \GL_N\rightarrow\mathbb{C}$ be the function
\[ P_{N}(Z)=P\left(\mx{V}(Z)\right) \]
where
\[ \mx{V}(Z)=\left\{V_{\varepsilon}(Z)\colon \varepsilon\in\EX\right\} \]
and
\[ V_{\varepsilon}(Z)=\tr(Z^\varepsilon) = \tr\left(Z^{\varepsilon_1}Z^{\varepsilon_2}\cdots Z^{\varepsilon_n}\right). \]
\end{notation}
\noindent The notation $\mx{V}$ here collides with Notation \ref{n.trace}, but there should be no confusion as to which is being used.  As in that case, we should technically write $V_\varepsilon = V_\varepsilon^N$ and $\mx{V} = \mx{V}^N$, but we suppress the $N$ throughout.  Also, in terms of Notation \ref{n.trace}, note that $V_{\varepsilon(k,0)}(Z) = \tr(Z^k) = V_k(Z)$, while $V_{\varepsilon(0,k)}(Z) = \tr((Z^{\ast})^k) = V_k(Z^\ast)$.  It is therefore natural to think of $\C[\mx{v}]$ as included in $\mathscr{W}$, in the following way.

\begin{notation} \label{notation iota} We can identify $\C[\mx{v}]$ as a subalgebra of $\mathscr{W}$ in two ways: $\iota,\iota^\ast\colon\C[\mx{v}]\hookrightarrow\mathscr{W}$, with $\iota$ linear and $\iota^\ast$ conjugate linear, are determined by
\begin{equation} \label{eq P inclusion W} \iota(v_k) = v_{\varepsilon(k,0)} \qquad \iota^\ast(v_k) = v_{\varepsilon(0,k)}.
\end{equation}
The inclusions $\iota$ and $\iota^\ast$ intertwine with the evaluation maps as follows: for $Q\in\C[\mx{v}]$,
\begin{equation} \label{eq iota intertwiner} [\iota(Q)]_N(Z) = Q_N(Z) \qquad [\iota^\ast(Q)]_N(Z) = Q_N(Z)^\ast. \end{equation}
\end{notation}

The trace degree on $\C[\mx{v}]$ extends consistently to the larger space $\mathscr{W}$.

\begin{definition} The {\bf trace degree} of a monomial $\prod_{i=1}^{m}v_{\varepsilon_{j}}^{k_{j}}\in\mathscr{W}$ is given by
\[ \deg\left(\prod_{j=1}^{m}v_{\varepsilon_{j}}^{k_{j}}\right)=\sum_{j=1}^{m}|k_{j}||\varepsilon_{j}|, \]
and the trace degree of any element in $\mathscr{W}$ is the highest trace degree of any of its monomial terms.  Since $|\varepsilon(k,0)|=|\varepsilon(0,k)|=k$, we have
\begin{equation} \label{eq iota degree} \deg \iota(Q) = \deg \iota^\ast(Q) = \deg Q \end{equation}
for $Q\in\C[\mx{v}]$.  Note, moreover, that $\mathrm{deg}(RS) = \mathrm{deg}(R) + \mathrm{deg}(S)$ for $R,S\in\mathscr{W}$ not identically $0$.  Finally, for $n\in\N$ we set
\[
\mathscr{W}_{n}=\left\{  P\in\mathscr{W}:\deg\left(  P\right)  \leq n\right\}.
\]
Note that $\mathscr{W}_n$ is finite dimensional, $\mathscr{W}_n\subset \C[\{v_\varepsilon\}_{|\varepsilon|\le n}]$, and $\mathscr{W} = \bigcup_n \mathscr{W}_n$.
\end{definition}

We now proceed to describe the action of $A_{s,t}^N$ on functions on $\U_N$ or $\GL_N$ of the form $R_N$ for some $R\in\mathscr{W}$; recall from (\ref{eq Ast def}) that
\[ A_{s,t}^{N} \equiv \left(  s-\frac{t}{2}\right)  \sum_{X\in\beta_N}\partial_{X}^{2}+\frac{t}{2}\sum_{X\in\beta_N}\partial_{iX}^{2}, \]
where $\beta_{N}$ is an orthonormal basis for $\u_N$.

\begin{theorem}
\label{t.inter1} Fix $s,t\in\R$.    There are collections $\left\{Q^{s,t}_\varepsilon\colon \varepsilon\in\EX \right\}$ and $\left\{R^{s,t}_{\varepsilon,\delta}\colon \varepsilon,\delta\in\EX\right\}$ in $\mathscr{W}$ with the following properties:
\begin{enumerate}
\item[(1)] for each $\varepsilon \in\EX$, $Q_{\varepsilon}^{s,t}$ is a certain finite sum of monomials of trace degree $|\varepsilon|$ such that
\begin{equation} \label{eq q formula} A_{s,t}^{N}V_{\varepsilon}=[Q^{s,t}_{\varepsilon}]_N = Q^{s,t}_{\varepsilon}(\mx{V}), \end{equation}
\item[(2)] for $\varepsilon,\delta\in\EX$, $R^{s,t}_{\varepsilon,\delta}$ is a certain finite sum of monomials of trace degree
$|\varepsilon|+|\delta|$ such that
\begin{equation} \label{eq r formula}
\left(s-\frac{t}{2}\right)  \sum_{X\in\beta_N}\left(  \partial_{X}V_{\varepsilon}\right)  \left(\partial_{X}V_{\delta}\right)  +\frac{t}{2}\sum_{X\in \beta_N}\left(  \partial_{iX}V_{\varepsilon}\right) 
\left(\partial_{iX}V_{\delta}\right)  = \frac{1}{N^2}[R_{\varepsilon,\delta}^{s,t}]_N = \frac{1}{N^{2}}R^{s,t}_{\varepsilon,\delta}(\mx{V}). \end{equation}
\end{enumerate}
\end{theorem}
\noindent Please note that the polynomials $Q_\varepsilon^{s,t}$ and $R_{\varepsilon,\delta}^{s,t}$ {\em do not depend on $N$}.  The $1/N^2$ in (\ref{eq r formula}) comes from the magic formula (\ref{e.m3a}), as we will see in the proof.

\begin{proof}
Fix $\varepsilon\in\EX$, and let $n=|\varepsilon|$.  Let $\beta_N$ denote an orthonormal basis for $\u_N$, and let $\beta_+=\beta_N$ while $\beta_-=i\beta_N$.  For any $\xi \in \u_N\oplus i\u_N = \gl_N = \M_N$ and $Z\in\GL_N$, we make the following conventions (for this proof only):
\begin{equation} \label{e.intIIp1} (Z\xi)^1 \equiv Z\xi, \qquad (Z\xi)^{-1} \equiv -\xi Z^{-1}, \qquad (Z\xi)^{\ast} \equiv \xi^\ast Z^\ast, \qquad (Z\xi)^{-\ast}\equiv -Z^\ast\xi^\ast. \end{equation}
Note that, for $\xi\in\beta_\pm$, $\xi^\ast = \mp \xi$.  In the proof to follow, we do not precisely track all of the signs, and so $\pm$ denotes a sign that may be different in different terms and on different sides of an equation.  Thus, we have
\[ \left(  \partial_{\xi}V_{\varepsilon}\right)(Z) =\sum_{j=1}^{n}\tr\left(Z^{\varepsilon_{1}}Z^{\varepsilon_{2}}\dots\left(Z\xi\right)^{\varepsilon_{j}}\dots Z^{\varepsilon_n}\right) \]
and so
\begin{align} \label{eq intertwine II 1}
\left(  \partial_{\xi}^{2}V_{\varepsilon}\right)(Z)   &
=\sum_{j=1}^{n}\tr\left(  Z^{\varepsilon_{1}}Z^{\varepsilon_{2}}\dots\left(Z\xi^2\right)^{\varepsilon_{j}}\dots
Z^{\varepsilon_n}\right) \\
\label{eq intertwine II 2} &  +2\sum_{1\le j< k\le n}\tr\left(Z^{\varepsilon_{1}}Z^{\varepsilon_{2}}\dots\left(  Z\xi\right)^{\varepsilon_{j}}\dots\left(Z\xi\right)^{\varepsilon_{k}}\dots Z^{\varepsilon_{n}}\right).
\end{align}
We must now sum over $\xi\in\beta_\pm$.   It follows from magic formula (\ref{e.m1a}) and convention (\ref{e.intIIp1}) that each term in (\ref{eq intertwine II 1}) simplifies to
\[ \sum_{\xi\in\beta_\pm}\tr\left(  Z^{\varepsilon_{1}}Z^{\varepsilon_{2}}\dots\left(Z\xi^2\right)^{\varepsilon_{j}}\dots Z^{\varepsilon_n}\right)  =\pm \tr\left(  Z^{\varepsilon_{1}}Z^{\varepsilon_{2}}\dots Z^{\varepsilon_{j}}\dots Z^{\varepsilon_{n}}\right)  =\pm V_\varepsilon(Z). \]
To be clear: the $\pm$ on the right varies with $j$ and whether the sum is over $\beta_+$ or $\beta_-$.  Summing each of these terms over $1\le j\le n$ shows that (\ref{eq intertwine II 1}) summed over $\beta_\pm$ is
\begin{equation} \label{e.intIIp2.5} \sum_{\xi\in\beta_\pm} \sum_{j=1}^n \tr\left(  Z^{\varepsilon_{1}}Z^{\varepsilon_{2}}\dots\left(Z\xi^2\right)^{\varepsilon_{j}}\dots Z^{\varepsilon_n}\right) =  n_\pm(\varepsilon) V_\varepsilon(Z) \end{equation}
for some $n_\pm(\varepsilon)\in\Z$ with $|n_\pm(\varepsilon)|\le |\varepsilon|$.  For the terms in (\ref{eq intertwine II 2}), applying (\ref{e.intIIp1}) shows that
\begin{equation} \label{e.intIIp2} \tr\left(Z^{\varepsilon_{1}}Z^{\varepsilon_{2}}\dots\left(Z\xi\right)^{\varepsilon_{j}}\dots\left(Z\xi\right)  ^{\varepsilon_{k}}\dots Z^{\varepsilon_{n}}\right) = \pm \tr(Z^{\varepsilon^0_{j,k}}\xi Z^{\varepsilon^1_{j,k}}\xi Z^{\varepsilon^2_{j,k}}) \end{equation}
where $\{\varepsilon^\ell_{j,k}\}_{\ell=0,1,2}$ are certain substrings of $\varepsilon$, whose concatenation is all of $\varepsilon$: $\varepsilon_{j,k}^0\varepsilon_{j,k}^1\varepsilon_{j,k}^2 = \varepsilon$.  Applying magic formula (\ref{e.m2a}) to (\ref{e.intIIp2}) gives
\[ \sum_{\xi\in\beta_\pm} \tr(Z^{\varepsilon^0_{j,k}}\xi Z^{\varepsilon^1_{j,k}}\xi Z^{\varepsilon^2_{j,k}}) = \pm\tr(Z^{\varepsilon^0_{j,k}}Z^{\varepsilon^2_{j,k}})\tr(Z^{\varepsilon^1_{j,k}}) = \pm\tr(Z^{\varepsilon_{j,k}})\tr(Z^{\varepsilon^1_{j,k}}) \]
where $\varepsilon_{j,k} = \varepsilon^0_{j,k}\varepsilon^2_{j,k}$.  Note that $|\varepsilon_{j,k}| + |\varepsilon^1_{j,k}| = |\varepsilon|$.  Hence, the sum in (\ref{eq intertwine II 2}) summed over $\beta_{\pm}$ is equal to
\begin{equation} \label{e.intIIp3} \sum_{1\le j<k\le n}\pm\tr(Z^{\varepsilon_{j,k}})\tr(Z^{\varepsilon^1_{j,k}}) =  \sum_{1\le j<k\le n} \pm V_{\varepsilon_{j,k}}(Z)V_{\varepsilon^1_{j,k}}(Z).
\end{equation}
Hence, if we define
\begin{equation} \label{e.intIIp4} Q^{\pm}_\varepsilon = n_\pm(\varepsilon)v_\varepsilon + 2\sum_{1\le j<k\le |\varepsilon|} \pm v_{\varepsilon_{j,k}}v_{\varepsilon^1_{j,k}}, \end{equation}
which have homogeneous trace degree $|\varepsilon|$, then (\ref{eq intertwine II 1}) -- (\ref{e.intIIp3}) show that
\[ Q^{s,t}_\varepsilon = \left(s-\frac{t}{2}\right)Q^+_\varepsilon +\frac{t}{2}Q^-_\varepsilon \]
satisfies (\ref{eq q formula}), proving item (1) of the theorem.

For item (2), fix $\delta\in\EX$ and let $m=|\delta|$.   We calculate for each $\xi\in \M_N$
\[ (\del_\xi V_\varepsilon)(Z)(\del_\xi V_\delta)(Z) = \sum_{j=1}^n\sum_{k=1}^m \tr(Z^{\varepsilon_1}Z^{\varepsilon_2}\cdots (Z\xi)^{\varepsilon_j}\cdots Z^{\varepsilon_n}) \cdot \tr(Z^{\delta_1}Z^{\delta_2}\cdots (Z\xi)^{\delta_k}\cdots Z^{\delta_m}), \]
again making use of convention (\ref{e.intIIp1}).  Using the cyclic property of the trace, we can write the terms in this sum in the form
\[ \pm\tr(\xi Z^{\varepsilon^{(j)}})\tr(\xi Z^{\delta^{(k)}}) \]
where $\varepsilon^{(j)}$ is a certain cyclic permutation of $\varepsilon$, and $\delta^{(k)}$ is a certain cyclic permutation of $\delta$.  Summing over $\xi\in\beta_\pm$ and using magic formula (\ref{e.m4a}), we then have
\begin{equation} \label{e.intIIp5} \sum_{\xi\in\beta_\pm} (\del_\xi V_\varepsilon)(Z)(\del_\xi V_\delta)(Z) = \frac{1}{N^2}\sum_{j=1}^n\sum_{k=1}^m \pm \tr(Z^{\varepsilon^{(j)}}Z^{\delta^{(k)}}) = \frac{1}{N^2}\sum_{j=1}^n\sum_{k=1}^m \pm V_{\varepsilon^{(j)}\delta^{(k)}}(Z). \end{equation}
Since $\varepsilon^{(j)}\delta^{(k)}$ has length $|\varepsilon|+|\delta|$, the $\mathscr{W}$ elements
\begin{equation} \label{eq intertwiner II 5} 
R^{\pm}_{\varepsilon,\delta} = \sum_{j=1}^{|\varepsilon|}\sum_{k=1}^{|\delta|} \pm v_{\varepsilon^{(j)}\delta^{(k)}} \end{equation}
have homogeneous trace degree $|\varepsilon|+|\delta|$, and (\ref{e.intIIp5}) therefore shows that
\begin{equation} \label{e.intIIp6} r^{s,t}_{\varepsilon,\delta} = \left(s-\frac{t}{2}\right)R^+_{\varepsilon,\delta} + \frac{t}{2}R^-_{\varepsilon,\delta} \end{equation}
satisfies (\ref{eq r formula}), proving item (2) of the theorem.
\end{proof}

\begin{theorem}[Intertwining Formula II]\label{t.inter2}Fix $s,t\in\R$. Let $\left\{Q^{s,t}_{\varepsilon}\colon \varepsilon
\in\EX\right\}$ and $\left\{R^{s,t}_{\varepsilon,\delta}\colon \varepsilon,\delta\in\EX\right\}$ be the polynomials from Theorem \ref{t.inter1}, and define
\begin{equation}
\widetilde{\mathcal{D}}_{s,t}=\frac12\sum_{\varepsilon\in\EX}Q^{s,t}_{\varepsilon}%
\frac{\partial}{\partial v_{\varepsilon}}\qquad\text{ and }\qquad
\widetilde{\mathcal{L}}_{s,t}=\frac12\sum_{\varepsilon,\delta\in\EX}R^{s,t}_{\varepsilon
,\delta}\frac{\partial^{2}}{\partial v_{\varepsilon}\partial v_{\delta}%
}\label{e.zl}%
\end{equation}
which are first and second order differential operators on $\mathscr{W}$
which preserve trace degree. Then, for all $N\in\mathbb{N}$ and $P\in\mathscr{W}$,
\begin{equation}
\frac{1}{2}A_{s,t}^{N}P_{N}=\left[\widetilde{\mathcal{D}}_{s,t}P+\frac{1}{N^{2}%
}\widetilde{\mathcal{L}}_{s,t}P\right]_{N}.\label{e.apn}%
\end{equation}
\end{theorem}

\begin{remark} Definition \ref{definition Ast and must} of $A_{s,t}^N$ is stated for $s,t>0$ and $s>t/2$; it is only in this regime that the operator $A_{s,t}^N$ is negative-definite and the tools of heat kernel analysis apply.  The operator itself is well-defined-for any $s,t\in\R$, however, and it will be convenient to utilize this in some of what follows.
\end{remark}

\begin{proof} By the chain rule, if $\xi \in \M_N$ then
\begin{align*}
\partial_{\xi}^{2}P_{N} &  = \sum_{\varepsilon\in\EX} \partial_{\xi}\left[  \left(  \frac{\partial
P}{\partial v_{\varepsilon}}\right)(\mx{V})\cdot  \partial_{\xi}V_{\varepsilon}\right]  \\
&  = \sum_{\varepsilon\in\EX}\left(\frac{\partial P}{\partial v_{\varepsilon}}\right)(\mx{V})\cdot  \partial_{\xi}^{2}V_{\varepsilon}+ \sum_{\varepsilon,\delta\in\EX} \left(\frac{\partial^{2}P}{\partial v_{\varepsilon}\partial v_{\delta}
}\right)(\mx{V})\cdot \left(\partial_{\xi}V_{\varepsilon}\right)\left(\partial_{\xi}V_{\delta}\right)
\end{align*}
from which it follows that
\begin{align*}
A_{s,t}^{N}P_{N} &  = \sum_{\varepsilon\in\EX} \left(  \frac{\partial P}{\partial v_{\varepsilon}}\right)(\mx{V})\cdot  A_{s,t}^{N}V_{\varepsilon}\\
&  +\sum_{\varepsilon,\delta\in\EX} \left(\frac{\partial^{2}P}{\partial v_{\varepsilon}\partial v_{\delta}}\right)  (\mx{V})  \left[  \left(  s-\frac{t}{2}\right)
\sum_{\xi\in\beta}\partial_{\xi}V_{\varepsilon}\cdot\partial_{\xi}V_{\delta}+\frac{t}{2}%
\sum_{\xi\in i\beta} \left(  \partial_{\xi}V_{\varepsilon}\right)  \left(  \partial
_{\xi}V_{\delta}\right)  \right].
\end{align*}
Combining this equation with the results of Theorem \ref{t.inter1} completes
the proof.
\end{proof}

\noindent We record one further intertwining formula that will be useful in the proofs of Theorems \ref{Main Theorem} and  \ref{t.concentration}.

\begin{lemma}
\label{l.sesqui}There exists a sequilinear form (conjugate linear in the second variable)
\[ \mathcal{B}:\C[u,u^{-1};\mx{v}]\times\C[u,u^{-1};\mx{v}]\rightarrow \mathscr{W} \]
such that, for all $P,Q\in\C[u,u^{-1};\mx{v}]$, we have $\deg\left(\mathcal{B}(P,Q)\right)=\deg(P)+\deg(Q)$
and
\[ [\mathcal{B}(P,Q)]_{N}(Z)=\tr[P_{N}(Z)Q_{N}(Z)^\ast] \qquad \text{ for all }Z\in \GL_N. \]
\end{lemma}

\begin{proof}
By sesquilinearity, it suffices to define $\mathcal{B}$ on $P,Q\in\C[u,u^{-1};\mx{v}]$ of the form $P(u;\mx{v}) = u^kp(\mx{v})$ and $Q(u;\mx{v}) = u^\ell q(\mx{v})$ for $k,\ell\in\Z$ and $p,q\in \C[\mx{v}]$.  We compute, for $Z\in \GL_N$, that
\begin{align*} \tr[P_N(Z)Q_N(Z)^\ast] = \tr[Z^k p_N(Z) Z^{\ast \ell} q_N(Z)^\ast] &= \tr(Z^kZ^{\ast \ell}) p_N(Z)q_N(Z)^\ast \\
&= [v_{\varepsilon(k,\ell)}]_N(Z) [\iota(p)]_N(Z)[\iota^\ast(q)]_N(Z) \end{align*}
by (\ref{eq iota intertwiner}), where $\varepsilon(k,\ell)$ is defined in (\ref{e.emn}).  Thus, we take $\mathscr{B}\colon\C[u,u^{-1};\mx{v}]\times\C[u,u^{-1};\mx{v}]\to\mathscr{W}$ to be the unique sesquilinear form such that, for $p,q\in\C[\mx{v}]$,
\[ \mathcal{B}(u^kp,u^\ell q) = v_{\varepsilon(k,\ell)} \iota(p)\iota^\ast(q). \]
This is trace degree additive by (\ref{eq iota degree}).  This concludes the proof.
\end{proof}

\section{Limit Theorems \label{Section Limit Theorems}}

In this section, we prove that the heat kernel measures $\rho^N_{s}$ on $\U_N$
and $\mu^N_{s,t}$ on $\GL_N$ each concentrate all their mass in such a way that the space of trace polynomials $\left[\C[u,u^{-1};\mx{v}]\right]_N$ collapses onto the space of Laurent polynomials $[\C[u,u^{-1}]]_N$ as $N\to\infty$.  To motivate this, consider the scalar-valued case: if $Q\in \C[\mx{v}]$, then Theorem \ref{t.intertwine.new} shows
that
\begin{equation} \label{e.limintro1} e^{\frac{s}{2}\Delta_{\U_N}}(Q_N) = \left[e^{\frac{s}{2}(\mathcal{D}-\frac{1}{N^2}\mathcal{L})}Q\right]_N = \left[e^{\frac{s}{2}\mathcal{D}}Q\right]_N + O\left(\frac{1}{N^2}\right), \end{equation}
where the second equality will be made precise in Lemma \ref{l.findim} below.  Evaluating (\ref{e.limintro1}) at $I_N$ and using (\ref{eq rho t def}) shows that
\begin{equation} \label{e.4.2*} \E_{\rho^N_s}(Q_N) = \left(e^{\frac{s}{2}\Delta_{\U_N}}Q_N\right)(I_N) = \left(e^{\frac{s}{2}\mathcal{D}}Q\right)(\mx{1}) + O\left(\frac{1}{N^2}\right), \end{equation}
where $Q(\mx{1}) = \left.Q(\mx{v})\right|_{\mx{v}=\mx{1}}$ is the evaluation of $Q$ at all variables $v_k=1$.  Theorem \ref{cor product rule and homom} show that $e^{\frac{s}{2}\mathcal{D}}$ is an algebra homomorphism on $\C[\mx{v}]$, and so
\begin{equation} \label{e.limintro2} \left[e^{\frac{s}{2}\mathcal{D}}Q^2\right]_N = \left(\left[e^{\frac{s}{2}\mathcal{D}}Q\right]_N\right)^2. \end{equation}
If $Q$ has real coefficients, then $Q^2 = |Q|^2$, and so (\ref{e.4.2*}) applied to $Q^2$ and (\ref{e.limintro2}) evaluated at $\mx{1}$ show that
\begin{equation*}
\mathrm{Var}_{\rho^N_{s}}(Q_{N}) = \int_{\U_N}
|Q_{N}(U)|^{2}\,\rho^N_{s}(dU) - \left| \int_{\U_N} Q_{N}(U)\,\rho^N
_{s}(dU)\right| ^{2} = O\left( \frac{1}{N^{2}}\right).
\end{equation*}
Thus, the random variables $Q_N$ concentrate on their limit mean (which is $\pi_s Q$ by Theorem \ref{t.new1.9}), summably fast.  Section \ref{s.concentration} fleshes out this argument in the general case (where $Q$ need not have real coefficients, and is more generally in $\C[u,u^{-1};\mx{v}]$).  Sections \ref{section Main Theorem} and \ref{section Limit Norms} then use these ideas to prove Theorems \ref{Main Theorem} and \ref{thm inverse is inverse}.

\subsection{Concentration of Measures\label{s.concentration}}

We begin with an abstract result that will be the gist of all our concentration of measure theorems.

\begin{lemma} \label{l.findim} Let $V$ be a finite dimensional normed $\C$-space and supposed that $D$ and $L$ are two operators on $V$. Then there exists a constant $C=C(D,L,\|\cdot\|_V)<\infty$ such that
\begin{equation} \label{e.findim1}\left\Vert e^{D+\epsilon L}-e^{D}\right\Vert _{\mathrm{End}(V)} \leq C\left\vert \epsilon\right\vert \text{ for all }\left\vert \epsilon
\right\vert \leq1, \end{equation}
where $\|\cdot\|_{\mathrm{End}(V)}$ is the operator norm on $V$. It follows that, if $\varphi\in V^\ast$ is a linear functional, then
\begin{equation} \label{e.findim2} \left|\varphi(e^{D+\epsilon L}x)-\varphi(e^Dx)\right| \le C\|\varphi\|_{V^\ast} \|x\|_V|\epsilon|, \quad x\in V, \; |\epsilon|\le 1, \end{equation}
where $\|\cdot\|_{V^\ast}$ is the dual norm on $V^\ast$.
\end{lemma}

\begin{proof}
Using the well known differential of the exponential map (see for 
example \cite[Theorem 1.5.3, p.\ 23]{Duistermaat2000}, \cite[Theorem 3.5, p.\ 70]{HallLieBook}, or \cite[Lemma 3.4, p.\ 35]{Sattinger1993}),
\begin{align*}
\frac{d}{ds}e^{D+sL}  &  =e^{D+sL}\int_{0}^{1}e^{-t\left(  D+sL\right)
}Le^{t\left(  D+sL\right)  }dt\\
&  =\int_{0}^{1}e^{\left(  1-t\right)  \left(  D+sL\right)  }Le^{t\left(
D+sL\right)  }dt,
\end{align*}
we may write
\[
e^{D+\epsilon L}-e^{D}=\int_{0}^{\epsilon}\frac{d}{ds}e^{D+sL}%
ds=\int_{0}^{\epsilon}\left[  \int_{0}^{1}e^{\left(  1-t\right)  \left(
D+sL\right)  }Le^{t\left(  D+sL\right)  }dt\right]  ds.
\]
Crude bounds now show
\[
\left\Vert e^{D+\epsilon L}-e^{D}\right\Vert _{\mathrm{End}\left(
V\right)  }\leq\int_{0}^{|\epsilon|}\left[  \int_{0}^{1}\left\Vert e^{\left(
1-t\right)  \left(  D+sL\right)  }Le^{t\left(  D+sL\right)  }\right\Vert
_{\mathrm{End}\left(  V\right)  }dt\right]  ds\leq C(D,L,\|\cdot\|_V)|\epsilon|, \]
proving (\ref{e.findim1}); (\ref{e.findim2}) follows immediately.
\end{proof}

Theorem \ref{t.new1.9} and Lemma \ref{l.findim} now allow us to give a useful alternate characterization of the evaluations maps $\pi_s$.


\begin{lemma} \label{t.intertwine.1.5.new} For $P\in\C[u,u^{-1};\mx{v}]$ and $s\in\R$, the evaluation map $\pi_s$ can be written in the form
\begin{equation} \label{e.pi_s.new2} (\pi_s P)(u) = \left(e^{-\frac{s}{2}(\mathcal{N}_0+2\mathcal{Z})}P\right)(u;\mx{1}) \end{equation}
where, for $Q\in\C[u,u^{-1};\mx{v}]$, $\left.Q(u;\mx{1}) = Q(u;\mx{v})\right|_{v_k=1, k\ne 0}$.
\end{lemma}

\begin{proof} \label{proof of t.intertwine.1.5.new} First, note from Definition \ref{Def LN L} that, for $p\in\C[\mx{v}]$, $\mathcal{N}_1 p = \mathcal{Y}p = 0$; thus, from (\ref{def DN}), we have
\[ \left.\D\right|_{\C[\mx{v}]} = \left.\left(-\mathcal{N}_0-2\mathcal{Z}\right)\right|_{\C[\mx{v}]}. \]
If $P(u;\mx{v}) = \sum_k u^kp_k(\mx{v})$ with $p_k\in\C[\mx{v}]$, then $\big((-\mathcal{N}_0-2\mathcal{Z})P\big)(u;\mx{v}) = \sum_k u^k (\D p_k)(\mx{v})$; hence, to prove (\ref{e.pi_s.new2}), it suffices to show that
\begin{equation} \label{e.pi_s.p.proof} \pi_s(p) = \left(e^{\frac{s}{2}\D}p\right)(\mx{1}), \qquad p\in\C[\mx{v}]. \end{equation}
By Theorem \ref{cor product rule and homom}, $e^{\frac{s}{2}\D}$ is a homomorphism of $\C[\mx{v}]$.  Hence, to prove \ref{e.pi_s.p.proof}, it suffices to show that
\begin{equation} \label{e.pi_s.p.proof2} \left(e^{\frac{s}{2}\D}v_k\right)(\mx{1}) = \pi_s(v_k) = \nu_k(s), \qquad k\in\Z\setminus\{0\}. \end{equation}
Theorem \ref{t.new1.9}, together with (\ref{eq rho t def}) and (\ref{e.int2.new}), shows that
\begin{equation} \label{e.pi_s.p.proof3} \nu_k(s) = \lim_{N\to\infty} \left(e^{\frac{s}{2}\Delta_{\U_N}}\tr[(\cdot)^k]\right)(I_N) = \lim_{N\to\infty}\left(e^{\frac{s}{2}\D_N}v_k\right)(\mx{1}). \end{equation}
On the other hand, $\varphi(p) = p(\mx{1})$ is a linear functional on $\C[\mx{v}]$, and $v_k\in\C_k[\mx{v}]$ which is finite-dimensional.  Since $\D_N = \D-\frac{1}{N^2}\L$ and both $\D$ and $\L$ leave $\C_k[\mx{v}]$ invariant, Lemma \ref{l.findim} shows that
\begin{equation} \label{e.pi_s.p.proof4} \left|\left(e^{\frac{s}{2}\D_N}v_k\right)(\mx{1})- \left(e^{\frac{s}{2}\D}v_k\right)(\mx{1})\right| = O\left(\frac{1}{N^2}\right). \end{equation}
Equations (\ref{e.pi_s.p.proof3}) and (\ref{e.pi_s.p.proof4}) imply (\ref{e.pi_s.p.proof2}), concluding the proof.
\end{proof}

The next lemma relates $\widetilde{\mathcal{D}}_{s,t}$ to the evaluation map $\pi_{s-t}$, which will lead to the proof of Theorem \ref{t.new1.10}.  Recall the inclusion maps $\iota,\iota^\ast\colon \C[\mx{v}]\hookrightarrow \mathscr{W}$ of Notation \ref{notation iota}.

\begin{lemma} \label{lemma Linfty nu(s)} Let $s,t>0$ with $s>t/2$.  Let $\widetilde{\mathcal{D}}_{s,t}$ be given as in (\ref{e.zl}).  Then, for any $Q\in\C[\mx{v}]$,
\begin{equation} \label{eq Zst nu(s-t)} [e^{\widetilde{\mathcal{D}}_{s,t}}\iota(Q)](\mx{1}) = \pi_{s-t}Q. \end{equation} 
\end{lemma}

\begin{proof} If $f\colon \GL_N\to \M_N$ is holomorphic, then $\partial_{iX}f= i\partial_{X}f$ for all $X\in\u_N$, which then implies
\[
\left.A_{s,t}^{N}f\right|_{\U_N}=\left(  s-\frac{t}{2}\right)  \sum_{X\in\beta_{N}}\partial
_{X}^{2}f-\frac{t}{2}\sum_{X\in\beta_{N}}\partial_{X}^{2}f=\left(  s-t\right)
\Delta_{\U_N}f. \]
Since the scalar trace polynomial $Q_N$ is holomorphic, it follows that
\begin{equation} \label{eq A to Delta} e^{\frac12 A^N_{s,t}}Q_N = e^{\frac12(s-t)\Delta_{\U_N}}Q_N. \end{equation}
(Note: when $s<t$ the expression $e^{\frac12(s-t)\Delta_{\U_N}}$ is not meaningful in general, but makes perfect sense as a power series when applied to a polynomial function such as $Q_N$.)
Using intertwining formulas (\ref{eq iota intertwiner}) and (\ref{e.apn}) on the left-hand-side of (\ref{eq A to Delta}) and intertwining formula  (\ref{eq intertwining formula 1}) on the right-hand-side, we have
\[ \left[e^{\widetilde{\mathcal{D}}_{s,t} + \frac{1}{N^2}\widetilde{\mathcal{L}}_{s,t}} \iota(Q)\right]_N = e^{\frac12 A^N_{s,t}}Q_N = e^{\frac12(s-t)\Delta_{\U_N}}Q_N = \left[e^{\frac12(s-t)\mathcal{D}_N}Q\right]_N,  \]
and evaluating both sides at $I_N$ and using $\mathcal{D}_N = \mathcal{D}-\frac{1}{N^2}\mathcal{L}$, we have
\begin{equation} \label{e.int2wice} \left(e^{\widetilde{\mathcal{D}}_{s,t} + \frac{1}{N^2}\widetilde{\mathcal{L}}_{s,t}} \iota(Q)\right)(\mx{1}) =  \left(e^{\frac12(s-t)(\mathcal{D}-\frac{1}{N^2}\mathcal{L})}Q\right)(\mx{1}). \end{equation}
Let $n=\mathrm{deg}(Q)$.  Using the linear functional $\varphi(R) = R(\mx{1})$ on the finite-dimensional spaces $\C_n[\mx{v}]$ and $\mathscr{W}_n$, Lemma \ref{l.findim} allows us to take the limit as $N\to\infty$ in (\ref{e.int2wice}), yielding
\begin{equation}
\label{eq Delta to A new 3} \left(e^{\widetilde{\mathcal{D}}_{s,t}} \iota(Q)\right)(\mx{1})
 = \left(e^{\frac12(s-t)\mathcal{D}}Q\right)(\mx{1}). \end{equation}
Finally, since $Q\in \C[\mx{v}]$, Lemma \ref{t.intertwine.1.5.new} shows that the right-hand-side of (\ref{eq Delta to A new 3}) is $\pi_{s-t}Q$.  This concludes the proof. \end{proof}

\begin{remark} A similar calculation shows that $ \left(e^{\widetilde{\mathcal{D}}_{s,t}} \iota^\ast(Q)\right)(\mx{1}) = \pi_{s-t}(\overline{Q})$. \end{remark}

Theorem \ref{t.new1.10} was really proved in the above proof.

\begin{proof}[Proof of Theorem \ref{t.new1.10}] \label{proof of t.new1.10} From (\ref{eq must def}) and Remark \ref{r.Robinson}, together with intertwining formulas (\ref{eq iota intertwiner}) and (\ref{e.apn}), we have
\[ \int_{\GL_N} \tr(Z^k)\,\mu^N_{s,t}(dZ) = \left(e^{\widetilde{\mathcal{D}}_{s,t} + \frac{1}{N^2}\widetilde{\mathcal{L}}_{s,t}} \iota(v_k)\right)(\mx{1}).   \]
The result now follows as in the justification of (\ref{eq Delta to A new 3}) from (\ref{e.int2wice}). \end{proof}

We now proceed with the proof of Theorem \ref{t.concentration}.

\begin{proof}[Proof of Theorem \ref{t.concentration}] \label{proof of t.concentration}  We begin with the proof of (\ref{e.conc2}).  By the triangle inequality, it suffices to prove the theorem for polynomials of the form $P(u;\mx{v}) = u^kQ(\mx{v})$ for $k\in\Z$ and $Q\in\C[\mx{v}]$.  Therefore
\[ P(u;\mx{v})-\pi_{s-t} P(u;\mx{v}) = u^k[Q(\mx{v})-\pi_{s-t} Q] = u^k R_{s-t}(\mx{v})  \]
where $R_{s-t} = Q-\pi_{s-t} Q$.  Note that $\pi_{s-t}R_{s-t}=0$.  Now, for $Z\in \GL_N$,
\begin{align} \nonumber \|P_N(Z)-(\pi_sP)_N(Z)\|_{\M_N}^2 &= \tr(Z^k [R_{s-t}]_N(Z) [R_{s-t}]_N(Z)^\ast Z^{\ast k} ) \\
&= \tr(Z^kZ^{\ast k})[R_{s-t}]_N(Z)[R_{s-t}]_N(Z)^\ast. \label{eq proof trace concentration 0} \end{align}
Thus
\begin{equation} \label{eq proof trace concentration 0.5} \|[P]_N(Z)-(\pi_{s-t}P)_N(Z)\|_{\M_N}^2 = [v_{\varepsilon(k,k)} \iota(R_{s-t})\iota^\ast(R_{s-t})]_N(Z) \end{equation}
where, in the case $k=0$, we interpret $v_{\varepsilon(0,0)}=1$.  We calculate the $L^2(\mu_{s,t}^N)$-norm of the function $[P-\pi_{s-t}P]_N = [u^kR_{s-t}]_N$ using (\ref{eq rho t def}).  Thus, using the intertwining formula (\ref{e.apn}) together with (\ref{eq proof trace concentration 0.5}), we have
\begin{align} \nonumber \|P_N-(\pi_s P)_N\|_{L^2(\mu^N_{s,t})}^2 &= e^{\frac12A^N_{s,t}}\left(\|P_N-(\pi_{s-t} P)_N\|_{\M_N}^2\right)(I_N) \\
 \label{eq proof trace concentration 1} &= \left(e^{\widetilde{\mathcal{D}}_{s,t}+\frac{1}{N^2}\widetilde{\mathcal{L}}_{s,t}}\left(v_{\varepsilon(k,k)}\iota(R_{s-t})\iota^\ast(R_{s-t})\right)\right)(\mx{1}).
\end{align}
Now, let $n=\deg Q = \deg R_{s-t}$.  Using the linear functional $\varphi(R) = R(\mx{1})$ on $\mathscr{W}_{2n}$, Lemma  \ref{l.findim} then yields
\begin{align}\nonumber &\left|\left(e^{\widetilde{\mathcal{D}}_{s,t}+\frac{1}{N^2}\widetilde{\mathcal{L}}_{s,t}}\left(v_{\varepsilon(k,k)}\iota(R_{s-t})\iota^\ast(R_{s-t})\right)\right)(\mx{1})-\left(e^{\widetilde{\mathcal{D}}_{s,t}}\left(v_{\varepsilon(k,k)}\iota(R_{s-t})\iota^\ast(R_{s-t})\right)\right)(\mx{1})\right| \\ \label{eq proof trace concentration 2}
& \hspace{3.5in} =O\left(\frac{1}{N^2}\right). \end{align}
But, since $\widetilde{\mathcal{D}}_{s,t}$ is a first-order differential operator acting on $\mathscr{W}_{2n}$,  $e^{\widetilde{\mathcal{D}}_{s,t}}$ is an algebra homomorphism, and we have
\begin{equation} e^{\widetilde{\mathcal{D}}_{s,t}}\left(v_{\varepsilon(k,k)}\iota(R_{s-t})\iota^\ast(R_{s-t})\right) = e^{\widetilde{\mathcal{D}}_{s,t}}v_{\varepsilon(k,k)}\cdot e^{\widetilde{\mathcal{D}}_{s,t}}\iota(R_{s-t})\cdot  e^{\widetilde{\mathcal{D}}_{s,t}}\iota^\ast(R_{s-t}) = 0  \label{e.piR=0} \end{equation}
since  $e^{\widetilde{\mathcal{D}}_{s,t}}\iota(R_{s-t}) = \pi_{s-t}R_{s-t} =0$ by Lemma \ref{lemma Linfty nu(s)}. Thus, (\ref{eq proof trace concentration 1}) -- (\ref{e.piR=0}) prove (\ref{e.conc2}).

Note that $\frac{s}{2}\Delta_{\U_N} = \frac12A^N_{s,0}$; thus taking $t=0$ in (\ref{eq proof trace concentration 2}) and restricting the function to $\U_N$ also proves (\ref{e.conc1}), concluding the proof.
\end{proof}

\subsection{Proof of Main Limit Theorem \ref{Main Theorem}
\label{section Main Theorem}}

\begin{proof} [Proof of Theorem \ref{Main Theorem}] \label{proof of Main Theorem} We define $\G_{s,t}$ and $\H_{s,t}$ by (\ref{e.SBdef.new}); evidently, these are linear maps on $\C[u,u^{-1}]$. Let $f\in\C[u,u^{-1}]$; then by the
intertwining formula (\ref{eq intertwining formula 1}),
\[ e^{\frac{t}{2}\Delta_{\U_N}}f_{N}=[e^{\frac{t}{2}\mathcal{D}_N}f]_{N}, \]
where $\mathcal{D}_N$ is defined in (\ref{def DN}).

The function on the right is a trace polynomial function of $U\in \U_N$ (with
no $U^{\ast}$s), and therefore its analytic continuation to $\GL_N$
is given by \emph{the same} trace polynomial function in $Z\in \GL_N$. Thus
\[
\lbrack\mathbf{B}_{s,t}^{N}f_{N}](Z)=[e^{\frac{t}{2}\mathcal{D}_N}%
f]_{N}(Z),\qquad Z\in \GL_N.
\]
Hence
\[ \Vert\mathbf{B}_{s,t}^{N}f_{N}-[\G_{s,t}f]_{N}\Vert_{L^2(\mu_{s,t}^N)}=\Vert\lbrack e^{\frac
{t}{2}\mathcal{D}_N}f]_{N}-[\pi_{s-t}\circ e^{\frac{t}{2}\mathcal{D}}f]_{N}\Vert_{L^2(\mu_{s,t}^N)}.
\]
By the triangle inequality, the
last quantity is
\begin{equation}
\leq\Vert\lbrack e^{\frac{t}{2}\mathcal{D}_N}f]_{N}-[e^{\frac{t}%
{2}\mathcal{D}}f]_{N}\Vert_{L^{2}(\mu_{s,t}^{N})}+\Vert\lbrack
e^{\frac{t}{2}\mathcal{D}}f]_{N}-[\pi_{s-t}\circ e^{\frac{t}%
{2}\mathcal{D}}f]_{N}\Vert_{L^{2}(\mu_{s,t}^{N})}%
.\label{eq SB limit proof 0}%
\end{equation}
The second term in (\ref{eq SB limit proof 0}) is $O(1/N)$ by
(\ref{e.conc2}) (Theorem \ref{t.concentration}). Thus, to complete the (existence) proof of
(\ref{eq SB limit 1}), it suffices to show that
\begin{equation}
\Vert\lbrack e^{\frac{t}{2}\mathcal{D}_N}f]_{N}-[e^{\frac{t}{2}%
\mathcal{D}}f]_{N}\Vert^2_{L^{2}(\mu_{s,t}^{N})}=O\left(  \frac
{1}{N^{2}}\right) \label{eq SB limit proof 0.1}%
\end{equation}
for each $f\in\C[u,u^{-1}]$. Let $n= \deg f$, let $\mathcal{B}$ be the
sesquilinear form in Lemma \ref{l.sesqui}, and let $R^{(N)} = e^{\frac{t}{2}\mathcal{D}_N}f-e^{\frac{t}{2}\mathcal{D}}f$. Then by (\ref{eq must def}) and (\ref{e.apn}), the left side of (\ref{eq SB limit proof 0.1}) is given by
\begin{equation}
\Vert\lbrack R^{\left(  N\right)  }]_{N}\Vert_{L^{2}(\mu_{s,t}^{N})}
^{2}=e^{\frac{1}{2}A_{s,t}^{N}}\left(\|[R^{(N)}]_N\|_{\M_N}^2\right)
= \left(e^{\widetilde{\mathcal{D}}_{s,t}+\frac{1}{N^2}
\widetilde{\mathcal{L}}_{s,t}}\mathcal{B}(R^{(N)},R^{(N)})\right)(\mx{1}).\label{e.brq}%
\end{equation}
Using the linear functional $\varphi(P) = P(\mx{1})$ on $\mathscr{W}_{2n}$ and any norm $\|\cdot\|_{\mathscr{W}_{2n}}$, Lemma \ref{l.findim} ensures there is a constant $C$ (depending on $n,s,t$ but {\em not} on $N$) such that
\begin{equation} \label{e.phiW1} \left| \left(e^{\widetilde{\mathcal{D}}_{s,t}+\frac{1}{N^2}
\widetilde{\mathcal{L}}_{s,t}}\mathcal{B}(R^{(N)},R^{(N)})\right)(\mx{1}) -  \left(e^{\widetilde{\mathcal{D}}_{s,t}}\mathcal{B}(R^{(N)},R^{(N)})\right)(\mx{1})\right| \le \frac{C}{N^2}\|\mathcal{B}(R^{(N)},R^{(N)})\|_{\mathscr{W}_{2n}}. \end{equation}
Let $\psi(P) = \left(e^{\widetilde{\mathcal{D}}_{s,t}} P\right)(\mx{1})$, another linear functional on $\mathscr{W}_{2n}$; then
\[ \left|\left(e^{\widetilde{\mathcal{D}}_{s,t}}\mathcal{B}(R^{(N)},R^{(N)})\right)(\mx{1})\right| \le \|\psi\|_{2n}^\ast \|\mathcal{B}(R^{(N)},R^{(N)})\|_{\mathscr{W}_{2n}}. \]
This, in conjunction with (\ref{e.brq}) and (\ref{e.phiW1}), shows that
\begin{equation} \label{eq Main Theorem 1} \|[R^{(N)}]_N\|^2_{L^2(\mu_{s,t}^N)} \le \left(\|\psi\|_{2n}^\ast +\frac{C}{N^2}\right)\|\mathcal{B}(R^{(N)},R^{(N)})\|_{\mathscr{W}_{2n}}. \end{equation}
Since $\mathcal{B}:\C_n[u,u^{-1};\mx{v}]\times\C_n[u,u^{-1};\mx{v}]\rightarrow\mathscr{W}_{2n}$ is sesquilinear with finite dimensional domain and range, it is bounded with any choice of norms; in particular, given any norm $\|\cdot\|_{\C_n[u,u^{-1};\mx{v}]}$ on $\C_n[u,u^{-1};\mx{v}]$, there is a constant $C'$ (depending on $n$ but not $N$) so that
\[ \|\mathcal{B}(P,Q)\|_{\mathscr{W}_{2n}} \le C'\|P\|_{\C_n[u,u^{-1};\mx{v}]}\|Q\|_{\C_n[u,u^{-1};\mx{v}]}  \qquad\text{for all }P,Q\in\C_n[u,u^{-1};\mx{v}].  \]
Together with (\ref{eq Main Theorem 1}), this yields
\begin{equation} \label{eq Main Theorem 2}  \|[R^{(N)}]_N\|^2_{L^2(\mu_{s,t}^N)} \le C'\left(\|\psi\|_{2n}^\ast +\frac{C}{N^2}\right)\|R^{(N)}\|_{\C_n[u,u^{-1};\mx{v}]}^2. \end{equation}
Finally, Lemma \ref{l.findim} gives
\[ \|R^{(N)}\|_{\C_n[u,u^{-1};\mx{v}]} = \|e^{\frac{t}{2}[\mathcal{D}-\frac{1}{N^2}\mathcal{L}]}f-e^{\frac{t}{2}\mathcal{D}}f \|_{\C_n[u,u^{-1};\mx{v}]} = O\left(\frac{1}{N^2}\right) \]
which proves (\ref{eq SB limit proof 0.1}).  (In fact it shows this term is $O(1/N^4)$; however, since the square of the second term in (\ref{eq SB limit proof 0}) is $O(1/N^2)$, this faster convergence doesn't improve matters.)

The proof of (\ref{eq SB limit 2}) is similar: the restriction of $(\mx{B}_{s,t}^N)^{-1}f_N$ to $\U_N$ is simply $e^{-\frac{t}{2}\Delta_{\U_N}}f_N$, and a similar triangle inequality argument now using (\ref{e.conc1}) shows that it suffices to prove
\begin{equation}
\Vert\lbrack e^{-\frac{t}{2}\mathcal{D}_N}f]_{N}-[e^{-\frac{t}{2}%
\mathcal{D}}f]_{N}\Vert^2_{L^{2}(\rho_s^{N})}=O\left(  \frac{1}{N^{2}}\right) \label{eq SB limit proof 0.2}.
\end{equation}
The argument now proceeds identically to above, by redefining $R^{(N)}$ with the substitution $t\mapsto -t$, and taking all norms with the substitution $(s,t)\mapsto(s,0)$ in all formulas from (\ref{e.brq}) onward.

Thus, we have shown that, with $g_{s,t}$ and $h_{s,t}$ defined as in (\ref{e.SBdef.new}), (\ref{eq SB limit 1}) and (\ref{eq SB limit 2}) hold.  We reserve the proof of uniqueness until Corollary \ref{cor.SBunique} below.

\end{proof}

\subsection{Limit Norms and the Proof of Theorem \ref{thm inverse is inverse} \label{section Limit Norms}}

We begin by proving that the transforms $\G_{s,t}$ and $\H_{s,t}$ are invertible on $\C[u,u^{-1}]$.  (This will be subsumed by Theorem \ref{thm inverse is inverse}, but it will be useful to have this fact in the proof.)

\begin{lemma} \label{lemma B H 1-1} $\G_{s,t}$ and $\H_{s,t}$ are invertible operators on $\C_n[u,u^{-1}]$ for each $n>0$, and hence on $\C[u,u^{-1}]$.
\end{lemma}

\begin{proof} Consider $e^{\pm \frac{t}{2}\mathcal{D}}$ restricted to $\C_n[u,u^{-1};\mx{v}]$.  Expanding as power-series, a straightforward induction using the forms of the composite operators $\mathcal{N}$, $\mathcal{Z}$, and $\mathcal{Y}$ shows that there exist $q_k^{\pm t}\in\C[\mx{v}]$ with
\begin{align*} e^{\pm \frac{t}{2}\mathcal{D}} u^n &= e^{\mp \frac{n}{2} t}u^n + \sum_{k=0}^{n-1} q_k^{\pm t}(\mx{v})u^k, \\
 e^{\pm \frac{t}{2}\mathcal{D}} u^{-n} &= e^{\pm \frac{n}{2} t}u^{-n} + \sum_{k=-n+1}^{0} q_k^{\pm t}(\mx{v})u^k.
\end{align*}
This shows that $e^{\pm\frac{t}{2}\mathcal{D}}$ preserves $\C_n[u]$ and $\C_n[u^{-1}]$.
Incorporating the evaluation maps $\pi_s$ or $\pi_{s-t}$, we find that
\[ \G_{s,t}(u^{\pm n}), \H_{s,t}(u^{\pm n}) \in e^{\pm\frac{n}{2}t}u^{\pm n} + \C_{n-1}[u,u^{-1}] \]
Consider, then, the standard basis $\{1,u^1,\ldots,u^n\}$ of $\C_n[u]$; it follows that, in this basis, $\G_{s,t}|_{\C_n[u]}$ and $\H_{s,t}|_{\C_n[u]}$ are upper-triangular, with diagonal entries $e^{\mp \frac{k}{2}t}$ for $0\le k\le n$.  Thus the restrictions of $\G_{s,t}$ and $\H_{s,t}$ to $\C_n[u]$ are invertible.  A similar argument shows the invertibility on $\C_n[u^{-1}]$, thus yielding the result on $\C_n[u,u^{-1}]$. Since $\C[u,u^{-1}] = \bigcup_n \C_n[u,u^{-1}]$, the proof is complete.  \end{proof}

We now introduce two seminorms on $\C[u,u^{-1};\mx{v}]$.

\begin{definition} \label{def seminorms} Let $s,t>0$ with $s>t/2$.  For each $N$, define the seminorms $\|\cdot\|_{s,N}$ and $\|\cdot\|^{s,t,N}$ on $\C[u,u^{-1};\mx{v}]$ by
\begin{align} \|P\|_{s,N} &= \|P_N\|_{L^2(\U_N,\rho^N_s;\M_N)}  \label{eq norm s N} \\
\|P\|^{s,t,N} &= \|P_N\|_{L^2(\GL_N,\mu^N_{s,t};\M_N)}. \label{eq norm s t N}
\end{align}
\end{definition}
\noindent In fact, for any $n>0$ and sufficiently large $N$, seminorms (\ref{eq norm s N}) and (\ref{eq norm s t N}) are actually {\em norms} when restricted to $\C_n[u,u^{-1};\mx{v}]$.  Indeed, if $\|P\|_{s,N}=0$ then $P_N = 0$ in $L^2(\U_N,\rho^N_s;\M_N)$, and since $P_N$ is a smooth function and $\rho_s^N$ has a strictly positive density, this means $P_N$ is identically $0$.  By Proposition \ref{prop P-->P_N unique}, when $N$ is sufficiently large (relative to $n$) it follows that $P=0$.

For $P\in\C[u,u^{-1};\mx{v}]$, define
\begin{align} \|P\|_{s} &= \lim_{N\to\infty} \|P\|_{s,N}  \label{eq norm s} \\
\|P\|^{s,t} &= \lim_{N\to\infty} \|P\|^{s,t,N}  \label{eq norm s t}.
\end{align}
These are also seminorms on $\C[u,u^{-1};\mx{v}]$, but they are {\em not} norms on all of $\C[u,u^{-1};\mx{v}]$, or even on $\C_n[u,u^{-1};\mx{v}]$ for any $n>1$.  However, restricted to $\C[u,u^{-1}]$, they are in fact norms.  To prove this, we look to the measure $\nu_s$ described following Theorem \ref{t.new1.9}: the law of the free unitary Brownian motion at time $s>0$.  The measure $\nu_s$ is the weak limit of $\nu_s^N$ of (\ref{eq nu_s^N}) (which exists by the L\'evy continuity theorem).  In \cite[Proposition 10]{Biane1997b}, it is shown that $\nu_s$ is absolutely continuous with respect to Lebesgue measure on $\U$, with a continuous density that is strictly positive in a neighborhood of $1\in\U$; we will need this result (in particular the fact that $\mathrm{supp}(\nu_s)$ is not a finite set) in the following.

\begin{lemma} \label{lemma limit norms} The seminorms (\ref{eq norm s}) and (\ref{eq norm s t}) are norms on $\C[u,u^{-1}]$. \end{lemma}

\begin{proof} We begin with norm (\ref{eq norm s}). Identify the Laurent polynomial $P\in\C[u,u^{-1}]$ as a trigonometric polynomial function $P_1$ on the unit circle $\U$.  Then (\ref{eq L^2 nu_s^N}) shows that
\[ \|P\|_{s,N} = \|P_N\|_{L^2(\U_N,\rho^N_s;\M_N)}  = \|P_1\|_{L^2(\U,\nu_s^N)}. \]
Thus, since $\nu^N_s \rightharpoonup \nu_s$,
\begin{equation} \label{eq norm s nu_s} \|P\|_s = \lim_{N\to\infty} \|P_1\|_{L^2(\U,\nu_s^N)} = \|P_1\|_{L^2(\U,\nu_s)}. \end{equation}
Since the support of $\nu_s$ is infinite, (\ref{eq norm s nu_s}) shows that seminorm (\ref{eq norm s}) is indeed a norm on $\C[u,u^{-1}]$.

For seminorm (\ref{eq norm s t}), we will utilize the isometry property of the finite dimensional Segal--Bargmann transform $\mx{B}_{s,t}^N$.  Fix $Q\in\C[u,u^{-1}]$, and let $\deg Q = n$. By Lemma \ref{lemma B H 1-1}, there is a unique Laurent polynomial $P\in\C_n[u,u^{-1}]$ so that $\G_{s,t}(P)=Q$.  Thus
\[ \|Q\|^{s,t} = \| \G_{s,t}P\|^{s,t} = \lim_{N\to\infty} \| \G_{s,t}P\|^{s,t,N}. \]
By Theorem \ref{Main Theorem} and (\ref{eq norm s t N}) we have
\[ \lim_{N\to\infty}  \| \G_{s,t}P\|^{s,t,N} = \lim_{N\to\infty}  \|[\mx{B}^N_{s,t}P]_N\|_{L^2(\GL_N,\mu^N_{s,t};\M_N)} \]
and by the isometry property of the Segal--Bargmann transform, we therefore have
\[ \|Q\|^{s,t} = \lim_{N\to\infty} \|P_N\|_{L^2(\U_N),\rho^N_s;\M_N)} = \|P\|_s. \]
Thus, if $\|Q\|^{s,t}=0$ then $\|P\|_s=0$, so $Q=\G_{s,t}(0) = 0$.  This concludes the proof.
\end{proof}

\begin{remark} \label{remark Biane's norm} Eq.\ (\ref{eq norm s nu_s}) shows that norm (\ref{eq norm s}) is just an $L^2$-norm, with respect to a well-understood measure.  Norm (\ref{eq norm s t}) is, at present, much more mysterious.  In \cite{Biane1997b}, a great deal of work is spent trying to understand this norm in the case $s=t$.  It can, in that case, be identified as the norm of a certain reproducing kernel Hilbert space, built out of holomorphic functions on a bounded region $\Sigma_t\subset\C^\ast$ which has few obvious symmetries, and which becomes non-simply-connected when $t\ge 4$.  Understanding the norm (\ref{eq norm s t}) in general is a goal for future research of the present authors.
\end{remark}

\begin{corollary} \label{cor.SBunique} For $s,t>0$ with $s>t/2$ and $f\in\C[u,u^{-1}]$, the only Laurent polynomials  $g_{s,t}$ and $h_{s,t}$ satisfying (\ref{eq SB limit 1}) and (\ref{eq SB limit 2}) are $g_{s,t} = \G_{s,t}f$ and $h_{s,t}=\H_{s,t}f$ as defined in (\ref{e.SBdef.new}).
\end{corollary}

\begin{proof} Suppose that $g_{s,t},g_{s,t}'\in\C[u,u^{-1}]$ both satisfy
\[ \| \mx{B}^N_{s,t} f_N - [g_{s,t}]_N \|^2_{L^2(\GL_N,\mu_{s,t}^N;\M_N)} = O\left(\frac{1}{N^2}\right) = \| \mx{B}^N_{s,t} f_N - [g_{s,t}']_N \|^2_{L^2(\GL_N,\mu_{s,t}^N;\M_N)}. \]
Then, by the triangle inequality, it follows that $\|g_{s,t}-g_{s,t}'\|_{L^2(\GL_N,\mu^N_{s,t};\M_N)} = O(1/N^2)$.  Taking limits as $N\to\infty$, it follows that $\|g_{s,t}-g_{s,t}'\|^{s,t} = 0$, and it follows from Lemma \ref{lemma limit norms} that $g_{s,t}=g_{s,t}'$.  A similar argument shows uniqueness of $h_{s,t}$.  The result now follows from the proof of Theorem \ref{Main Theorem} on page \pageref{proof of Main Theorem}.
\end{proof}

This leads us to the proof of Theorem \ref{thm inverse is inverse}.

\begin{proof}[Proof of Theorem \ref{thm inverse is inverse}] \label{proof of thm inverse is inverse} Fix $P\in\C[u,u^{-1}]$, and consider the Laurent polynomial $\G_{s,t}\H_{s,t}P \in \C[u,u^{-1}]$.  By definition
\begin{align} \nonumber \|\G_{s,t}\H_{s,t}P-P\|^{s,t} &= \lim_{N\to\infty} \| \G_{s,t}\H_{s,t}P-P \|^{s,t,N} \\
&= \lim_{N\to\infty} \|[\G_{s,t}\H_{s,t} P]_N - P_N\|_{L^2(\mu^N_{s,t})}. \label{proof inverse 1} \end{align}
The triangle inequality yields
\begin{align*} &\|[\G_{s,t}\H_{s,t} P]_N-P_N\|_{L^2(\mu^N_{s,t})} \\
&\qquad\qquad \le \|[\G_{s,t}\H_{s,t} P]_N - \mx{B}^N_{s,t}[\H_{s,t}P]_N\|_{L^2(\mu^N_{s,t})} + \|\mx{B}_{s,t}^N[\H_{s,t}P]_N - P_N\|_{L^2(\mu^N_{s,t})}.  \end{align*}
Applying (\ref{eq SB limit 1}) with $f=\H_{s,t}P$ shows that the first term is $O(1/N)$.  For the second term, we use the isometry property of the Segal--Bargmann transform. The trace polynomial $P_N$ is in the range of $\mx{B}^N_{s,t}$, by Corollary \ref{c.Brange}, and so
\begin{align*}  \|\mx{B}_{s,t}^N[\H_{s,t}P]_N - P_N\|_{L^2(\mu^N_{s,t})} &=  \|\mx{B}_{s,t}^N\left([\H_{s,t}P]_N - (\mx{B}_{s,t}^N)^{-1}P_N\right)\|_{L^2(\mu^N_{s,t})} \\
&= \|[\H_{s,t} P]_N - (\mx{B}_{s,t}^N)^{-1} P_N\|_{L^2(\rho^N_s)} = O\left(\frac{1}{N}\right), \end{align*}
by (\ref{eq SB limit 2}).  Hence, the quantity in the limit on the right-hand-side of (\ref{proof inverse 1}) is $O(1/N)$, so its limit is $0$.  We therefore have $\|\G_{s,t}\H_{s,t}P-P\|^{s,t}=0$.  Lemma \ref{lemma limit norms} shows that $\|\cdot\|^{s,t}$ is a norm on $\C[u,u^{-1}]$, and so it follows that $\G_{s,t}\H_{s,t}P-P=0$.  Hence, since $\G_{s,t}$ and $\H_{s,t}$ are known to be invertible (Lemma \ref{lemma B H 1-1}), it follows that $\H_{s,t} = \G_{s,t}^{-1}$ as desired. \end{proof}

\section{The Free Unitary Segal--Bargmann Transform \label{section Free SB Transform}}

In this final section, we identify the limit Segal--Bargmann transform $\G_{s,t}$, which has been constructed as a linear operator on the space $\C[u,u^{-1}]$ of single-variable Laurent polynomials.  We will characterize the {\em Biane polynomials} for $\G_{s,t}$:
\begin{equation} \label{eq Biane polynomials} p^{s,t}_k = \H_{s,t}((\,\cdot\,)^k) = \pi_s\circ e^{-\frac{t}{2}\mathcal{D}}(\,\cdot\,)^k, \qquad k\in\Z \end{equation}
defined so that
\[ \G_{s,t}(p^{s,t}_k)(z) = z^k \]
when $s,t>0$ and $s>t/2$.  We call them Biane polynomials since, as we will prove, they match the polynomials that Biane introduced in \cite[Lemma 18]{Biane1997b} to characterize the free Hall transform $\mathscr{G}^t$, in the special case $s=t$.  There is classical motivation to understand these polynomials.  Consider the classical Segal--Bargmann transform $S_t$ acting on $L^2(\R,\gamma_t^1)$.  Since polynomials are dense in this Gaussian $L^2$-space, $S_t$ is completely determined by the polynomials $H_k(t,\cdot)$ satisfying $S_t(H_k(t,\cdot))(z) = z^k$.  In this case, since the measure $\gamma^2_{t/2}$ is rotationally-invariant, the monomials $z\mapsto z^k$ are orthogonal, and since $S^1_t$ is an isometry, it follows that $H_k(t,\cdot)$ are the orthogonal polynomials of the Gaussian measure $\gamma^1_t$: the {\em Hermite polynomials} (of variance $t$).  Hence, the Biane polynomials are the unitary version of the Hermite polynomials.  We will determine the generating function $\Pi$ of these polynomials; cf.\ (\ref{eq formula for gen fn}).  In the case $s=t$, this precisely matches the generating function in \cite[Lemma 18]{Biane1997b}, modulo a small correction; in this way, we verify that our limit Segal--Bargmann transform {\em is} the aforementioned free unitary Segal--Bargmann transform $\mathscr{G}^t$.

Before proceeding, we make an observation.  It is immediate from the form of the operators $\mathcal{N}$, $\mathcal{Z}$, and $\mathcal{Y}$ in Definition \ref{Def LN L} that $\mathcal{D} = -\mathcal{N}-2\mathcal{Z}-2\mathcal{Y}$ satisfies
\[ \mathcal{D}(u^{-k}) = \left(\mathcal{D}(\,\cdot\,)^k\right)(u^{-1}), \qquad k\in\Z. \]
Expanding $e^{-\frac{t}{2}\mathcal{D}}$ as a power series shows that the semigroup also commutes with the reciprocal map, and applying the algebra homomorphism $\pi_s$ then shows that
\begin{equation} \label{eq pk p-k} p_{-k}^{s,t}(u) = p_k^{s,t}(u^{-1}), \qquad k\in\Z. \end{equation}
Note also that $\mathcal{D}$ preserves the subspaces $\C[u;\mx{v}]$ and $\C[u^{-1};\mx{v}]$, and hence $p_k^{s,t}(u)$ is a polynomial in $u$ for $k\ge 0$, while $p_k^{s,t}(u)$ is a polynomial in $u^{-1}$ for $k<0$. Hence, since $p^{s,t}_0=1$, it will suffice to identify $p^{s,t}_k$ only for $k\ge 1$.

\subsection{Biane Polynomials and Differential Recursion \label{section Recursion}}

It will be convenient to look at the related family of ``unevaluated'' polynomials.

\begin{definition} \label{def bk ck} For $t\in\R$ and $k\in\N$, define $B_k^t\in\C[u,u^{-1};\mx{v}]$ and $C_k^t\in \C[\mx{v}]$ by
\begin{equation} \label{eq B C} B_k^t(u;\mx{v}) = e^{-\frac{k}{2}t} e^{-\frac{t}{2}\mathcal{D}}u^k \qquad \text{and} \qquad C^t_k(\mx{v}) = \mathcal{T}(B^t_k)(\mx{v}), \end{equation}
where $\mathcal{T}$ is the tracing map of (\ref{eq tracing map}).  For $s\in\R$, define $b_k(s,t,\,\cdot\,)\in\C[u,u^{-1}]$ and $c_k(s,t)\in\C$ by
\begin{equation} \label{eq b c} b_k(s,t,u) = \pi_s(B^t_k)(u) \qquad \text{and} \qquad c_k(s,t) = \pi_s(C^t_k). \end{equation}
Note, by (\ref{eq Biane polynomials}) and the linearity of $\pi_s$, that
\begin{equation}  \label{eq b vs p}  b_k(s,t,u) = e^{-\frac{k}{2}t}p^{s,t}_k(u). \end{equation}
\end{definition}
\noindent It is useful to note the following alternative expression for $c_k(s,t)$.  From (\ref{eq B C}),
\begin{equation} \label{eq C alt 1} C^t_k(\mx{v}) = e^{-\frac{k}{2}t} \mathcal{T}(e^{-\frac{t}{2}\mathcal{D}} u^k) =  e^{-\frac{k}{2}t} e^{-\frac{t}{2}\mathcal{D}} v_k \end{equation}
since, by Lemma \ref{lemma degree and trace preserving}, $\mathcal{T}$ commutes with $\mathcal{D}$.  Thus, from Theorem \ref{t.intertwine.1.5.new}, we have
\begin{equation} \label{eq c alt 0} c_k(s,t) = e^{-\frac{k}{2}t}\pi_s(e^{-\frac{t}{2}\mathcal{D}}v_k) = e^{-\frac{k}{2}t}\left.\left(e^{\frac12(s-t)\mathcal{D}}v_k\right)\right|_{\mx{v}=\mx{1}} = e^{-\frac{k}{2}t}\nu_k(s-t). \end{equation}


The main computational tool that will lead to the identification of the Biane polynomials $p^{s,t}_k$ is the following recursion.

\begin{proposition} \label{prop Recursions} Let $s,t\in\R$, $u\in\C$, and $k\ge 1$.  Let $c_k(s,t)$ and $b_k(s,t,u)$ be given as in Definition \ref{def bk ck}.  Then
\begin{equation} \label{eq ck recursion} c_k(s,t) = \nu_k(s) + \sum_{m=1}^{k-1} \int_0^t mc_{k-m}(s,\tau)c_m(s,\tau)\,d\tau, \qquad k\ge 2 \end{equation}
with $c_1(s,t)=\nu_1(s)$; and
\begin{equation} \label{eq bk recursion} b_k(s,t,u) = u^k+\sum_{m=1}^{k-1}\int_0^t mc_{k-m}(s,\tau)b_m(s,\tau,u)\,d\tau, \qquad k\ge 2 \end{equation}
with $b_1(s,t,u) = u$.
\end{proposition}

\begin{proof} First note that $B^0_k(u;\mx{v}) = u^k$ and $C^0_k(\mx{v}) = v_k$ by definition, and thus $b_k(s,0,u) = \pi_s(u^k) = u^k$, while $c_k(s,0) = \pi_s(v_k) = \nu_k(s)$.  For $k=1$, we have
\[ B^t_1(u) = e^{-\frac{t}{2}}e^{-\frac{t}{2}\mathcal{D}} u = u \]
because $\mathcal{D} u = -u$.  For $k\ge 2$,
\[ \frac{d}{dt} B^t_k = \frac{d}{dt} e^{-\frac{t}{2}(k+\mathcal{D})}u^k = -\frac12e^{-\frac{t}{2}(k+\mathcal{D})}(k+\mathcal{D})u^k. \]
Recall (\ref{def D}) that $\mathcal{D} = -\mathcal{N}-2\mathcal{Z}-2\mathcal{Y}$.  Eq.\ (\ref{eq N1}) shows that $\mathcal{N}(u^k) = ku^k$; (\ref{eq Z}) shows that $\mathcal{Z}$ annihilates $u^k$; and Example \ref{example Q} works out that $\mathcal{Y}(u^k) = \sum_{j=1}^{k-1}(k-j)v_j u^{k-j} = \sum_{m=1}^{k-1} mu^mv_{k-m}$.  Thus
\[ (k+\mathcal{D})(u^k) = ku^k  -ku^k - 2\sum_{m=1}^{k-1} mu^mv_{k-m} = -2\sum_{m=1}^{k-1}mu^mv_{k-m}. \]
Hence
\begin{equation} \label{eq recursion proof 1} \frac{d}{dt}B^t_k = e^{-\frac{k}{2}t}e^{-\frac{t}{2}\mathcal{D}}\left(\sum_{m=1}^{k-1}mu^mv_{k-m}\right)  = e^{-\frac{k}{2}t}\sum_{m=1}^{k-1} m e^{-\frac{t}{2}\mathcal{D}}(u^m v_{k-m}). \end{equation}
We now use the partial homomorphism property of (\ref{eq partial product 2}) at time $-t$, which yields\begin{equation} \label{eq recursion proof 1.5} e^{-\frac{t}{2}\mathcal{D}}(u^m v_{k-m}) = (e^{-\frac{t}{2}\mathcal{D}} u^m)(e^{-\frac{t}{2}\mathcal{D}} v_{k-m}). \end{equation}
Now, $v_{k-m} = \mathcal{T}(u^{k-m})$, and, by Lemma \ref{lemma degree and trace preserving}, $\mathcal{T}$ and $\mathcal{D}$ commute.  We may therefore rewrite (\ref{eq recursion proof 1.5}) as
\begin{equation} \label{eq recursion proof 2}
e^{-\frac{t}{2}\mathcal{D}}(u^m v_{k-m}) = (e^{-\frac{t}{2}\mathcal{D}} u^m)\mathcal{T}(e^{-\frac{t}{2}\mathcal{D}} u^{k-m})
\end{equation}
Eq.\ (\ref{eq B C}) gives
\[ e^{-\frac{t}{2}\mathcal{D}} (\,\cdot\,)^m = e^{\frac{m}{2}t}B_m^t \qquad \text{and} \qquad \mathcal{T}[e^{-\frac{t}{2}\mathcal{D}} (\,\cdot\,)^{k-m}] = e^{\frac{k-m}{2}t}C_{k-m}^t. \] 
Thus, (\ref{eq recursion proof 1}) and (\ref{eq recursion proof 2}) combine to give
\begin{equation} \label{eq recursion proof 3} \frac{d}{dt}B_k^t = e^{-\frac{k}{2}t}\sum_{m=1}^{k-1} me^{\frac{m}{2}t}B^t_me^{\frac{k-m}{2}t}C^t_{k-m} = \sum_{m=1}^{k-1} mC^t_{k-m}B^t_m. \end{equation}

Integrating both sides of (\ref{eq recursion proof 3}) from $0$ to $t$, and using the initial condition $B^t_k(u;\mx{v}) = u^k$, gives
\begin{equation} \label{eq recursion proof 4} B_k^t = u^k + \sum_{m=1}^{k-1} m\int_0^t C^{\tau}_{k-m}B^\tau_m\,d\tau. \end{equation}
The tracing map $\mathcal{T}$ is linear, and commutes with the integral (easily verified since all terms are polynomials); moreover, if $C\in\C[\mx{v}]$, then $\mathcal{T}(CB) = C\mathcal{T}(B)$.  Thus
\begin{equation} \label{eq recursion proof 5} C_k^t = \mathcal{T}(B_k^t) = \mathcal{T}(u^k) + \sum_{m=1}^{k-1} m\int_0^t \mathcal{T}[C^\tau_{k-m}B^\tau_m]\,d\tau = v_k + \sum_{m=1}^{k-1} m\int_0^t C^\tau_{k-m}C^\tau_m\,d\tau. \end{equation}
Finally, the evaluation map $\pi_s$ is an algebra homomorphism, and (as with $\mathcal{T}$) commutes with the integral; applying $\pi_s$ to (\ref{eq recursion proof 4}) and (\ref{eq recursion proof 5}) yields the desired equations (\ref{eq ck recursion}) and (\ref{eq bk recursion}), concluding the proof.
\end{proof}

\begin{remark} By changing the index $m\mapsto k-m$ in (\ref{eq ck recursion}) and averaging the results, we may alternatively state the recursion for $c_k$ as
\begin{equation} \label{eq ck recursion 2}c_k(s,t) = \nu_k(s) + \frac{k}{2}\sum_{m=1}^{k-1} \int_0^t c_{k-m}(s,\tau)c_m(s,\tau)\,d\tau. \end{equation}
A transformation of this form is not possible for the $b_k(s,t,u)$ recursion, however.
\end{remark}

\subsection{Exponential Growth Bounds \label{section Exponential Growth Bounds}}

In Section \ref{section PDE}, we will study the generating functions of the quantities $\nu_k(s)$, $c_k(s,t)$, and $b_k(s,t,u)$.  As such, we will need a priori exponential growth bounds.

\begin{lemma} \label{lemma estimate nu k} For $s,t\in\R$ and $k\ge 2$,
\begin{align} \label{eq nu k growth bound} |\nu_k(t)| &\le C_{k-1}(1+|t|)^{k-1} e^{-\frac{k}{2}t}, \qquad \text{and} \\
\label{eq ck growth bound} |c_k(s,t)| &\le C_{k-1}(1+|s-t|)^{k-1}e^{-\frac{k}{2}s},
\end{align}
where $C_k = \frac{1}{k+1}\binom{2k}{k}$ are the Catalan numbers.
\end{lemma}

\begin{remark} When $t>0$, $\nu_k(t)$ is the $k$th moment of the probability measure $\nu_t$ on the unit circle $\U$, and we therefore have the much better bound $|\nu_k(t)|\le 1$; similarly, if $s\ge t$, $|c_k(s,t)|\le e^{-\frac{k}{2}t}$.  It is necessary to have a priori bounds for negative $t$ and $s-t$ as well, however.  While (\ref{eq nu k growth bound}) is by no means sharp, the known exact formula (\ref{e.new1.4}) for $\nu_k(t)$ shows that, when $t<0$, $\nu_k(t)$ does grown exponentially with $k$ (at least for small $|t|$).
\end{remark}

In the proof of Lemma \ref{lemma estimate nu k}, we will use the well-known fact that the Catalan numbers satisfy Segner's recurrence relation
\[ C_k = \sum_{m=1}^k C_{m-1}C_{k-m}, \qquad k\ge 1. \]

\begin{proof} Taking $s=0$ in (\ref{eq ck recursion 2}), and noting that $\nu_k(0)=1$ for all $k$, we have
\begin{equation} \label{eq ck(0,t) recursion} c_k(0,t) = 1 + \frac{k}{2}\sum_{m=1}^{k-1}\int_0^t c_m(0,\tau)c_{k-m}(0,\tau)\,d\tau, \qquad k\ge 2. \end{equation}
We claim that
\begin{equation} \label{ck Catalan bound} |c_k(0,t)| \le C_{k-1}(1+|t|)^{k-1}, \qquad k\ge 1. \end{equation}
Since $c_1(0,t)=1=C_1$, we proceed by induction.   Let $k\ge 2$, and assume that (\ref{ck Catalan bound}) holds below level $k$; then (\ref{eq ck(0,t) recursion}) yields
\begin{align}
|c_{k}(0,t)|    & \leq1+\frac{k}{2}\int_{0}^{|t|}\sum_{m=1}^{k-1}%
C_{m-1}C_{k-m-1}\left(  1+\tau\right)  ^{k-2}d\tau\nonumber\\
& =1+\frac{k}{2\left(  k-1\right)  }\left(  \left(  1+|t|\right)  ^{k-1}%
-1\right)  \sum_{m=1}^{k-1}C_{m-1}C_{k-m-1}\nonumber\\
& =1-\frac{k}{2\left(  k-1\right)  }C_{k-1}+\left(  1+|t|\right)^{k-1}C_{k-1}\leq
C_{k-1}\left(1+|t|\right)  ^{k-1}\label{e.Yk}%
\end{align}
wherein we have used $\frac{k}{2\left(  k-1\right)  }C_{k-1}\ge 1$ for all $k\ge 2$.  This completes the induction argument, proving (\ref{ck Catalan bound}) holds.  Now, taking $s=0$ in (\ref{eq c alt 0}) yields
\begin{equation} \label{eq ck(0,t) 1} c_k(0,t) = e^{-\frac{k}{2}t}\nu_k(-t) \end{equation}
meaning that $\nu_k(t) = e^{-\frac{k}{2}t}c_k(0,-t)$, and this together with (\ref{ck Catalan bound}) proves (\ref{eq nu k growth bound}).  Then, using (\ref{eq c alt 0}) once more, (\ref{eq nu k growth bound}) implies that
\[ |c_k(s,t)| = e^{-\frac{k}{2}t}|\nu_k(s-t)| \le e^{-\frac{k}{2}t} e^{-\frac{k}{2}(s-t)}\cdot C_{k-1}(1+|s-t|)^{k-1} \]
which prove (\ref{eq ck growth bound}).
\end{proof}

\begin{remark} \label{remark varrho recursion} Equations (\ref{eq ck(0,t) recursion}) and  (\ref{eq ck(0,t) 1})  together yield a recursion for the coefficients $\varrho_k(t) = e^{\frac{k}{2}t}\nu_k(t) = c_k(0,-t)$:
\begin{equation} \label{eq varrho recursion} \varrho_k(t) = 1-\frac{k}{2}\sum_{m=1}^{k-1}\int_0^t \varrho_m(\tau)\varrho_{k-m}(\tau)\,d\tau. \end{equation}
This same recursion was derived in \cite[Lemma 11]{Biane1997b}, using free stochastic calculus, with $\nu_k(s)$ being identified as the limit moments of the free unitary Brownian motion distribution.  It is interesting that we can derive it directly from derivative formulas on the unitary group.
\end{remark}

\begin{lemma} \label{lemma b k bound} Let $s,t\in\R$ and $u\in\C$.  For $k\ge 2$, the $b_k(s,t,u)$ of (\ref{eq bk recursion}) satisfy
\begin{equation} \label{eq growth estimate 3} |b_k(s,t,u)| \le [5(1+|s|)(1+|t|)]^{k-1}|u|^k. \end{equation}
\end{lemma}

\begin{proof} Since $b_1(s,t,u)=u$, (\ref{eq growth estimate 3}) holds for $k=1$.  We proceed by induction, assuming (\ref{eq growth estimate 3}) holds below level $k$.  Then (\ref{eq bk recursion}) gives us
\begin{equation} \label{eq bk estimate 1} |b_k(s,t,u)| \le |u|^k + \sum_{m=1}^{k-1} \int_0^{|t|} m|c_{k-m}(s,\tau)||b_m(s,\tau,u)|\,d\tau. \end{equation}
The Catalan number $C_k$ is $\le 4^k$ (in fact it is $\sim 4^k/k^{3/2}\sqrt{\pi}$). Using the estimate $1+|s-t| \le (1+|s|)(1+|t|)$, (\ref{eq ck growth bound}) implies that $|c_k(s,t)| \le [4(1+|s|)(1+|t|)]^{k-1}$.  Thus (\ref{eq bk estimate 1}) and the inductive hypothesis give us, for $k\ge 2$,
\begin{align} \nonumber |b_k(s,t,u)|  &\le |u|^k + \sum_{m=1}^{k-1} \int_0^{|t|} m [4(1+|s|)(1+\tau)]^{k-m-1}\cdot [5(1+|s|)(1+\tau)]^{m-1}|u|^k\,d\tau \\
&=|u|^k + |u|^k\cdot(1+|s|)^{k-2}\int_0^{|t|} (1+\tau)^{k-2}\,d\tau\cdot \sum_{m=1}^{k-1} m4^{k-m-1}5^{m-1}. \label{eq bk estimate 2}
\end{align}
Summing the geometric series, we may estimate
\[ 5^{k-1}-4^{k-1} \le \sum_{m=1}^{k-1} m4^{k-m-1}5^{m-1} \le (k-1)5^{k-1}. \]
Substituting this into (\ref{eq bk estimate 2}) we have
\begin{align*} |b_k(s,t,u)| &\le |u|^k + |u|^k(1+|s|)^{k-2}[(1+|t|)^{k-1}-1]\frac{1}{k-1}\sum_{m=1}^{k-1}m4^{k-m-1}5^{m-1} \\
&\le |u|^k\left(1-(1+|s|)^{k-2}\frac{5^{k-1}-4^{k-1}}{k-1}\right) + 5^{k-1}(1+|s|)^{k-2}(1+|t|)^{k-1}|u|^k \\
&\le [5(1+|s|)(1+|t|)]^{k-1}|u|^k
\end{align*}
where we have used that $1+|s|\ge 1$ and $\frac{5^{k-1}-4^{k-1}}{k-1}\ge 1$ for $k\ge 2$.  This concludes the inductive  proof.
\end{proof}

\subsection{Holomorphic PDE \label{section PDE}}

The double recursion of Proposition \ref{prop Recursions} can be written in the form of coupled holomorphic PDEs for the generating functions of $c_k(s,t)$ and $b_k(s,t,u)$.

\begin{definition} \label{def psi phi} Let $s,t\in\R$.  For $z\in\C$, define
\[ \psi^s(t,z) = \sum_{k=1}^\infty c_k(s,t)z^k. \]
Additionally, for $u\in\C$ define
\[ \phi^{s,u}(t,z) = \sum_{k=1}^\infty b_k(s,t,u)z^k. \]
\end{definition}
\noindent By (\ref{eq ck growth bound}) and the Catalan bound $C_k\le 4^k$, the power series $z\mapsto \psi^s(t,z)$ is convergent whenever $|z|<e^{s/2}/4(1+|s-t|)$; similarly, by (\ref{eq growth estimate 3}), the power series $z\mapsto \phi^{s,u}(t,z)$ is convergent whenever $|z|<[5(1+|s|)(1+|t|)|u|]^{-1}$.  Hence, $\psi^s(t,\,\cdot\,)$ and $\phi^{s,u}(t,\,\cdot\,)$ are holomorphic on a nontrivial disk with radius that depends continuously on $s,t$.  Note that, by (\ref{eq b vs p}),
\begin{equation} \label{eq gen fn p vs b} \Pi(s,t,u,z) = \sum_{k\ge 1} p_k^{s,t}(u)z^k = \sum_{k\ge 1} e^{\frac{k}{2}t}b_k(s,t,u)z^k = \phi^{s,u}(t,e^{\frac{t}{2}} z). \end{equation}
So, identifying $\phi^{s,u}(t,z)$ will also identify the sought-after generating function $\Pi(s,t,u,z)$.

\begin{proposition} \label{prop diff t} For fixed $s\in\R$, the functions $\R\ni t\mapsto \psi^s(t,z)$ and $\R\ni t\mapsto \phi^{s,u}(t,z)$ are differentiable for all sufficiently small $|z|$ and $|u|$.  Their derivatives are given by
\[ \frac{\del}{\del t}\psi^s(t,z) = \sum_{k=1}^\infty \frac{\del}{\del t}c_k(s,t) z^k \qquad \text{and} \qquad \frac{\del}{\del t}\phi^{s,u}(t,z) = \sum_{k=1}^\infty \frac{\del }{\del t}b_k(s,t,u) z^k. \]
\end{proposition}

\begin{proof} From (\ref{eq ck recursion}), $\frac{\del}{\del t}c_1(s,t)=0$, while for $k\ge 2$ we have
\[ \frac{\del }{\del t}c_k(s,t) = k\sum_{m=1}^{k-1} c_{k-m}(s,t)c_m(s,t). \]
Hence, from (\ref{eq ck growth bound}) and the Catalan bound $C_k\le \frac{4^k}{k}$,
\[ \left|\frac{\del }{\del t}c_k(s,t)\right| \le \sum_{m=1}^{k-1} m|c_{k-m}(s,t)||c_m(s,t)| \le (k-1)4^ke^{-\frac{k}{2}s}(1+|s-t|)^k \]
for $k\ge 2$.  It follows that $\sum_{k=1}^\infty \frac{\del}{\del t}c_k(s,t)z^k$ converges to an analytic function of $z$ on the domain $|z|<e^{s/2}/4(1+|s-t|)$.  Integrating this series term-by-term over the interval $[0,t]$ shows that it is the derivative of $\psi^s(t,z)$, as claimed.  A completely analogous argument applies to $\phi^{s,u}(t,z)$.
\end{proof}

We will shortly write down coupled PDEs satisfies by $\psi^s$ and $\phi^{s,u}$.  First, we remark on their initial conditions.  From Proposition \ref{prop Recursions}, we have
\[ c_k(s,0) = \nu_k(s) \qquad \text{and} \qquad b_k(s,0,u) = u^k. \]
Thus
\begin{align} \psi^s(0,z) &= \sum_{k\ge 1} \nu_k(s)z^k, \label{eq psi initial 1}\\
\phi^{s,u}(0,z) &= \sum_{k\ge 1} u^kz^k = \frac{uz}{1-uz}. \label{eq phi initial}
\end{align}
It will be convenient to express $\psi^s(0,z)$ in terms of the shifted coefficients $\varrho_k(s) = e^{\frac{k}{2}s}\nu_k(s)$ considered in Remark \ref{remark varrho recursion}.  Define
\begin{equation} \label{eq def varrho(s,z)} \varrho(s,z) = \sum_{k\ge 1} \varrho_k(s)z^k = \psi^s(0,e^{\frac{s}{2}}z). \end{equation}
Note that, since $\nu_k(0)=1$ for all $k$, $\varrho(0,z) = \frac{z}{1-z}$.

\begin{proposition} \label{prop PDEs} For $s,t\in\R$ and $|z|$ and $|u|$ sufficiently small, the functions $\varrho$, $\psi^s$, and $\phi^{s,u}$ satisfy the following holomorphic PDEs:
\begin{align}  \label{eq PDE rho} \frac{\del\varrho}{\del s} &= - z\varrho\frac{\del\varrho}{\del z}, \qquad\qquad\quad \;\;\; \varrho(0,z) = \frac{z}{1-z}, \\
 \label{eq PDE 1} \frac{\partial\psi^s}{\partial t} &=z\psi^s\frac{\partial\psi^s}{\partial z}, \qquad\qquad \;\;\;   \psi^s(0,z)= \varrho(s,e^{-\frac{s}{2}}z), \\
\label{eq PDE 2} \frac{\partial\phi^{s,u}}{\partial t} &= z\psi^s\frac{\partial\phi^{s,u}}{\partial z},  \qquad\quad\;\;\; \phi^{s,u}(0,z) = \frac{uz}{1-uz}. \end{align}
\end{proposition}

\begin{remark} \begin{itemize} \item[(1)] PDE (\ref{eq PDE rho}) was proved in \cite[Lemma 1]{Biane1997c}, using the recursion (\ref{eq varrho recursion}).  We reprove it here, as a special case of (\ref{eq PDE 1}).
\item[(2)] It is unusual that nonlinear PDEs with given ``initial'' conditions should have well-defined solutions for time flowing forwards or backwards.  In fact, this is the case presently.  In terms of (\ref{eq PDE rho}), this is indicative of the fact that the measure $\nu_s$ exists for all $s\in\R$; although it becomes singular at $s=0$, it is well-behaved for $s>0$ and $s<0$; see \cite[Proposition 2.24]{Kemp2013} for a summary of known results about $\nu_s$.
\end{itemize} \end{remark}

\begin{proof} First, Remark \ref{remark varrho recursion} and (\ref{eq ck(0,t) 1}) show that $\varrho_k(t) = c_k(0,-t)$, and hence $\varrho(t,z) = \psi^0(-t,z)$.  Hence, (\ref{eq PDE rho}) follows immediately from (\ref{eq PDE 1}).  Now, Proposition \ref{prop diff t} yields that $\psi^s(t,z)$ is differentiable in $t$, and so by Proposition  \ref{prop Recursions}
\begin{equation} \label{eq del t psi} \frac{\partial}{\partial t}\psi^s(t,z) = \sum_{k=2}^\infty \frac{\partial}{\partial t} c_k(s,t)\,z^k = \sum_{k=2}^\infty \sum_{m=1}^{k-1}mc_m(s,t)c_{k-m}(s,t)\,z^k. \end{equation}
On the other hand, $\psi^s(t,z)$ is analytic in $z$, and
\[ z\frac{\partial}{\partial z}\psi^s(t,z) = \sum_{k=1}^\infty c_k(s,t)\cdot z\frac{\partial}{\partial z}z^k = \sum_{k=1}^\infty kc_k(s,t)z^k, \]
and so
\begin{align*} z\psi^s(t,z)\frac{\partial}{\partial z}\psi^s(t,z) &= \sum_{k_1=1}^\infty c_{k_1}(s,t)z^{k_1}\cdot\sum_{k_2=1}^\infty  k_2c_{k_2}(s,t)z^{k_2} \\
&= \sum_{k=2}^\infty z^k \sum_{k_1+k_2=k\atop k_1,k_2\ge 1} k_2 c_{k_1}(s,t)c_{k_2}(s,t).
\end{align*}
Reindexing the internal sum and comparing with (\ref{eq del t psi}) proves (\ref{eq PDE 1}).  The proof of (\ref{eq PDE 2}) is entirely analogous.

\end{proof}

\subsection{Generating Function\label{section Generating Function}}

We now proceed to prove the implicit formula (\ref{eq formula for gen fn}), by solving the coupled PDEs (\ref{eq PDE rho}) -- (\ref{eq PDE 2}).  We do this essentially by the method of characteristics.  These quasilinear PDEs have a fairly simple form; as a result, the characteristic curves are the same as the level curves in this case.  As we will see, all three equations have the same level curves.

\begin{lemma} \label{lemma char rho} Fix $s_0\ge 0$ and $w_0\in\C$ with $|w_0|<[4(1+s_0)]^{-1}$.  Consider the exponential curve
\[ \mx{w}(s) = w_0\,e^{\varrho(0,w_0)s}. \]
Then $s\mapsto \varrho(s,\mx{w}(s))$ is constant.  In particular, $\varrho(s,\mx{w}(s)) = \varrho(0,w_0)$ for all $s\in[0,s_0)$.
\end{lemma}


\begin{proof} Lemma \ref{lemma estimate nu k} shows that $e^{\frac{k}{2}s}|\nu_k(s)|\le [4(1+s)]^k$; thus
\[ \varrho(s,w) = \psi_{\nu_s}(e^{\frac{s}{2}}w) = \sum_{k\ge 1} e^{\frac{k}{2}s}\nu_k(s)w^k \]
converges to an analytic function of $w$ for $|w|<[4(1+s)]^{-1}$. Thus, since $s\mapsto[4(1+s)]^{-1}$ is decreasing, $\varrho(s,w)$ is differentiable in $s$ and analytic in $w$ for $|w|<[4(1+s_0)]^{-1}$ and $0\le s<s_0$.  Since $4(1+s_0)>1$, the initial condition $\varrho(0,w)=\frac{w}{1-w}$ is also analytic on this domain.  Thus, subject to these constraints, we can simply differentiate. To avoid confusion, we denote $\dot{\varrho}(s,w) = \frac{\del\varrho}{\del s}(s,w)$ and $\varrho'(s,w) = \frac{\del\varrho}{\del w}(s,w)$.  Thus
\begin{equation} \label{eq rho diff 1} \frac{d}{ds} \varrho(s,\mx{w}(s)) = \dot\varrho(s,\mx{w}(s)) + \varrho'(s,\mx{w}(s))\dot{\mx{w}}(s). \end{equation}
We now use (\ref{eq PDE rho}), which asserts that $\dot\varrho(s,w) = -w\varrho(s,w)\varrho'(s,w)$; hence
\[ \dot\varrho(s,\mx{w}(s)) = -\mx{w}(s)\varrho(s,\mx{w}(s))\,\varrho'(s,\mx{w}(s)). \]
Plugging this into (\ref{eq rho diff 1}) yields
\begin{equation} \label{eq rho diff 2}  \frac{d}{ds} \varrho(s,\mx{w}(s)) = \varrho'(s,\mx{w}(s))\left[-\mx{w}(s)\varrho(s,\mx{w}(s)) + \dot{\mx{w}}(s)\right]. \end{equation}
Note that $\mx{w}$ satisfies the ODE
\[
\dot{\mathbf{w}}(s)=\frac{d}{ds}w_{0}\,e^{\varrho(0,w_{0})s}=\varrho(0,w_{0}%
)w_{0}\,e^{\varrho(0,w_{0})s}=\varrho(0,w_{0})\mathbf{w}(s).
\]
Substituting this into (\ref{eq rho diff 1}) yields
\begin{align}
\frac{d}{ds}\varrho(s,\mathbf{w}(s))  &  =\varrho^{\prime}(s,\mathbf{w}%
(s))\mathbf{w}(s)  \left[  \varrho(0,w_{0})-\varrho(s,\mathbf{w}%
(s))\right], \label{eq rho diff 3} \\
\varrho\left(  s,\mathbf{w}(s)\right)  |_{s=0}  &  =\varrho
(0,w_{0}).\nonumber
\end{align}
We now easily see that $\varrho\left(s,\mx{w}(s)\right)  \equiv
\varrho(0,w_{0})=\frac{w_0}{1-w_0}$ is indeed the (unique) solution to this this ODE.  \end{proof}

\begin{corollary} \label{cor char upsilon} Subject to the constraints on $s,w$ in Lemma \ref{lemma char rho}, the function $\psi^s(0,w) = \varrho(s,e^{-\frac{s}{2}}w)$ is constant along the curves $s\mapsto e^{\frac{s}{2}}\mx{w}(s) = w_0 e^{[\varrho(0,w_0)+\frac12]s}$.  Note that
\[ \textstyle{\varrho(0,w_0)+\frac12 = \frac{w_0}{1-w_0}+\frac12 = \frac12\frac{1+w_0}{1-w_0}}. \]
Thus, for all sufficiently small $w$ and $s$,
\begin{equation} \label{eq implicit upsilon} \psi^s(0,w\,e^{\frac{s}{2}\frac{1+w}{1-w}}) = \upsilon(0,w) = \varrho(0,w) = \frac{w}{1-w}. \end{equation}
\end{corollary}
\noindent Differentiation shows that the function $w\mapsto we^{\frac{s}{2}\frac{1+w}{1-w}}$ is strictly increasing for all $w\in\R$ (provided $s<4$); and in general for all $w>0$ for all $s$; hence, (\ref{eq implicit upsilon}) actually uniquely determines $\psi^s(0,z)$ for $z$ (by analytic continuation) when $s<4$; moreover, by the inverse function theorem, it is analytic in $z$.  

Following the idea of Lemma \ref{lemma char rho}, we now show that the level-curves of the functions $\psi^s$ and $\phi^{s,u}$ are also exponentials.

\begin{lemma} \label{lemma phi char} For $z_0\in\C$, consider the exponential curve
\[ \mx{z}(t) = z_0\,e^{-\psi^s(0,z_0)t}. \]
Then for $z_0$ and $t$ sufficiently small, $t\mapsto \psi^s(t,\mx{z}(t))$ and $t\mapsto \phi^{s,u}(t,\mx{z}(t))$ are constant.  In particular,
\[ \psi^s(t,\mx{z}(t)) = \psi^s(0,z_0), \quad\text{and}\quad \phi^{s,u}(t,\mx{z}(t)) = \phi^{s,u}(0,z_0). \]
\end{lemma}

\begin{proof} To improve readability, through this proof we suppress the parameters $s,u$ and simply write $\phi^{s,u}(t,z) = \phi(t,z)$ and $\psi^s(t,z) = \psi(t,z)$.  As per the discussion following Definition \ref{def psi phi}, these functions are differentiable in $t$ and analytic in $z$ for sufficiently small $z$.  As in the proof of Lemma \ref{lemma char rho}, we set $\dot{\psi}(t,z) =\frac{\del}{\del t}\psi(t,z)$, and $\psi'(t,z) = \frac{\del}{\del z}\psi(t,z)$, and similarly with $\dot\phi$ and $\phi'$.  Differentiating, we have
\begin{align*} \frac{d}{dt}\psi(t,\mx{z}(t)) &= \dot\psi(t,\mx{z}(t)) + \psi'(t,\mx{z}(t))\dot{\mx{z}}(t), \\
\frac{d}{dt}\phi(t,\mx{z}(t)) &= \dot\phi(t,\mx{z}(t)) + \phi'(t,\mx{z}(t))\dot{\mx{z}}(t). \end{align*}
PDEs (\ref{eq PDE 1}) and (\ref{eq PDE 2}) say $\dot\psi(t,z) = z\psi(t,z)\psi'(t,z)$ and $\dot\phi(t,z) = z\psi(t,z)\psi'(t,z)$, and so
\begin{align} \label{eq psi' 1} \frac{d}{dt}\psi(t,\mx{z}(t)) &= \left[\mx{z}(t)\psi(t,\mx{z}(t))+\dot{\mx{z}}(t)\right]\psi'(t,\mx{z}(t)), \\
\label{eq phi' 1} \frac{d}{dt}\phi(t,\mx{z}(t)) &= \left[\mx{z}(t)\psi(t,\mx{z}(t))+\dot{\mx{z}}(t)\right]\phi'(t,\mx{z}(t)).
\end{align}
As in the proof of Lemma \ref{lemma char rho}, we note that $\mx{z}$ satisfies the ODE
\[ \dot{\mx{z}}(t) -z_0\psi(0,z_0)e^{-\psi(0,z_0)t} = -\psi(0,z_0)\mx{z}(t). \]
Substituting this into (\ref{eq psi' 1}) and (\ref{eq phi' 1}) yields
\begin{align} \label{eq psi'} \frac{d}{dt}\psi(t,\mx{z}(t)) &= \left[\psi(t,\mx{z}(t))-\psi(0,z_0)\right]\mx{z}(t)\psi'(t,\mx{z}(t)), \\
\label{eq phi'} \frac{d}{dt}\phi(t,\mx{z}(t)) &= \left[\psi(t,\mx{z}(t))-\psi(0,z_0)\right]\mx{z}(t)\phi'(t,\mx{z}(t)).
\end{align}
The initial condition for (\ref{eq psi'}) is $\left.\psi(t,\mx{z}(t))\right|_{t=0} = \psi(0,z_0)$, and it follows immediately that $\psi(t,\mx{z}(t)) = \psi(0,z_0)$ is the unique solution of this ODE.  Hence, (\ref{eq phi'}) reduces to the equation $\frac{d}{dt}\phi(t,\mx{z}(t))=0$, and since its initial condition is $\left.\phi(t,\mx{z}(t))\right|_{t=0}=\phi(0,z_0)$, it follows that $\phi(t,\mx{z}(t)) = \phi(0,z_0)$ as well. \end{proof}

This brings us to the proof of (\ref{eq formula for gen fn}).  First, Lemma \ref{lemma phi char}, together with the initial condition in (\ref{eq PDE 2}), yields
\begin{equation} \label{final 1} \phi^{s,u}(t,ze^{-\psi^s(0,z)t}) = \phi^{s,u}(0,z) = \frac{uz}{1-uz} = \frac{1}{1-uz}-1. \end{equation}
Next, Corollary \ref{cor char upsilon} describes $(s,z)\mapsto \psi^s(0,z)$ in terms of its level curves; (\ref{eq implicit upsilon}) states that
\begin{equation} \label{final 2}  \psi^s(0,w\,e^{\frac{s}{2}\frac{1+w}{1-w}}) = \varrho(0,w) = \frac{w}{1-w}. \end{equation}
So set $z=we^{\frac{s}{2}\frac{1+w}{1-w}}$; then (\ref{final 1}) and (\ref{final 2}) say
\begin{equation} \label{final 3} \phi^{s,u}(t,e^{-\frac{w}{1-w}t}we^{\frac{s}{2}\frac{1+w}{1-w}}) = \phi^{s,u}(t,e^{-\psi^s(0,z)t}z) = \left(1-uwe^{\frac{s}{2}\frac{1+w}{1-w}}\right)^{-1}-1. \end{equation}
Finally, note that
\[ -\textstyle{\frac{w}{1-w} = -\frac12\frac{1+w}{1-w}+\frac12} \]
and so (\ref{final 3}) may be written in the form
\begin{equation} \label{final 4} \phi^{s,u}(t,e^{\frac{t}{2}}we^{\frac12(s-t)\frac{1+w}{1-w}}) = \left(1-uwe^{\frac{s}{2}\frac{1+w}{1-w}}\right)^{-1}-1. \end{equation}
Finally, recall (\ref{eq gen fn p vs b}), which says that 
\begin{equation} \label{final 4.5} \Pi(s,t,u,\zeta) = \phi^{s,u}(t,e^{\frac{t}{2}}\zeta). \end{equation}
Setting $\zeta = we^{\frac12(s-t)\frac{1+w}{1-w}}$, (\ref{final 4}) and (\ref{final 4.5}) combine to yield
\[  \left(1-uwe^{\frac{s}{2}\frac{1+w}{1-w}}\right)^{-1}-1 =  \phi^{s,u}(t,e^{\frac{t}{2}}we^{\frac12(s-t)\frac{1+w}{1-w}}) =  \phi^{s,u}(t,e^{\frac{t}{2}}\zeta) = \Pi(s,t,u,\zeta) \]
which is precisely the statement of (\ref{eq formula for gen fn}).

\subsection{Proof of Theorem \ref{thm Biane transform} ($\G_{t,t} = \mathscr{G}^t$) \label{section Proof Biane}}

We are now in a position to complete the proof of Theorem \ref{thm Biane transform}, modulo a small error in \cite{Biane1997b}.

\begin{remark} \label{remark Biane's error} In \cite[Lemma 18]{Biane1997b}, there is a typographical error that is propagated through the remainder of that paper.  In the second line of the proof of that lemma, the function $\iota(t,\cdot)$ should be the inverse of $z\mapsto ze^{\frac{t}{2}\frac{1+z}{1-z}}$ rather than the inverse of $z\mapsto \frac{z}{1+z}e^{\frac{t}{2}(1+2z)}$ as stated.  That $\iota(t,\cdot)$ has this different form follows from \cite[Lemma 11]{Biane1997b}, which defines the kernel function $\kappa(t,z)$ (formula 4.2.2.a) implicitly by $\frac{\kappa(t,z)-1}{\kappa(t,z)+1}e^{\frac{t}{2}\kappa(t,z)}=z$; then $\iota(t,z)=\frac{\kappa(t,1/z)+1}{\kappa(t,1/z)-1}$ yields the result.  Hence, the correct generating function for the Biane polynomials in \cite{Biane1997b} is the one in (\ref{eq formula for gen fn}) above, in the special case $s=t$.  The third author of the present paper discovered this error as the result of the present work: early versions of the calculations in this section suggested the generating function should have the form in (\ref{eq formula for gen fn}).  When Philippe Biane was consulted about this discrepancy, he confirmed the error, and tracked its source in \cite{Biane1997b}, in a private communication with the third author on October 27, 2011.
\end{remark}

\begin{proof}[Proof of Theorem \ref{thm Biane transform}] \label{proof of thm Biane transform} By the density of trigonometric polynomials in $L^2(\U,\nu_t)$ for any measure $\nu_t$, the transform $\mathscr{G}^t$ is determined by its action on Laurent polynomial functions.  Hence, to verify that $\G_{t,t}=\mathscr{G}^t$, it suffices to verify that $(\mathscr{G}^t)^{-1}$ agrees with $\H_{t,t}$ on monomials $z\mapsto z^k$ for $k\in\Z$.  Eq. (\ref{eq pk p-k}) is consistent with \cite[Lemma 18]{Biane1997b}, and so it suffices to prove this result for $k\ge 1$.  Eq.\ (\ref{eq formula for gen fn}) verifies that the Biane polynomials $p^{t,t}_k$ for $\H_{t,t}$ have the same generating function as the Biane polynomials of $\mathscr{G}^t$ (cf.\ Remark \ref{remark Biane's error}), and this concludes the proof.
\end{proof}

\appendix

\section{Heat Kernel Measures on Lie Groups \label{section heat kernels}}

Suppose that $G$ is a connected Lie group and $\beta$ is a basis for
$\operatorname*{Lie}\left(  G\right)$. Then $A=\sum_{X\in\beta}\partial
_{X}^{2}$ is a left-invariant non-positive elliptic differential operator
which is essentially self adjoint on $C_{c}^{\infty}\left(  G\right)  $ as an
operator on $L^{2}\left(G,dg\right)$ where $dg$ is a right Haar measure on
$G$. Associated to the contraction semigroup $\left\{e^{tA/2}\right\}
_{t>0}$ is a convolution semigroup of probability (heat kernel) densities 
$\left\{h_{t}\right\} _{t>0}$. In more detail, $\mathbb{R}_{+}\times
G\ni(t,g)  \rightarrow h_{t}(g)  \in\mathbb{R}_{+}$
is a smooth function such that
\[
\partial_{t} h_{t}(g)  =\frac{1}{2}A h_{t}(g)
\text{ for }t>0
\]
and
\[
\lim_{t\downarrow0}\int_{G}f(g)  h_{t}(g)\,dg=f(e)  \text{ for all }f\in C_{c}(G).
\]
(Throughout, $e=1_G$.) Basic properties of these heat kernels are summarized in \cite[Proposition
3.1]{Driver1997} and \cite[Section 3]{Driver2009a}. For an exhaustive
treatment of heat kernels on Lie groups see \cite{Robinson1991} and \cite{Varopoulos1992}. For our present purposes, we need to know that, if $G=\U_N$ or $G=\GL_N$ (and so $h_t$ is the density of $\rho_t^N$ or $\mu_{s,t}^N$, respectively), then
\begin{equation}
\int_{G}f(g) h_{t}(g)\,dg  =\sum_{n=0}^{\infty}\frac{1}{n!}\left(\frac{t}{2}\right)^{n}\left(A^{n}f\right)\left(I\right)  \text{ for all }t\geq0 \label{e.fgng}%
\end{equation}
whenever $f$ is a trace polynomial. This result can be seen as a consequence
of Langland's theorem; see, for example, \cite[Theorem 2.1 (p.\ 152)]
{Robinson1991}. As it is a bit heavy to get to Langland's theorem in Robinson
we will, for the reader's convenience, sketch a proof of (\ref{e.fgng});
see Theorem \ref{t.heat} below. For the rest of this section let $d$ denote
the left-invariant metric on $G$ such that $\left\{  \partial_{X}\right\}
_{X\in\beta}$ is an orthonormal frame on $G$ and set $\left\vert g\right\vert
=d(e,g)$. Also let us use the abbreviation $ h_{t}(f)  $ for $\int_{G}f(g)h_{t}(g)\,dg$.

\begin{lemma}
\label{l.difflem}Suppose $f\colon [0,T]\times G\rightarrow\mathbb{C}$ is a $C^{2}$ function such that $|k(t,g)| \leq Ce^{C|g|}$ for some $C<\infty$, where $k$ is any of the functions $f$, $\partial_{t}f$, or $\partial_{X}f$ for any $X\in\operatorname*{Lie}(A)$, or $Af$. Then%
\begin{equation}
\partial_{t} h_{t}\left(f(t,\cdot)\right)  = h_{t}\left(\partial_{t}f(t,\cdot)  +\frac{1}{2}Af(t,\cdot)\right)  \text{ for }t\in(0,T] \label{e.diff1}%
\end{equation}
and
\begin{equation}
\lim_{t\downarrow0} h_{t}\left(f(t,\cdot)\right)  =f(0,\cdot). \label{e.lim1}%
\end{equation}

\end{lemma}

\begin{proof}
Let $\left\{  h_{n}\right\}  \subset C_{c}^{\infty}\left(G,[0,1]\right)$ be smooth cutoff functions as in
\cite[Lemma 3.6]{Driver1997} and set $f_{n}(t,g)\equiv h_{n}(g)  f(t,g)$.
Then it is easy to verify that it is now permissible to differentiate past the
integrals and perform the required integration by parts in order to show that
\[
\frac{d}{dt}\left[h_{t}(f_{n}(t,\cdot))\right]  = h_{t}\left(\partial_{t}f(t,\cdot)  +\frac{1}{2}Af(t,\cdot)\right).
\]
Let $F(t,\cdot)  =\partial_{t}f(t,\cdot)  +\frac{1}{2}Af(t,\cdot)$ and
\begin{align*}
F_{n}(t,\cdot)   &=\partial_{t}f_{n}(t,\cdot)
+\frac{1}{2}Af_{n}(t,\cdot) \\
&  =F(t,\cdot)  h_{n}+\frac{1}{2}f(t,\cdot)
Ah_{n}+\sum_{X\in\beta}\partial_{X}f(t,\cdot)  \partial_{X}h_{n}.
\end{align*}
From the properties of $h_{n}$ and the assumed bounds on $f$, given
$\epsilon\in(0,T)$ there exist $C<\infty$ independent of $n$
such that
\[
\sup_{\epsilon\leq t\leq T}\left\vert F_{n}(t,g)-F(t,g)\right\vert \leq \1_{\left\vert g\right\vert \geq n}Ce^{C\left\vert g\right\vert}.
\]

It then follows by the standard heat kernel bounds (see for example \cite{Varopoulos1992} or \cite[page 286]{Robinson1991}) that
\[
\sup_{\epsilon\leq t\leq T}\left\vert h_{t}\left(F_{n}(t,\cdot)\right)  - h_{t}\left(F(t,\cdot)\right) \right\vert \rightarrow 0\text{ as }n\rightarrow\infty.
\]
Hence we may conclude that $\frac{d}{dt}\left[h_{t}(f(t,\cdot))\right]$ exists and%
\begin{align*}
\frac{d}{dt}\left[h_{t}\left(f(t,\cdot)\right)\right]
&  =\lim_{n\rightarrow\infty}\frac{d}{dt}\left[h_{t}\left(f_{n}(t,\cdot)\right)\right] \\
&  = h_{t}\left(  \partial_{t}f_{n}(t,\cdot)
+\frac{1}{2}Af_{n}(t,\cdot)\right)  \text{ for }\epsilon<t\leq T
\end{align*}
which proves (\ref{e.diff1}).  To prove (\ref{e.lim1}) we start with the estimate
\begin{align*}
\left\vert  h_{t}\left(  f(t,\cdot)\right) - f(0,e) \right\vert  &  =\left\vert \int_{G}\left[f(t,y)-f(0,e)\right]  h_{t}(y)\,dy\right\vert \\
&  \leq\int_{G}\left\vert f(t,y) - f(0,e)\right\vert  h_{t}(y)\,dy\\
&  \leq\delta(\epsilon,t)  +C\int_{\left\vert y\right\vert
>\epsilon}e^{C\left\vert y\right\vert } h_{t}(y)\,dy
\end{align*}
where
\[
\delta(\epsilon,t)  =\int_{\left\vert y\right\vert
\leq\varepsilon}\left\vert f(t,y) - f(0,e)\right\vert  h_{t}(y)\,dy\leq\sup_{\left\vert y\right\vert\leq\epsilon}\left\vert f(t,y) - f(0,e)\right\vert.
\]
From \cite[Lemma 4.3]{Driver1997} modified in a trivial way from its original
form where $\epsilon$ was take to be $1$, we know that%
\[
\limsup_{t\downarrow0}\int_{\left\vert y\right\vert >\epsilon
}e^{c\left\vert y\right\vert } h_{t}(y)\,dy=0\text{ for all }\epsilon>0\text{ and }c<\infty.
\]
Therefore, we conclude that
\[
\limsup_{t\downarrow0}\left\vert  h_{t}\left(  f(t,\cdot)\right) - f(0,e)\right\vert \leq\limsup_{t\downarrow0}%
\delta(\epsilon,t)  \rightarrow0\text{ as }\epsilon\downarrow 0
\]
as claimed. \end{proof}

\begin{theorem}
\label{t.heat}Suppose now that $G=\U_N$ or $G=\GL_N$ and $P_{N}$ is a trace polynomial function on $G$. Then for $T>0$,%
\begin{equation}
 h_{T}(P_{N})  =\left(\sum_{n=0}^{\infty}\frac{1}{n!}\left(\frac{T}{2}\right)^{n}A^{n}P_{N}\right)(I_N). \label{e.heat}%
\end{equation}

\end{theorem}

\begin{proof}
Fix $T>0$, and for $0<t<T$ let
\[
f(t,\cdot)  =\sum_{n=0}^{\infty}\frac{1}{n!}\left(  \frac
{T-t}{2}\right)^{n}A^{n}P_{N}%
\]
where the sum is convergent as $A$ is a bounded operator on the finite
dimensional subspace of trace polynomials of trace degree $\deg P$ or less. Moreover,
$f(t,\cdot)$ is again a trace polynomial with time dependent
coefficients and $f$ satisfies
\[
\partial_{t}f(t,\cdot)  +\frac{1}{2}Af(t,\cdot)
=0\text{ with }f(T,\cdot)  =P_{N}.
\]
From Lemma \ref{l.difflem} we may now conclude,%
\[
\frac{d}{dt}\left[h_{t}\left(f(t,\cdot)\right)\right]
= h_{t}\left(\partial_{t}f(t,\cdot)  +\frac{1}{2}Af(t,\cdot)\right)  =0.
\]
Therefore $t\rightarrow h_{t}\left(f(t,\cdot)\right)$ is
constant for $t>0$ and hence, using Lemma \ref{l.difflem} again,
\[ \displaystyle{h_{T}(P_{N})   = h_{T}\left(f(T,\cdot)\right)  =\lim_{t\downarrow0} h_{t}\left(f(t,\cdot)\right)
  =f(0,I_N)  =\left(\sum_{n=0}^{\infty}\frac{1}{n!}\left(\frac{T}{2}\right)^{n}A^{n}P_{N}\right)(I_N)}.\]
This concludes the proof.  \end{proof}

\subsection*{Acknowledgments} The authors wish to thank Matt Dyer and Brendon Rhoades, who provided very useful algebraic insights to the authors for the proofs of Propositions \ref{LaurentDense.prop } and \ref{prop P-->P_N unique}.  We also express our gratitude to the referee, whose suggestions greatly helped us to improve the exposition, particularly in the introduction.

\bibliographystyle{acm}
\bibliography{DHK-2013}

\end{document}